\documentclass[
10pt,																%SchriftgrÃÂ¶ÃÂe anpassen
a4paper,                           						%Anpassung auf A4-Papier
english,															%Anpassung auf englische Standards
parskip=half,													%Halbe Zeile zwischen AbsÃÂ¤tzen und kein Einzug in der ersten Zeile
]{scrartcl}														%Es soll ein KOMA-Skript vom Typ Artikel angelegt werden

\usepackage{babel}											%Schrift- und Rechtschreibungspaket
\usepackage[latin1]{inputenc}								%Umlaute und Sonderzeichen kÃÂ¶nnen im Quelltext verwendet werden
\usepackage[T1]{fontenc}									%Umlaute kÃÂ¶nnen in die pdf-Datei konvertiert werden 
\usepackage{lmodern}											%SchriftglÃÂ¤ttung
\usepackage{csquotes}										%Sprachenspezifische Zitierstile einbinden (fÃÂ¼r BibLaTeX empfohlen)
\usepackage[
style=alphabetic,												%Marken der EintrÃÂ¤ge haben die Form [XYZ11] statt [1]
sorting=nyt,													%Die EintrÃÂ¤ge werden nach Name(n), Jahr und Titel sortiert
maxnames=5,														%Ab 5+x Aurtoren wird die Marke mittels '+' abgekÃÂ¼rzt
minnames=3,														%AbgekÃÂ¼rzte Marken beinhalten die Namen von 3 Autoren und das '+'
backend=bibtex8,												%Verwendet BibTeX8 als Sort-Database (Regionale Spezifizierungen mÃÂ¶glich)
]{biblatex}														%Umfangreiches Literaturverzeichnis kann erstellt werden
\bibliography{PaperSTE}										%Name der fÃÂ¼r das Literaturverzeichnis benÃÂ¶tigten .bib-Datei

\usepackage{xcolor}											%ErmÃÂ¶glicht das Verwenden von Farben
\usepackage{graphicx} 										%ErmÃÂ¶glicht das Einbinden von Grafiken
\usepackage{picinpar}										%ErmÃÂ¶glicht textumflossene Grafiken
\usepackage[margin = 3cm]{geometry}						%Setzt den jeden Rand auf 3cm
\DeclareGraphicsExtensions{.jpg, .gif, .mps}			%Dateiendung einer Grafik kann weglassen werden
\usepackage{tikz}												%ErmÃÂ¶glicht das Zeichnen eigener Grafiken
\usetikzlibrary{arrows}										%LÃÂ¤dt die Pfeile-Bibliothek fÃÂ¼r die eigenen Grafiken
\usetikzlibrary{decorations.pathmorphing}				%ErmÃÂ¶glicht gewellte Pfeile und anderes

\usepackage{amsmath}											%ErmÃÂ¶glicht das Verwenden von mathematischen Zeichen
\usepackage{amssymb}											%ErmÃÂ¶glicht das Verwenden von speziellen mathematischen Symbolen
\usepackage{amsfonts}										%ErmÃÂ¶glicht das Verwenden von mathematischen Schriften
\usepackage[amsmath,thmmarks,hyperref]{ntheorem}	%ErmÃÂ¶glicht Theorem-Umgebungen von weiter unten
\usepackage{stmaryrd}										%ErmÃÂ¶glicht einen Widerspruchspfeil mittels \lightning
\usepackage{dsfont}											%ErmÃÂ¶glicht eine Indikatorfunktion mittels \mathds{1}
\usepackage{makeidx}    	          					%ErmÃÂ¶glicht ein Index-Verzeichnis
\usepackage[refpage]{nomencl}								%ErmÃÂ¶glicht ein Nomenklatur-Verzeichnis
\usepackage{hyperref}										%ErmÃÂ¶glicht verlinkte Verweise mit roten Boxen
\usepackage{mathtools}										%ErmÃÂ¶glicht Umgebung ''dcases''
\numberwithin{equation}{section}							%Nummerierung der Gleichungen sectionweise und nicht fortlaufend
\allowdisplaybreaks[3]										%Erlaubt Seitenumbruch in Formeln (Feintuning durch Zahl 1 bis 4)

\newcommand*{\N}{\mathbb{N}}

\newcommand*{\R}{\mathbb{R}}
\newcommand*{\C}{\mathbb{C}}
\renewcommand*{\epsilon}{\varepsilon}
\renewcommand*{\phi}{\varphi}
\renewcommand*{\kappa}{\varkappa}
\let\temptheta\theta											%In diesen drei Zeilen
\let\theta\vartheta											%werden die Befehle \theta
\let\vartheta\temptheta										%und \vartheta getauscht
\let\temprho\rho												%In diesen drei Zeilen
\let\rho\varrho												%werden die Befehle \rho
\let\varrho\temprho											%und \varrho getauscht
\renewcommand*{\tilde}{\widetilde}
\renewcommand*{\hat}{\widehat}
\renewcommand*{\bar}{\overline}
\renewcommand*{\emptyset}{\varnothing}
\renewcommand*{\angle}{\sphericalangle}
\renewcommand*{\d}{\partial}
\renewcommand*{\div}{\mathrm{div}}
\newcommand*{\entspr}{\mathrel{\widehat{=}}}
\newcommand*{\dH}{\, d\mathcal{H}}
\newcommand*{\skp}[2]{\left\langle #1,#2 \right\rangle}
\newcommand*{\norm}[1]{\left\|#1\right\|}
\newcommand*{\mint}{-\!\!\!\!\!\!\int}

\DeclareMathOperator{\sgn}{sgn}

\DeclareMathOperator{\im}{Im}
\DeclareMathOperator{\Vol}{Vol}
\DeclareMathOperator{\id}{Id}
\DeclareMathOperator{\ord}{ord}
\DeclareMathOperator{\WE}{WE}

\theoremstyle{plain}											%Nach dem Theorem-Kopf wird kein Zeilenumbruch eingefÃÂ¼gt
\theoremheaderfont{\normalfont\bfseries}				%Theorem-Kopf wird normal groÃÂ und fett gesetzt
\theorembodyfont{\normalfont}								%Schrift des Theorem-KÃÂ¶rpers ist normal
\theoremnumbering{arabic}									%Theoreme werden mit arabischen Ziffern versehen
\theoremseparator{:}											%Trennzeichen nach Theoremname ist ein Doppelpunkt
\theoremsymbol{\ensuremath{\square}}					%Theoremzeichen am Ende ist ein schwarzes Quadrat
\newtheorem{defi}{Definition}[section]
\newtheorem{rem}[defi]{Remark}

\theoremstyle{plain}											%Nach dem Theorem-Kopf wird kein Zeilenumbruch eingefÃÂ¼gt
\theoremheaderfont{\normalfont\bfseries}				%Theorem-Kopf wird normal groÃÂ und fett gesetzt
\theorembodyfont{\normalfont\itshape}					%Schrift des Theorem-KÃÂ¶rpers ist normal und kursiv
\theoremnumbering{arabic}									%Theoreme werden mit arabischen Ziffern versehen
\theoremseparator{:}											%Trennzeichen nach Theoremname ist ein Doppelpunkt
\theoremsymbol{}												%Kein Theoremzeichen am Ende
\newtheorem{thm}[defi]{Theorem}
\newtheorem{lemma}[defi]{Lemma}
\newtheorem{cor}[defi]{Corollary}

\theoremstyle{nonumberplain}								%Nach dem Theorem-Kopf wird kein Zeilenumbruch und keine Nummer eingefÃÂ¼gt
\theoremheaderfont{\normalfont}							%Theorem-Kopf wird normal gesetzt
\theorembodyfont{\normalfont}								%Schrift des Theorem-KÃÂ¶rpers ist normal
\theoremseparator{:}											%Trennzeichen nach Theoremname ist ein Doppelpunkt
\theoremsymbol{\ensuremath{\blacksquare}}				%Zeichen am Ende des Beweises ist ein schwarzes Quadrat
\newtheorem{proof}{\underline{Proof}}

\begin{document}

\begin{titlepage}
\title{Local Well-Posedness for Volume-Preserving Mean Curvature and Willmore Flows with Line Tension}
\author{Helmut Abels\footnote{Fakult\"at f\"ur Mathematik,  
Universit\"at Regensburg,
93040 Regensburg,
Germany, e-mail: {\sf helmut.abels@mathematik.uni-regensburg.de}}\ \ 
Harald Garcke\footnote{Fakult\"at f\"ur Mathematik,  
Universit\"at Regensburg,
93040 Regensburg,
Germany, e-mail: {\sf harald.garcke@mathematik.uni-regensburg.de}}
and Lars M\"uller\footnote{Fakult\"at f\"ur Mathematik,  
Universit\"at Regensburg,
93040 Regensburg,
Germany}}
%\date{}
\end{titlepage}
\maketitle

\begin{abstract}
We show the short-time existence and uniqueness of solutions for the motion of an evolving hypersurface in contact with a solid container driven by the volume-preserving mean curvature flow (MCF) taking line tension effects on the boundary into account. Difficulties arise due to dynamic boundary conditions and due to the contact angle and the non-local nature of the resulting second order, nonlinear PDE. In addition, we prove the same result for the Willmore flow with line tension, which results in a nonlinear PDE of fourth order. For both flows we will use a Hanzawa transformation to write the flows as graphs over a fixed reference hypersurface.
\end{abstract}


\small \textbf{AMS subject classifcations:} 53C44, 35K35, 35K55

\normalsize

\section{Introduction}\label{sec:Introduction}

This paper is devoted to two-dimensional evolving hypersurfaces $\Gamma$ 
in $\R^3$ that are brought in contact with a solid boundary
and which move under a geometrical flow taking line energy effects at the
boundary into account.
% of $\Omega$ and the rules that govern their movements.
% Modeling a drop of liquid or a soap bubble physics suggest that the air-liquid-interface or the soap layer tends to minimize its area. 
In many physical systems interfaces evolve in order to decrease surface
area or energies involving the curvature of the surface.
If mass is conserved the enclosed volume has to be preserved as for
example in attachment limited kinetics, see \cite{TC95}. If such a
surface gets into contact with some fixed impermeable boundary  $\d
\Omega$
%\footnote{Harald: (Please) Change beginning. (Thank you!) Irgendwo sollte \cite{Hui87} mit rein sonst taucht er bei der Literatur nicht auf.}
% the mass conservation law makes it necessary to demand a constant volume condition. 
the occurring contact angle is mainly determined by the material constants
which involve the surface energies of the interfaces involved.
% and thereby the wettability of the container. 
But in particular on small length scales a second effect is entering the scenery, namely the line tension (cf. Section 1 of \cite{BLK06}). This effect penalizes long contact curves and forces the drop or bubble to roll off the boundary. Mathematically line tension effects have been studied in the stationary case by Morgan, Taylor and Cook (cf. \cite{Mor94a}, \cite{Mor94b}, \cite{MT91}, \cite{Coo85}).

The geometric evolution law that we want to consider is
\begin{align*}
	V_\Gamma = H_\Gamma - \bar{H},
\end{align*}
which is known as the volume-preserving mean curvature flow (MCF),
see \cite{Hui87}, and is a simplified model for
interface motion under a volume constraint.
% such a motion of soap bubbles or liquid drops.
 During this motion it is often  unnatural to prescribe the boundary curve or the contact angle since an arbitrary drop or bubble, which is brought in contact with a solid container, will not instantly have a boundary curve or contact angle that is energetically minimal. Instead of doing so, we will impose dynamic boundary conditions to allow the contact angle to change and the boundary curve to move. We will prove in Theorem \ref{thm:ShortTimeExistenceMCF} that for a sufficiently smooth initial droplet there is a small time interval in which we can guarantee that the initial droplet can evolve following the rules of this motion.

Biomembranes, such as the surface of a red blood cell, however, are related to the so called Helfrich energy (cf. \cite{Can70} and \cite{CV12}). The Willmore energy can be viewed as the easiest example of the Helfrich energy. Instead of minimizing their areas these hypersurfaces try to minimize their bending energies, which results in the Willmore flow
\begin{align*}
	V_\Gamma = -\Delta_\Gamma H_\Gamma - \frac{1}{2} H_\Gamma \left(H_\Gamma^2 - 4 K_\Gamma\right).
\end{align*}
Here the motion of a surface proportional to the Laplace-Beltrami operator of the mean curvature plus some lower order curvature related terms. Including the wetting and line tension effect on the boundary we will again show well-posedness for sufficiently short times in Theorem \ref{thm:ShortTimeExistenceWillmore}.

In Section \ref{sec:LinearMCF} we describe the general setting and introduce the notation that will be used. After presenting the preliminaries we will move on to investigate the volume-preserving MCF of an evolving hypersurface with line tension effects on the contact curve. This motion will be governed by nonlinear PDEs of second order, which we will linearize around a fixed reference hypersurface. Section \ref{sec:LocalExistenceMCF} is devoted to the proof of the existence of solutions of the MCF for sufficiently short times. We will achieve this goal by first considering the short-time existence of solutions of the linearized flow and then apply a fixed point argument to prove the same statement for the original nonlinear flow. The non-local nature of the volume-preserving MCF will give rise to some technical difficulties. After the consideration concerning the MCF, we will study the Willmore Flow in the same context in Section \ref{sec:LinearWillmore}. The setting of the evolving hypersurface in contact with a container remains the same as before, but now the rules governing the motion of the hypersurface will be given as the Willmore flow. Again we will include line tension effects and boundary conditions of relaxation type. The resulting nonlinear PDE will be of fourth order, but purely local. The end of the section will be devoted to the linearization of the PDE. Following the same strategy as for the MCF we will prove the short-time existence of solutions for the motion driven by the Willmore flow in Section \ref{sec:LocalExistenceWillmore}.

\section{The volume-preserving MCF and its linearization}\label{sec:LinearMCF}

\subsection{The mean curvature flow}\label{ssec:MCF}

In this section we consider the motion of an evolving hypersurface $\Gamma = (\Gamma(t))_{t \in I}$ in $\R^3$ driven by the volume-preserving mean curvature flow
\begin{align*}
	V_\Gamma(t) = H_\Gamma(t) - \bar{H}(t),
\end{align*}
where $V_\Gamma$ is the normal velocity, $H_\Gamma$ is the mean curvature given as the sum of the principle curvatures and $\bar{H}(t)$ is the mean value of the mean curvature, defined as
\begin{align*}
	\bar{H}(t) := \mint_{\Gamma(t)} H_{\Gamma(t)}(t,p) \dH^2
				  := \frac{1}{\int_{\Gamma(t)}\limits 1 \dH^2} \int_{\Gamma(t)} H_{\Gamma(t)}(t,p) \dH^2.
\end{align*}
Here, $\bar{H}(t)$ is exactly the right choice to make this flow volume preserving. More precisely, by calculating the first variation of the volume functional we see
\begin{align*}
	\frac{d}{dt} \Vol(\Gamma(t)) & = \int_{\Gamma(t)} V_{\Gamma(t)} \dH^2 = \int_{\Gamma(t)} H_{\Gamma(t)} - \bar{H} \dH^2 \\
										  & = \int_{\Gamma(t)}  H_{\Gamma(t)} \dH^2 - \bar{H} \int_{\Gamma(t)} 1 \dH^2 = 0.
\end{align*}

The hypersurface shall evolve inside a container $\Omega$ and remain in contact with the fixed boundary $\d \Omega$. For two parameters $a \in \R$ and $b > 0$ we impose a dynamic boundary condition of relaxation type
\begin{align*}
	v_{\d D}(t) = a + b \kappa_{\d D}(t) + \cos(\alpha(t)),
\end{align*}
where $v_{\d D}$ is the normal boundary velocity of the contact curve, $\kappa_{\d D}$ is its geodesic curvature with respect to $\d \Omega$ and $\cos(\alpha(t)) = \skp{n_\Gamma}{n_D}$ is the contact angle of $\Gamma$ and $D$. The precise assumptions shall be introduced now.

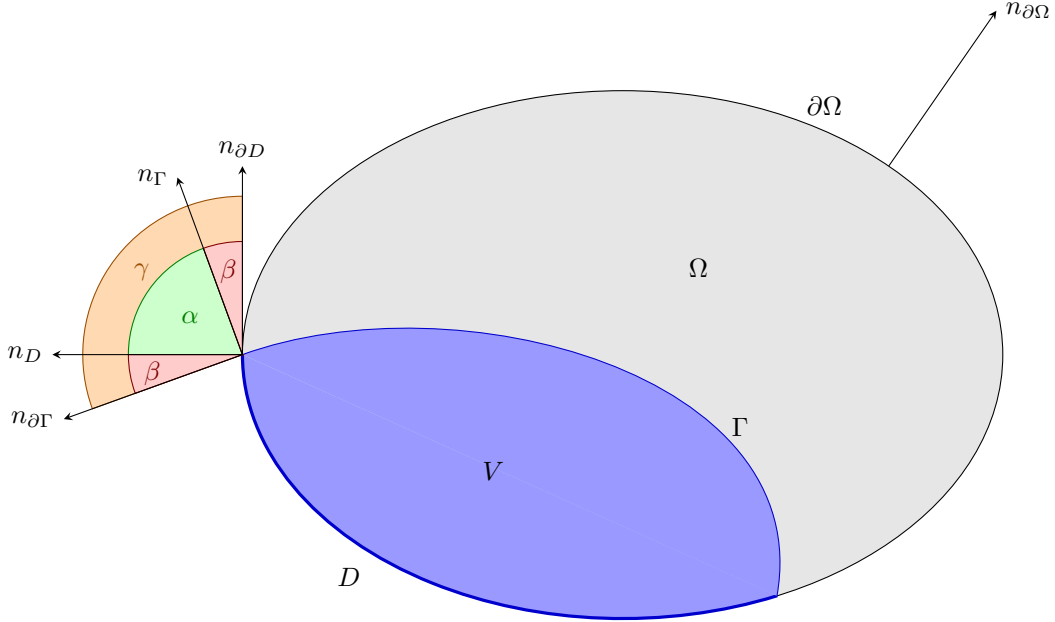
\begin{figure}[htbp]
	\centering
	\begin{tikzpicture}[scale=1,>=stealth]
		\filldraw[fill=black!10!white,draw=black] (0,0) ellipse (5 and 3.5);
		\draw[thin,white] (2.3,3.3) circle (0pt) node[right,black] {$\d \Omega$};
		\draw[thin,black!10!white] (1,1.4) circle (0pt) node[below,black] {$\Omega$};
		\filldraw[fill=blue!40!white,draw=black] (-5,0) arc (180:294:5 and 3.5);
		\draw[very thick,blue!80!black] (-5,0) arc (180:294:5 and 3.5);
		\draw[thin,white] (-3.6,-2.7) circle (0pt) node[below,black] {$D$};
		\filldraw[fill=blue!40!white,draw=blue!80!black] (-5,0) .. controls +(20:3cm) and +(80:3cm) .. (2.03,-3.20) node[near end,above=3pt,black] {$\Gamma$};
		\draw[thin,blue!40!white] (-1.7,-1.3) circle (0pt) node[below,black] {$V$};
		\filldraw[fill=orange!30!white, draw=orange!60!black] (-5,0) -- (-5,2.1) arc (90:200:2.1) -- cycle;
		\draw[thin,white] (-5,0) circle (0pt) node[above=1.1cm,left=1.1cm,color=orange!60!black] {$\gamma$};
		\filldraw[fill=green!20!white, draw=green!50!black] (-5,0) -- +(110:1.5) arc (110:180:1.5) -- cycle;
		\draw[thin,white] (-5,0) circle (0pt) node[above=0.5cm,left=0.45cm,color=green!50!black] {$\alpha$};
		\filldraw[fill=red!20!white, draw=red!50!black] (-5,0) -- (-5,1.5)  arc (90:110:1.5) -- cycle;
		\draw[thin,white] (-5,0) circle (0pt) node[above=1.1cm,left=-0.05cm,color=red!50!black] {$\beta$};
		\filldraw[fill=red!20!white, draw=red!50!black] (-5,0) -- (-6.5,0)  arc (180:200:1.5) -- cycle;
		\draw[thin,white] (-5,0) circle (0pt) node[below=0.25cm,left=0.95cm,color=red!50!black] {$\beta$};
		\draw[->] (-5,0) -- (-5,2.5) node[above] {$n_{\d D}$};
		\draw[->] (-5,0) -- (-7.5,0) node[left] {$n_D$};
		\draw[->] (-5,0) -- +(20:-2.5) node[left] {$n_{\d \Gamma}$};
		\draw[->] (-5,0) -- +(110:2.5) node[left] {$n_\Gamma$};
		\draw[->] (3.5,2.5) -- +(55.5:2.5) node[right] {$n_{\d \Omega}$};
	\end{tikzpicture}
	\caption{General situation and notation}
	\label{fig:Situation}
\end{figure}

Let $\Omega \subseteq \R^3$ be an open, non-empty, connected domain with smooth boundary $\d \Omega$. Furthermore, let $\Gamma \subseteq \Omega$ be a connected, smooth hypersurface with boundary $\d \Gamma$ such that $\Gamma \cup \d \Gamma$ is compact and $\emptyset \neq \d \Gamma \subseteq \d \Omega$. $V \subseteq \Omega$ denotes the region between $\Gamma$ and $\d \Omega$ and $D$ shall be defined as $D := \d V \cap \d \Omega$. In particular, we have $\d D = \d \Gamma$.
For a point $p \in \Gamma$ we denote the exterior normal to $\Gamma$ in $p$ by $n_\Gamma(p)$, where the term ``exterior'' should be understood with respect to $V$. Analogously, for the normal $n_{\d \Omega}(p)$ for $p \in \d \Omega$, which coincides with $n_D(p)$ if $p \in D \subseteq \d \Omega$. Furthermore, for a point $p \in \d \Gamma$ we want to denote by $n_{\d \Gamma}(p)$ and $n_{\d D}(p)$ the outer conormals to $\d \Gamma$ and $\d D$ in $p$. In addition, we define the tangent vector to the curve $\d \Gamma$ by $\vec{\tau}(p) := \frac{c'(t)}{|c'(t)|}$ and its curvature vector by $\vec{\kappa}(p) := \frac{1}{|c'(t)|} \left(\frac{c'(t)}{|c'(t)|}\right)'$, where $c: (t - \epsilon, t + \epsilon) \longrightarrow \d \Gamma$ is a parametrization of $\d \Gamma$ around $p \in \d \Gamma$ with $c(t) = p$. Moreover, we define the angles $\alpha(p) := \angle(n_\Gamma(p), n_D(p))$, $\beta(p) := \angle(n_D(p), n_{\d \Gamma}(p))$ and $\gamma(p) := \angle(n_{\d D}(p),n_{\d \Gamma}(p))$ for $p \in \d \Gamma$. We assume throughout the whole paper
\begin{align}\label{eq:AngleAssumption}
	0 < \alpha(p) < \pi \qquad \text{for all } p \in \d \Gamma,
\end{align}
which will be crucial later on. The whole situation is sketched in Figure \ref{fig:Situation}.

\begin{rem}\label{rem:ONBsAngles}
(i) The two triplets $\left(\vec{\tau}(p), n_{\d \Gamma}(p), n_\Gamma(p)\right)$ and $\left(\vec{\tau}(p), n_D(p), n_{\d D}(p)\right)$ form two right-handed orthonormal bases of $\R^3$ in every point $p \in \d \Gamma$ if we choose the right orientation of $c$. Moreover, we note that $\left\{\vec{\tau}(p), n_{\d \Gamma}(p)\right\}$ and $\left\{\vec{\tau}(p), n_{\d D}(p)\right\}$ are orthonormal bases of $T_p\Gamma$ and $T_pD$, respectively. \\
(ii) We have $\alpha = \frac{\pi}{2} - \beta$ and $\frac{\pi}{2} + \beta = \gamma$, which shows
\begin{align}\label{eq:AngleRelations}
	\skp{n_\Gamma}{n_D} & = \cos(\alpha) \notag \\
	\skp{n_D}{n_{\d \Gamma}} & = \cos(\beta) = \sin(\alpha) \\
	\skp{n_{\d D}}{n_{\d \Gamma}} & = \cos(\gamma) = -\cos(\alpha) \notag,
\end{align}
since all vectors have unit length.
\end{rem}

\begin{figure}[htbp]
	\centering
	\begin{tikzpicture}[scale=1]
		\draw[very thin] (-6:1.5) arc (-6:195:1.5);
		\draw[very thin] (2:2) arc (2:186:2);
		\draw[very thin] (13.5:3) arc (13.5:175:3);
		\draw[very thin] (18:3.5) arc (18:170:3.5);
		\draw[] (0,1.67) --  node[right=15pt,below=-8pt,fill=white]{$\rho(t,q)$} (0,2.5);
		\draw[] (0,1.67) --  (0,2.5);
		\draw[very thick] (8:2.5) arc (8:180:2.5) node[below=2pt] {$\Gamma^*$};
		\draw[thick] (3.05,0.8) .. controls (2,2) and (-3,2) .. (-3.3,0.47) node[below=10pt,left=-5pt] {$\Gamma_\rho(t)$};
		\fill[thick] (0,2.5) node[above] {$q$} circle (2pt);
		\draw[thick] (4,2) arc (-30:-150:5) node[left] {$\d \Omega$};
		\draw[shift={(0.1,-0.5)},<->] (3.35,1.3) arc (-44:-70:5) node[above=10pt,right=30pt] {$w$};
	\end{tikzpicture}
	\caption{The distance function $\rho$}
	\label{fig:Rho}
\end{figure}
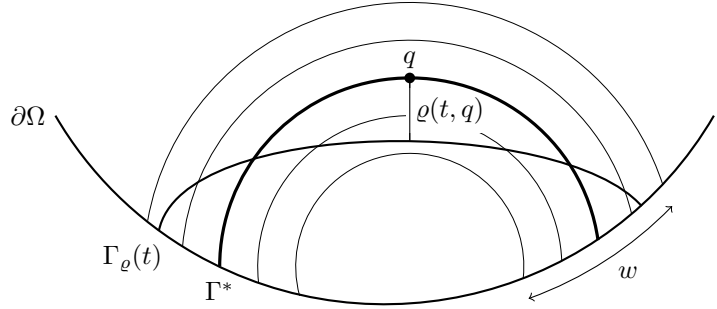

Our first goal is to perform a Hanzawa transformation and write the evolving hypersurface as a family of graphs of a time-dependent distance function $\rho: [0,T] \times \Gamma^* \longrightarrow (-\epsilon_0, \epsilon_0)$ over a fixed reference hypersurface $\Gamma^*$. This reference hypersurface $\Gamma^*$ is supposed to have the same properties as described above. The distance $\rho(t,q)$ of a point $q \in \Gamma^*$ shall be measured in normal direction as indicated in Figure \ref{fig:Rho}, but this is not possible for a boundary point $q \in \d \Gamma^*$. In our situation we need some correction term to ensure that the evolving hypersurface $\Gamma$ neither crosses $\d \Omega$ nor detaches from it.

To this purpose we need to introduce a curvilinear coordinate system $\Psi$ as invented by Vogel \cite{Vog00}. We introduce this coordinate system now because with its help we can write an evolving hypersurface as a graph over the fixed reference surface $\Gamma^*$.

For $q \in \d \Gamma^*$ and $w \in (-\epsilon_0, \epsilon_0)$ with $\epsilon_0 > 0$ sufficiently small there is a smooth function
\begin{align*}
	\tilde{t}: \d \Gamma^* \times (-\epsilon_0, \epsilon_0) \longrightarrow \R: (q,w) \longmapsto \tilde{t}(q,w)
\end{align*}
such that
\begin{align*}
	q + w n_{\Gamma^*}(q) + \tilde{t}(q,w) n_{\d \Gamma^*}(q) \in \d \Omega \qquad \forall \ w \in (-\epsilon_0, \epsilon_0).
\end{align*}
Obviously, $\tilde{t}(q,0) = 0$ and we can extend $\tilde{t}$ smoothly to a function
\begin{align*}
	t: \Gamma^* \times (-\epsilon_0, \epsilon_0) \longrightarrow \R: (q,w) \longmapsto t(q,w)
\end{align*}
such that $t(q,0) = 0$ for all $q \in \Gamma^*$. Next we will use a special coordinate system
\begin{align}\label{eq:CurvilinearCoordinates}
	\Psi: \Gamma^* \times (-\epsilon_0, \epsilon_0) \longrightarrow \Omega: (q,w) \longmapsto \Psi(q,w) := q + w n_{\Gamma^*}(q) + t(q,w) T(q),
\end{align}
where $T: \Gamma^* \longrightarrow \R^3$ is an arbitrary tangential vector field, that coincides with $n_{\d \Gamma^*}$ on $\d \Gamma^*$ and vanishes outside a small neighborhood of $\d \Gamma^*$. By construction this curvilinear coordinate system satisfies $\Psi(q,0) = q$ for all $q \in \Gamma^*$ and $\Psi(q,w) \in \d \Omega$ for all $q \in \d \Gamma^*$ and all $w \in (-\epsilon_0, \epsilon_0)$. Moreover, we can choose $\epsilon_0 > 0$ small enough so that $\Psi$ is a diffeomorphism onto its image. The existence of such a curvilinear coordinate system is guaranteed due to (\ref{eq:AngleAssumption}) which is a result from \cite{Vog00}, where one can also find more technical details concerning this $\Psi$.

We define our evolving hypersurface $\Gamma := (\Gamma_\rho(t))_{t \in I}$ via $\Gamma_\rho(t) := \im(\Psi(\bullet,\rho(t,\bullet)))$ and observe that by our construction of $\Psi$ we have $\Gamma_0(t) = \Gamma^*$ for all $t \in [0,\infty)$. We assume that $\rho$ is smooth enough such that all the upcoming terms are defined.

The precise flow that we want to consider is
\begin{align}\label{eq:Flow1}
	V_\Gamma(\Psi(q,\rho(t,q))) &= H_\Gamma(\Psi(q,\rho(t,q))) - \bar{H}(t) & &\text{in } \Gamma^*, \\ \label{eq:Flow2}
	v_{\d D}(\Psi(q,\rho(t,q))) &= a + b \kappa_{\d D}(\Psi(q,\rho(t,q))) & & \notag \\
										 &+ \skp{n_\Gamma(\Psi(q,\rho(t,q)))}{n_D(\Psi(q,\rho(t,q)))} & &\text{on } \d \Gamma^*.
\end{align}
This flow is motivated by the gradient flow of the energy functional
\begin{align}\label{eq:EnergyFuncVolume}
	E(\Gamma) := \int_\Gamma 1 \dH^2 - a \int_D 1 \dH^2 + b \int_{\d \Gamma} 1 \dH^1,
\end{align}
together with a constant volume constraint, where $a, b \in \R$ with $b > 0$ are given. If we vary the hypersurface by a smooth
\begin{align}\label{eq:Variation}
	\psi\colon \R \times \Gamma \longrightarrow \R^3\colon (t,p) \longmapsto \psi(t,p)% := p + t \zeta(p)
\end{align}
such that $\psi(0,p)=p$ and $\partial_t \psi(0,p)=\zeta(p)$ for all $p\in \Gamma$ for a vector field 
\begin{align}\label{eq:FeasibleSet}
	\zeta \in \mathcal{F}(\Gamma) := \{f \in C^\infty(\Gamma;\R^3) \mid f|_{\d \Gamma} \cdot n_D = 0\}.
\end{align}
and calculate the first variation of (\ref{eq:EnergyFuncVolume}) we end up with
\begin{align}\label{eq:FirstVariation}
	(\delta E(\Gamma))(\zeta) = \int_\Gamma (\lambda - H_\Gamma) (n_\Gamma \cdot \zeta) \dH^2 + \int_{\d \Gamma} (n_{\d \Gamma} - a n_{\d D} - b \vec{\kappa}) \cdot \zeta \dH^1
\end{align}
for some $\lambda \in \R$ and all $\zeta \in \mathcal{F}(\Gamma)$, where we dropped the argument ``$p$'' for a more convenient notation. Searching for necessary conditions that a critical hypersurface of the energy (\ref{eq:EnergyFuncVolume}) has to satisfy we end up with
\begin{align}\label{eq:NecessrayConditions1}
	0 &= H_\Gamma - const. & &\text{on } \Gamma \\ \label{eq:NecessrayConditions2}
	0 &= a + b \kappa_{\d D} + \skp{n_\Gamma}{n_D} & &\text{on } \d \Gamma.
\end{align}
This shows that a stationary hypersurface of the flow (\ref{eq:Flow1})-(\ref{eq:Flow2}) is a critical point for the energy (\ref{eq:EnergyFuncVolume}) and therefore a possible minimizer, for details we refer to \cite{Mue13}.

In the following we want to linearize these two equations around $\rho \equiv 0$. To this end we need some relations concerning the curvilinear coordinate system, which we will derive now. Whenever a star $*$ is added to some of the previous terms, we mean the respective term for the hypersurface.

Let $q \in \d \Gamma^* = \d D^*$ be fixed, $U \subseteq \R^3$ be an open neighborhood of $q$ and assume that $F: \R^3 \longrightarrow \R$ is a smooth function describing $U \cap \d \Omega$ as zero-level-set, i.e.,
\begin{align*}
	U \cap \d \Omega = \{p \in \R^3 \mid F(p) = 0\}.
\end{align*}
Then $\nabla F(q) \perp T_q \d \Omega$ if $\nabla F(a)\neq 0 $ and w.l.o.g. we assume $\frac{\nabla F}{\norm{\nabla F}} = n_{D^*}$ on $D^*$ - otherwise we replace $F$ by $-F$. By the choice of $\Psi$ we obtain for all $q \in \d \Gamma^*$
\begin{align*}
	0 = F(\Psi(q,w)) = F(q + w n_{\Gamma^*}(q) + t(q,w) n_{\d \Gamma^*}(q)) \qquad \forall \ w \in (-\epsilon_0, \epsilon_0).
\end{align*}
Differentiating this equation with respect to $w$ and setting $w = 0$ gives
\begin{align*}
	0 & = \nabla F(\Psi(q,0)) \cdot \d_w \Psi(q,0) = \nabla F(q) \cdot (n_{\Gamma^*}(q) + t_w(q,0) n_{\d \Gamma^*}(q)) \\
	  & = \skp{\norm{\nabla F(q)} n_{D^*}(q)}{n_{\Gamma^*}(q)} + t_w(q,0) \skp{\norm{\nabla F(q)} n_{D^*}(q)}{n_{\d \Gamma^*}(q)}.
\end{align*}
Keeping the assumption (\ref{eq:AngleAssumption}) in mind one can rewrite this identity with the help of (\ref{eq:AngleRelations}) to get
\begin{align*}
	t_w(q,0) = -\frac{\norm{\nabla F(q)} \skp{n_{D^*}(q)}{n_{\Gamma^*}(q)}}{\norm{\nabla F(q)} \skp{n_{D^*}(q)}{n_{\d \Gamma^*}(q)}}
				= -\frac{\cos(\alpha^*(q))}{\sin(\alpha^*(q))} = -\cot(\alpha^*(q)).
\end{align*}
Hence we can write the vector $\d_w \Psi(q,0)$ as
\begin{align}\label{eq:DwPsi}
	\d_w \Psi(q,0) = n_{\Gamma^*}(q) - \cot(\alpha^*(q)) n_{\d \Gamma^*}(q) \qquad \forall \, q \in \d \Gamma^*.
\end{align}
Utilizing (\ref{eq:AngleRelations}) this shows that on the boundary $\d \Gamma^*$ the vector $\d_w \Psi(0)$ has the following coordinates with respect to the two orthonormal bases introduced in Remark \ref{rem:ONBsAngles}:
\begin{align}\label{eq:CoordinatesDwPsi}
	\begin{array}{rllrl}
		\skp{\d_w \Psi(0)}{\vec{\tau}^*}    & = 0,               & \qquad & \skp{\d_w \Psi(0)}{\vec{\tau}^*} & = 0, \\
		\skp{\d_w \Psi(0)}{n_{\d \Gamma^*}} & = -\cot(\alpha^*), & \qquad & \skp{\d_w \Psi(0)}{n_{D^*}}      & = 0, \\
		\skp{\d_w \Psi(0)}{n_{\Gamma^*}}    & = 1,               & \qquad & \skp{\d_w \Psi(0)}{n_{\d D^*}}   & = \frac{1}{\sin(\alpha^*)}.
	\end{array}
\end{align}
We note that the relation $\skp{\d_w \Psi(0)}{n_{\Gamma^*}} = 1$ does also hold in $\Gamma^*$ since
\begin{align}\label{eq:DwPsiNGamma}
	\skp{\d_w \Psi(0)}{n_{\Gamma^*}} = \skp{n_{\Gamma^*} + t_w(w) T}{n_{\Gamma^*}}
		= \underbrace{\skp{n_{\Gamma^*}}{n_{\Gamma^*}}}_{= 1} + t_w(w) \underbrace{\skp{T}{n_{\Gamma^*}}}_{= 0} = 1,
\end{align}
which is a fact used in Lemma \ref{lem:LinearVGamma} below.

\subsection{Linearization of the MCF and the dynamic boundary condition}\label{ssec:LinearMCF}

In this subsection we want to linearize the volume-preserving MCF as given by (\ref{eq:Flow1})-(\ref{eq:Flow2}) around $\rho \equiv 0$, which corresponds to a linearization around $\Gamma^*$. This result will be distributed over several lemmas and uses calculations from \cite{Dep10}.

\begin{lemma}\label{lem:LinearVGamma}
For all $q \in \Gamma^*$ and all $t \in [0,\infty)$ we have
\begin{align*}
	\left.\frac{d}{d\epsilon} V_\Gamma(\Psi(q,\epsilon\rho(t,q)))\right|_{\epsilon=0} = \d_t \rho(t,q).
\end{align*}
\end{lemma}
\begin{proof}
This follows in exactly the same manner as in the linearization of the normal boundary velocity in the upcoming lemma (cf. \cite{Mue13} for details).
\end{proof}

\begin{lemma}\label{lem:LinearVDD}
For all $q \in \d \Gamma^*$ and all $t \in [0,\infty)$ we have
\begin{align*}
	\left.\frac{d}{d\epsilon} v_{\d D}(\Psi(q,\epsilon\rho(t,q)))\right|_{\epsilon=0} = \frac{1}{\sin(\alpha^*(q))} \d_t \rho(t,q).
\end{align*}
\end{lemma}
\begin{proof}
Since $\Psi(q,\rho(\cdot,q)): (t - \epsilon, t + \epsilon) \longrightarrow \R^3$ is a curve with $\Psi(q,\rho(s,q)) \in \Gamma_\rho(s)$, we observe by the definition of the normal boundary velocity
\begin{align}\label{eq:StructureVDD}
	v_{\d D}(\Psi(q,\rho(t,q))) & = n_{\d D}(\Psi(q,\rho(t,q))) \cdot \frac{d}{dt} \Psi(q,\rho(t,q)) \notag \\
										 & = (n_{\d D}(\Psi(q,\rho(t,q))) \cdot \d_w \Psi(q,\rho(t,q))) \d_t \rho(t,q)
\end{align}
and with the help of the product rule this leads to
\begin{align*}
	\left.\frac{d}{d\epsilon} v_{\d D}(\Psi(q,\epsilon\rho(t,q)))\right|_{\epsilon=0}
		& = \left.\frac{d}{d\epsilon} (n_{\d D}(\Psi(q,\epsilon\rho(t,q))) \cdot \d_w \Psi(q,\epsilon\rho(t,q)))\right|_{\epsilon=0} \underbrace{\left.\d_t \epsilon\rho(t,q)\right|_{\epsilon=0}}_{= 0} \\
		& + \left.(n_{\d D}(\Psi(q,\epsilon\rho(t,q))) \cdot \d_w \Psi(q,\epsilon\rho(t,q)))\right|_{\epsilon=0} \d_t \rho(t,q) \\
		& = (n_{\d D^*}(q) \cdot \d_w \Psi(q,0)) \d_t \rho(t,q) = \frac{1}{\sin(\alpha^*(q))} \d_t \rho(t,q),
\end{align*}
where we used (\ref{eq:CoordinatesDwPsi}) in the last line.
\end{proof}

\begin{lemma}\label{lem:LinearHGamma}
For all $q \in \Gamma^*$ and all $t \in [0,\infty)$ we have
\begin{align*}
	\left.\frac{d}{d\epsilon} H_\Gamma(\Psi(q,\epsilon\rho(t,q)))\right|_{\epsilon=0}
		& = \Delta_{\Gamma^*} \rho(t,q) + |\sigma^*|^2(q) \rho(t,q) \\
		& + \left(\nabla_{\Gamma^*} H_{\Gamma^*}(q) \cdot P\left(\d_w \Psi(q,0)\right)\right) \rho(t,q),
\end{align*}
where $\Delta_{\Gamma^*}$ denotes the Laplace-Beltrami operator and $\nabla_{\Gamma^*}$ the surface gradient of $\Gamma^*$, $|\sigma^*|^2$ is defined as $|\sigma^*|^2 := (\kappa_1^*)^2 + (\kappa_2^*)^2$ with the principal curvatures $\kappa_1^*, \kappa_2^*$ of $\Gamma^*$ and $P$ denotes the orthogonal projection onto the tangent space of $\Gamma^*$.
\end{lemma}
\begin{proof}
Exactly the same as in Lemma 3.5 of \cite{Dep10}. The difference to the proof there is that we do not use $H_{\Gamma^*} \equiv const.$ in the last step because we did not assume $\Gamma^*$ to be a stationary solution of the energy (\ref{eq:EnergyFuncVolume}).
\end{proof}

\begin{rem}
For a boundary point $q \in \d \Gamma^*$ we express the term $P(\d_w \Psi(q,0))$ using (\ref{eq:CoordinatesDwPsi}) as \begin{align*}
	P(\d_w \Psi(q,0)) = -\cot(\alpha(q)) n_{\d \Gamma^*}(q).
\end{align*}
\end{rem}

\begin{lemma}\label{lem:LinearHMean}
For the linearization of the mean value of the mean curvature we get
\begin{align*}
	\left.\frac{d}{d\epsilon} \bar{H}(\epsilon \rho(t))\right|_{\epsilon=0}
		& = \mint_{\Gamma^*} (\Delta_{\Gamma^*} + |\sigma^*|^2(q) - H_{\Gamma^*}(q)^2 + \bar{H}^* H_{\Gamma^*}(q)) \rho(t,q) \dH^2 \\
		& - \frac{1}{\int_{\Gamma^*}\limits 1 \dH^2} \int_{\d \Gamma^*} \left(H_{\Gamma^*}(q) - \bar{H}^*\right) \cot(\alpha(q)) \rho(t,q) \dH^1.
\end{align*}
\end{lemma}
\begin{proof}
Here we rename the surfaces $\Gamma_\rho(t)$ in a way that makes the following calculations easier. We fix a time $t$ and for $\epsilon \in (-\epsilon_0,\epsilon_0)$ we set
\begin{align*}
	\tilde{\Gamma}(\epsilon) := \im(\Psi(\bullet,\epsilon \rho(t,\bullet))).
\end{align*}
Obviously, these new hypersurfaces $\tilde{\Gamma}$ are just renamed versions of the previous $\Gamma_\rho(t)$ because $\tilde{\Gamma}(\epsilon) = \Gamma_{\epsilon \rho}(t)$. But now $\epsilon$ can be considered to be the time parameter of the evolution. We will write $\tilde{H}_{\tilde{\Gamma}(\epsilon)}$, $\tilde{V}_{\tilde{\Gamma}(\epsilon)}$, etc. for the terms of the evolving hypersurface $\left(\tilde{\Gamma}(\epsilon)\right)_{\epsilon \in (-\epsilon_0,\epsilon_0)}$. In particular, an expression related to $\tilde{\Gamma}$ evaluated for $\epsilon = 0$ will result in the respective expression on $\Gamma^*$, because $\tilde{\Gamma}(0) = \Gamma^*$. \\
With this new notation (cf. Lemma 3.5 from \cite{Dep10}) we write
\begin{align*}
	\bar{H}(\epsilon \rho(t)) = \left(\int_{\tilde{\Gamma}(\epsilon)} 1 \dH^2\right)^{-1} \left(\int_{\tilde{\Gamma}(\epsilon)} \tilde{H}_{\tilde{\Gamma}(\epsilon)}(\epsilon,p) \dH^2\right).
\end{align*}
Then we get
\begin{align}\label{eq:LinearHMean}
	\left.\frac{d}{d\epsilon} \bar{H}(\epsilon \rho(t))\right|_{\epsilon=0}
		& = \left(\int_{\Gamma^*} 1 \dH^2\right)^{-1} \left(\left.\frac{d}{d\epsilon} \int_{\tilde{\Gamma}(\epsilon)} \tilde{H}_{\tilde{\Gamma}(\epsilon)}(\epsilon,p) \dH^2\right|_{\epsilon=0}\right) \notag \\
		& - \left(\int_{\Gamma^*} 1 \dH^2\right)^{-1} \left(\left.\frac{d}{d\epsilon} \int_{\tilde{\Gamma}(\epsilon)} 1 \dH^2\right|_{\epsilon=0}\right) \underbrace{\left(\mint_{\Gamma^*} H_{\Gamma^*}(q) \dH^2\right)}_{= \bar{H}(\mathbb{O})}.
\end{align}
With the help of the transport theorem from the appendix of \cite{GW06}, the derivative of the area integral can be written as
\begin{align}\label{eq:DInt1}
	\left.\frac{d}{d\epsilon} \int_{\tilde{\Gamma}(\epsilon)} 1 \dH^2\right|_{\epsilon=0}
		& = -\int_{\tilde{\Gamma}(0)} \tilde{V}_{\tilde{\Gamma}(0)}(0,q) H_{\Gamma^*}(q) \dH^2 + \int_{\d \tilde{\Gamma}(0)} \tilde{v}_{\d \tilde{\Gamma}(0)}(0,q) \dH^1
\end{align}
and we observe
\begin{align*}
	\tilde{V}_{\tilde{\Gamma}(\epsilon)}(\epsilon,\Psi(q,\epsilon\rho(t,q)))
		& = \tilde{n}_{\tilde{\Gamma}(\epsilon)}(\epsilon,\Psi(q,\epsilon\rho(t,q))) \cdot \frac{d}{d\epsilon} \Psi(q,\epsilon\rho(t,q)) \\
		& = \left(\tilde{n}_{\tilde{\Gamma}(\epsilon)}(\epsilon,\Psi(q,\epsilon\rho(t,q))) \cdot \d_w \Psi(q,\epsilon\rho(t,q))\right) \rho(t,q).
\end{align*}
Evaluated in $\epsilon = 0$ we get using (\ref{eq:CoordinatesDwPsi})
\begin{align}\label{eq:VelocityRho}
	\tilde{V}_{\tilde{\Gamma}(0)}(0,q) & = \tilde{V}_{\tilde{\Gamma}(0)}(0,\Psi(q,0)) = (n_{\Gamma^*}(q) \cdot \d_w \Psi(q,0)) \rho(t,q) = \rho(t,q).
\end{align}
Exactly the same calculations with $\tilde{n}_{\d \tilde{\Gamma}}$ instead of $\tilde{n}_{\tilde{\Gamma}}$ show $\tilde{v}_{\d \tilde{\Gamma}(0)}(0,q) = -\cot(\alpha(q)) \rho(t,q)$ for all $q \in \d \Gamma^*$. Hence (\ref{eq:DInt1}) can be rewritten as
\begin{align*}
	\left.\frac{d}{d\epsilon} \int_{\tilde{\Gamma}(\epsilon)} 1 \dH^2\right|_{\epsilon=0} = -\int_{\Gamma^*} H_{\Gamma^*}(q) \rho(t,q) \dH^2 - \int_{\d \Gamma^*} \cot(\alpha(q)) \rho(t,q) \dH^1.
\end{align*}
With the help of the transport theorem of \cite{GW06} and Lemma 5.1 from \cite{Dep10} the derivative of the curvature integral is
\begin{align*}
	&\left.\frac{d}{d\epsilon} \int_{\tilde{\Gamma}(\epsilon)} \tilde{H}_{\tilde{\Gamma}(\epsilon)}(\epsilon,p) \dH^2\right|_{\epsilon=0} \\
	& \hspace*{5mm} = \left.\int_{\tilde{\Gamma}(\epsilon)} \Delta_{\tilde{\Gamma}(\epsilon)} \tilde{V}_{\tilde{\Gamma}(\epsilon)}(\epsilon,p) + |\tilde{\sigma}|^2(\epsilon,p) \tilde{V}_{\tilde{\Gamma}(\epsilon)}(\epsilon,p) \dH^2\right|_{\epsilon=0} \\
	& \hspace*{5mm} - \left.\int_{\tilde{\Gamma}(\epsilon)} \tilde{H}_{\tilde{\Gamma}(\epsilon)}(\epsilon,p)^2 \tilde{V}_{\tilde{\Gamma}(\epsilon)}(\epsilon,p) \dH^2\right|_{\epsilon=0} + \left.\int_{\d \tilde{\Gamma}(\epsilon)} \tilde{H}_{\tilde{\Gamma}(\epsilon)}(\epsilon,p) \tilde{v}_{\d \tilde{\Gamma}(\epsilon)}(\epsilon,p) \dH^1\right|_{\epsilon=0} \\
	& \hspace*{5mm} = \int_{\Gamma^*} \Delta_{\Gamma^*} \rho(t,q) + |\sigma^*|^2(q) \rho(t,q) - H_{\Gamma^*}(q)^2 \rho(t,q) \dH^2 - \int_{\d \Gamma^*} H_{\Gamma^*}(q) \cot(\alpha(q)) \rho(t,q) \dH^1.
\end{align*}
Using these variations we transform equation (\ref{eq:LinearHMean}) as follows
\begin{align*}
	\left.\frac{d}{d\epsilon} \bar{H}(\epsilon \rho(t))\right|_{\epsilon=0}
		& = \mint_{\Gamma^*} \left(\Delta_{\Gamma^*} + |\sigma^*|^2(q) - H_{\Gamma^*}(q)^2 + \bar{H}(\mathbb{O}) H_{\Gamma^*}(q)\right) \rho(t,q) \dH^2 \\
		& + \left(\int_{\Gamma^*} 1 \dH^2\right)^{-1} \left(\int_{\d \Gamma^*} (\bar{H}(\mathbb{O}) - H_{\Gamma^*}(q)) \cot(\alpha(q)) \rho(t,q) \dH^1\right).
\end{align*}
\end{proof}

\begin{lemma}\label{lem:DiffNormedVector}
For a vector-valued $v \in C^1((-\epsilon_0,\epsilon_0))$ with $v(0) \neq 0$ one has
\begin{align*}
	\left.\frac{d}{d\epsilon} \frac{v(\epsilon)}{\norm{v(\epsilon)}}\right|_{\epsilon=0} & = P_{\frac{v(0)}{\norm{v(0)}}}\left(\frac{v'(0)}{\norm{v(0)}}\right),
\end{align*}
where $P_v: \R^n \longrightarrow \R^n: x \longmapsto P_v(x) := x - \skp{x}{v} v$ denotes the projection along $v$.
\end{lemma}
\begin{proof}
Elementary calculations.
\end{proof}

To be able to linearize the boundary condition we need the following lemma.

\begin{lemma}\label{lem:DwPsiRelations}
Let $G \subseteq \bar{\R^2_+}$ be relatively open (i.e., $G$ is open in the subspace topology of $\bar{\R^2_+} \subseteq \R^2$) and let $F: \bar{G} \longrightarrow \Omega$ be the parametrization for a part of $\Gamma^* \cup \d \Gamma^*$ with the properties
\begin{align}\label{eq:Property}
	&\d_1 F(x_0), \d_2 F(x_0) \quad \text{form an orthonormal basis of } \ T_{F(x_0)}\Gamma^* \notag \\
	&\d_1 F(x_0) \times \d_2 F(x_0) = n_{\Gamma^*}(F(x_0)) \\
	&\d_1 F(x_0) = \vec{\tau}^*(F(x_0)) \quad \text{and} \quad \d_2 F(x_0) = n_{\d \Gamma^*}(F(x_0)) \quad \text{on } \ \d \Gamma^* \notag\nomenclature[0.2]{$v \times w$}{Cross product in $\R^n$}
\end{align}
for a fixed $x_0 \in G \cap \d \R^2_+$. Then we have for all $F(x) = q \in \Gamma^*$ and $i \in \{1,2\}$
\begin{align*}
	&\text{(i)}    & \Psi(F(x),0) &= F(x) \qquad \text{ and } \qquad \d_i \Psi(F(x),0) = \d_i F(x)
\end{align*}
and for all $F(x) = q \in \d \Gamma^*$ and $i \in \{1,2\}$
\begin{align*}
	&\text{(ii)}   & \d_w \Psi(F(x),0) &= n_{\Gamma^*}(F(x)) - \cot(\alpha(F(x))) n_{\d \Gamma^*}(F(x)), \\
	&\text{(iii)}  & \skp{\d_i \d_w \Psi(F(x),0)}{n_{\Gamma^*}(F(x))} &= \cot(\alpha(F(x))) \skp{\d_i n_{\Gamma^*}(F(x))}{n_{\d \Gamma^*}(F(x))}
\end{align*}
and for the fixed $F(x_0) = q_0 \in \d \Gamma^*$ and $i \in \{1,2\}$
\begin{align*}
	&\text{(iv)}   & (\d_1 \Psi \times \d_2 \Psi)(q_0,0) &= n_{\Gamma^*}(q_0), \\
	&\text{(v)}    & (\d_w \Psi \times \d_2 \Psi)(q_0,0) &= -\vec{\tau}^*(q_0), \\
	&\text{(vi)}   & (\d_1 \Psi \times \d_w \Psi)(q_0,0) &= -n_{\d \Gamma^*}(q_0) - \cot(\alpha(q_0)) n_{\Gamma^*}(q_0), \\
	&\text{(vii)}  & (\d_1 \Psi \times \d_i \d_w \Psi)(q_0,0) &= \skp{\d_i \d_w \Psi(q_0,0)}{n_{\d \Gamma^*}(q_0)} n_{\Gamma^*}(q_0) \\
	&					& &- \cot(\alpha(q_0)) \skp{\d_i n_{\Gamma^*}(q_0)}{n_{\d \Gamma^*}(q_0)} n_{\d \Gamma^*}(q_0), \\
	&\text{(viii)} & (\d_i \d_w \Psi \times \d_2 \Psi)(q_0,0) &= \skp{\d_i \d_w \Psi(q_0,0)}{\vec{\tau}^*(q_0)} n_{\Gamma^*}(q_0) \\
	&					& &- \cot(\alpha(q_0)) \skp{\d_i n_{\Gamma^*}(q_0)}{n_{\d \Gamma^*}(q_0)} \vec{\tau}^*(q_0).
\end{align*}
\end{lemma}
\begin{proof}
(i) The first equation is a property of the curvilinear coordinate system as we constructed it. The second claim we obtain by differentiation. \\
(ii) Equation (\ref{eq:DwPsi}) for $q = F(x) \in \d \Gamma^*$. \\
(iii) Using (ii) and (\ref{eq:DwPsiNGamma}) we get
\begin{align*}
	0 & = \d_i 1 = \d_i \skp{\d_w \Psi(F(x),0)}{n_{\Gamma^*}(F(x))} \\
	  & = \skp{\d_i\d_w \Psi(F(x),0)}{n_{\Gamma^*}(F(x))} + \skp{\d_w \Psi(F(x),0)}{\d_i n_{\Gamma^*}(F(x))}.
\end{align*}
Since $\d_i n_{\Gamma^*}$ lies in the $\vec{\tau}^*$-$n_{\d \Gamma^*}$-plane we see
\begin{align*}
	\skp{\d_i\d_w \Psi(F(x),0)}{n_{\Gamma^*}(F(x))}
		& = -\skp{n_{\d \Gamma^*}(F(x))}{\d_i n_{\Gamma^*}(F(x))} \underbrace{\skp{\d_w \Psi(F(x),0)}{n_{\d \Gamma^*}(F(x))}}_{= -\cot(\alpha(F(x)))} \\
		& - \skp{\vec{\tau}^*(F(x))}{\d_i n_{\Gamma^*}(F(x))} \underbrace{\skp{\d_w \Psi(F(x),0)}{\vec{\tau}^*(F(x))}}_{= 0} \\
		& = \cot(\alpha(F(x))) \skp{n_{\d \Gamma^*}(F(x))}{\d_i n_{\Gamma^*}(F(x))},
\end{align*}
where we used (\ref{eq:CoordinatesDwPsi}). \\
(iv) With (i) and (\ref{eq:Property}) we obtain
\begin{align*}
	(\d_1 \Psi \times \d_2 \Psi)(F(x_0),0) = (\d_1 F \times \d_2 F)(x_0) = n_{\Gamma^*}(F(x_0)).
\end{align*}
(v) By (i), (ii) and (\ref{eq:Property}) we have
\begin{align*}
	(\d_w \Psi \times \d_2 \Psi)(F(x_0),0)
		& = ((n_{\Gamma^*} \circ F) \times \d_2 F)(x_0) - \cot(\alpha(F(x_0))) (\d_2 F \times \d_2 F)(x_0) \\
		& = -\vec{\tau}^*(F(x_0)).
\end{align*}
(vi) Similar to (v), where we use (iv) in addition, we get
\begin{align*}
	(\d_1 \Psi \times \d_w \Psi)(F(x_0),0)
		& = (\d_1 F \times (n_{\Gamma^*} \circ F))(x_0) - \cot(\alpha(F(x_0))) (\d_1 F \times \d_2 F)(x_0) \\
		& = -n_{\d \Gamma^*}(F(x_0)) - \cot(\alpha(F(x_0))) n_{\Gamma^*}(F(x_0)).
\end{align*}
(vii) Using the previous equations one observes
\begin{align*}
	(\d_1 \Psi \times \d_i \d_w \Psi)(F(x_0),0)
		& = (\d_1 F \times (\skp{\d_i \d_w \Psi(F(x_0),0)}{\d_1 F} \d_1 F \\
		& + \skp{\d_i \d_w \Psi(F(x_0),0)}{\d_2 F(x_0)} \d_2 F \\
		& + \skp{\d_i \d_w \Psi(F(x_0),0)}{(n_{\Gamma^*} \circ F)} (n_{\Gamma^*} \circ F)))(x_0) \\
		& = \skp{\d_i \d_w \Psi(F(x_0),0)}{n_{\d \Gamma^*}(F(x_0))} n_{\Gamma^*}(F(x_0)) \\
		& - \cot(\alpha(F(x_0))) \skp{\d_i n_{\Gamma^*}(F(x_0))}{n_{\d \Gamma^*}(F(x_0))} n_{\d \Gamma^*}(F(x_0)).
\end{align*}
(viii) Analogously to (vii).
\end{proof}

\begin{lemma}\label{lem:LinearAngle}
For the linearization of the angle condition we have
\begin{align*}
	& \left.\frac{d}{d\epsilon} \skp{n_\Gamma(t,\Psi(q,\epsilon \rho(t,q)))}{n_D(\Psi(q,\epsilon \rho(t,q)))}\right|_{\epsilon=0} = -\sin(\alpha(q)) (\nabla_{\Gamma^*} \rho(t,q) \cdot n_{\d \Gamma^*}(q)) \\
	& + \cos(\alpha(q)) II_{\Gamma^*}(n_{\d \Gamma^*}(q),n_{\d \Gamma^*}(q)) \rho(t,q)
	  - II_{D^*}(n_{\d D^*}(q),n_{\d D^*}(q)) \rho(t,q),
\end{align*}
where $II_{\Gamma^*}$ and $II_{D^*}$ are the second fundamental forms of $\Gamma^*$ and $D^*$ with respect to the normals $n_{\Gamma^*}$ and $n_{D^*}$, respectively.
\end{lemma}
\begin{proof}
By the product rule we get
\begin{align}\label{eq:LinearAngleProd}
	\left.\frac{d}{d\epsilon} n_\Gamma(\Psi(q,\epsilon \rho(t,q))) \cdot n_D(\Psi(q,\epsilon \rho(t,q)))\right|_{\epsilon=0}
		& = \left.\frac{d}{d\epsilon} n_\Gamma(\Psi(q,\epsilon \rho(t,q)))\right|_{\epsilon=0} \cdot n_{D^*}(q) \notag \\
		& + n_{\Gamma^*}(q) \cdot \left.\frac{d}{d\epsilon} \underbrace{n_D(\Psi(q,\epsilon \rho(t,q)))}_{=: (3)}\right|_{\epsilon=0}
\end{align}
and the normal can be written as
\begin{align*}
	n_\Gamma(\Psi(q,\epsilon \rho(t,q))) = \frac{\d_1 (\Psi(q,\epsilon \rho(t,q))) \times \d_2 (\Psi(q,\epsilon \rho(t,q)))}{\norm{\d_1 (\Psi(q,\epsilon \rho(t,q))) \times \d_2 (\Psi(q,\epsilon \rho(t,q)))}}.
\end{align*}
For the vector
\begin{align*}
	v(\epsilon) & := \d_1 (\Psi(q,\epsilon \rho(t,q))) \times \d_2 (\Psi(q,\epsilon \rho(t,q))) \\
					& = ((\d_1 \Psi)(q,\epsilon \rho(t,q)) + (\d_w \Psi)(q,\epsilon \rho(t,q)) \epsilon \d_1 \rho(t,q)) \\
					& \times ((\d_2 \Psi)(q,\epsilon \rho(t,q)) + (\d_w \Psi)(q,\epsilon \rho(t,q)) \epsilon \d_2 \rho(t,q)) \\
					& = \underbrace{(\d_1 \Psi \times \d_2 \Psi)(q,\epsilon \rho(t,q))}_{=: (1)} \\
					& + \underbrace{(\d_w \Psi \times \d_2 \Psi)(q,\epsilon \rho(t,q)) \epsilon \d_1 \rho(t,q) + (\d_1 \Psi \times \d_w \Psi)(q,\epsilon \rho(t,q)) \epsilon \d_2 \rho(t,q)}_{=: (2)}
\end{align*}
we get using Lemma \ref{lem:DiffNormedVector} and Lemma \ref{lem:DwPsiRelations}(iv)
\begin{align}\label{eq:LinearNormal}
	\left.\frac{d}{d\epsilon} n_\Gamma(\Psi(q,\epsilon \rho(t,q)))\right|_{\epsilon=0} = P_{\frac{v(0)}{\norm{v(0)}}}\left(\frac{v'(0)}{\norm{v(0)}}\right) = P_{n_{\Gamma^*}(q)}(v'(0)).
\end{align}
Inspired by \cite{Dep10} we decompose the term $v'(0)$ as
\begin{align}\label{eq:NormalDecomp}
	\left.\frac{d}{d\epsilon} v(\epsilon)\right|_{\epsilon=0} = \left.\frac{d}{d\epsilon} (1)\right|_{\epsilon=0} + \left.\frac{d}{d\epsilon} (2)\right|_{\epsilon=0}
\end{align}
and consider the terms (1), (2) and (3) separately. With Lemma \ref{lem:DwPsiRelations}(vii)+(viii) we first observe
\begin{align*}
	\left.\frac{d}{d\epsilon} (1)\right|_{\epsilon=0} & = \div_{\Gamma^*}(\d_w \Psi(q,0)) \rho(t,q) n_{\Gamma^*}(q) \\
		& + \cot(\alpha(q)) (II_{\Gamma^*}(\vec{\tau}^*(q),n_{\d \Gamma^*}(q)) \vec{\tau}^*(q) + II_{\Gamma^*}(n_{\d \Gamma^*}(q),n_{\d \Gamma^*}(q)) n_{\d \Gamma^*}(q)) \rho(t,q).
\end{align*}
For the second term we obtain with Lemma \ref{lem:DwPsiRelations}(v)+(vi)
\begin{align*}
	\left.\frac{d}{d\epsilon} (2)\right|_{\epsilon=0} = -\nabla_{\Gamma^*} \rho(t,q) + \skp{\d_w \Psi(q,0)}{\nabla_{\Gamma^*} \rho(t,q)} n_{\Gamma^*}(q).
\end{align*}
Inserting this into (\ref{eq:NormalDecomp}) we get
\begin{align*}
	\left.\frac{d}{d\epsilon} v(\epsilon)\right|_{\epsilon=0} & = -\nabla_{\Gamma^*} \rho(t,q) + \div_{\Gamma^*}(\d_w \Psi(q,0) \rho(t,q)) n_{\Gamma^*}(q) \\
		& + \cot(\alpha(q)) (II_{\Gamma^*}(\vec{\tau}^*(q),n_{\d \Gamma^*}(q)) \vec{\tau}^*(q) + II_{\Gamma^*}(n_{\d \Gamma^*}(q),n_{\d \Gamma^*}(q)) n_{\d \Gamma^*}(q)) \rho(t,q).
\end{align*}
Due to (\ref{eq:LinearNormal}) we have to project this vector along the normal $n_{\Gamma^*}$ and multiplying with $n_{D^*}(q)$ this leads to
\begin{align}\label{eq:LinearAnglePart1}
	\left.\frac{d}{d\epsilon} n_\Gamma(\Psi(q,\epsilon \rho(t,q)))\right|_{\epsilon=0} \cdot n_{D^*}(q)
		& = -\sin(\alpha(q)) (\nabla_{\Gamma^*} \rho(t,q) \cdot n_{\d \Gamma^*}(q)) \notag \\
		& + \cos(\alpha(q)) II_{\Gamma^*}(n_{\d \Gamma^*}(q),n_{\d \Gamma^*}(q)) \rho(t,q).
\end{align}
Finally for part (3) we use the curve
\begin{align*}
	c: [0, \epsilon_0) \longrightarrow \Omega: \epsilon \longmapsto c(\epsilon) := \Psi(q,\epsilon\rho(t,q)).
\end{align*}
Then we have $c(\epsilon) \in \d \Omega$ for all $\epsilon \geq 0$, $c(0) = q$ and $c'(0) = \d_w \Psi(q,0) \rho(t,q)$. Since the vector $\d_w \Psi(q,0) = \frac{1}{\sin(\alpha(q))} n_{\d D^*}(q) \in T_qD^*$ and due to (\ref{eq:CoordinatesDwPsi}) we get
\begin{align*}
	\left.\frac{d}{d\epsilon} (3)\right|_{\epsilon=0}
		& = \left.\frac{d}{d\epsilon} n_D(\Psi(q,\epsilon \rho(t,q)))\right|_{\epsilon=0} = \d_{\d_w \Psi(q,0) \rho(t,q)} n_{D^*}(q) \\
		& = \d_{\frac{1}{\sin(\alpha(q))} n_{\d D^*}(q)} n_{D^*}(q) \rho(t,q) = \frac{1}{\sin(\alpha(q))} \d_{n_{\d D^*}(q)} n_{D^*}(q) \rho(t,q).
\end{align*}
Therefore we obtain
\begin{align*}
	n_{\Gamma^*}(q) \cdot \left.\frac{d}{d\epsilon} (3)\right|_{\epsilon=0} = \frac{1}{\sin(\alpha(q))} \skp{n_{\Gamma^*}(q)}{\d_{n_{\d D^*}(q)} n_{D^*}(q)} \rho(t,q).
\end{align*}
Writing $n_{\Gamma^*} = \sin(\alpha) n_{\d D^*} + \cos(\alpha) n_{D^*}$ and considering the fact that $\d_{n_{\d D^*}(q)} n_{D^*}(q)$ lies in the $n_{\d D^*}$-$\vec{\tau}^*$-plane shows
\begin{align*}
	\skp{n_{\Gamma^*}}{\d_{n_{\d D^*}} n_{D^*}} = -\sin(\alpha) II_{D^*}(n_{\d D^*},n_{\d D^*}).
\end{align*}
In combination this leads to
\begin{align}\label{eq:LinearAnglePart2}
	n_{\Gamma^*}(q) \cdot \left.\frac{d}{d\epsilon} (3)\right|_{\epsilon=0} = -II_{D^*}(n_{\d D^*}(q),n_{\d D^*}(q)) \rho(t,q).
\end{align}
Coupling (\ref{eq:LinearAnglePart1})-(\ref{eq:LinearAnglePart2}) as in (\ref{eq:LinearAngleProd}) we finally arrive at
\begin{align*}
	& \left.\frac{d}{d\epsilon} \skp{n_\Gamma(t,\Psi(q,\epsilon \rho(t,q)))}{n_D(\Psi(q,\epsilon \rho(t,q)))}\right|_{\epsilon=0} = -\sin(\alpha(q)) (\nabla_{\Gamma^*} \rho(t,q) \cdot n_{\d \Gamma^*}(q)) \\
	& + \cos(\alpha(q)) II_{\Gamma^*}(n_{\d \Gamma^*}(q),n_{\d \Gamma^*}(q)) \rho(t,q) - II_{D^*}(n_{\d D^*}(q),n_{\d D^*}(q)) \rho(t,q),
\end{align*}
which is the desired statement.
\end{proof}

\begin{lemma}\label{lem:LinearGeodesicCurvature}
For the linearization of the geodesic curvature we have
\begin{align*}
	\left.\frac{d}{d\epsilon} \kappa_{\d D}(\Psi(q,\epsilon \rho(t,q)))\right|_{\epsilon=0}
		& = \rho_{\sigma\sigma}(t,q) + \kappa_{D^*}(q) II_{D^*}(n_{\d D^*}(q),n_{\d D^*}(q)) \rho(t,q) \\
		& - \kappa_{\d D^*}(q) \skp{\vec{\tau}^*(q)}{(n_{\d D^*})_\sigma(q)} \rho(t,q) \\
		& - \skp{n_{\d D^*}(q)}{(n_{D^*})_\sigma(q)}^2 \rho(t,q),
\end{align*}
where $II_{D^*}$ is the second fundamental form of $D^*$ with respect to the normal $n_{D^*}$, $\kappa_{D^*}$ is the normal curvature of $\d \Gamma^*$ in $D^*$ defined as $\kappa_{D^*}(q) := \skp{\vec{\kappa}^*(q)}{n_{D^*}(q)}$ and $\sigma$ denotes the arc-length-parameter of $\d \Gamma^* = \d D^*$.
\end{lemma}
\begin{proof}
First we denote by
\begin{align*}
	c: [0,1] \longrightarrow \d \Gamma^*: s \longmapsto c(s) \quad \text{ with } \quad c(0) = c(1)
\end{align*}
a parametrization of $\d \Gamma^*$ and with
\begin{align*}
	&\tilde{c}: [0,\epsilon_0) \times [0,1] \longrightarrow \R^3: (\epsilon,s) \longmapsto \tilde{c}(\epsilon,s) \\
	&\text{with} \quad \tilde{c}(\epsilon,0) = \tilde{c}(\epsilon,1) \quad \text{and} \quad \tilde{c}(0,s) = c(s) \\
	&\text{and} \quad \tilde{c}_\epsilon(0,s) = \rho(t,c(s)) n_{\d D^*}(c(s)) =: \zeta(s) \quad \forall \ s \in [0,1] 
\end{align*}
we denote a parametrization of $\d \tilde{\Gamma}(\epsilon)$, where $\tilde{\Gamma}(\epsilon)$ is the evolving hypersurface from the proof of Lemma \ref{lem:LinearHMean}. Choosing the orientation of $\tilde{c}$ appropriately, we have
\begin{align*}
	\frac{\tilde{c}_s(\epsilon,s)}{\norm{\tilde{c}_s(\epsilon,s)}} = \vec{\tau}(\epsilon,c(s))
	\qquad \forall \ \epsilon \in [0,\epsilon_0) \ \forall \ s \in [0,1]
\end{align*}
and particularly
\begin{align*}
	\frac{c_s(s)}{\norm{c_s(s)}} = \vec{\tau}^*(c(s)) \qquad \forall \ s \in [0,1].
\end{align*}
Therefore we see
\begin{align*}
	\frac{1}{\norm{\tilde{c}_s(\epsilon,s)}} \left(\frac{\tilde{c}_s(\epsilon,s)}{\norm{\tilde{c}_s(\epsilon,s)}}\right)_s
		= \kappa_{\d \tilde{D}(\epsilon)}(\epsilon,c(s)) n_{\d \tilde{D}(\epsilon)}(\epsilon,c(s))
		+ \kappa_{\tilde{D}(\epsilon)}(\epsilon,c(s)) n_{\tilde{D}(\epsilon)}(\epsilon,c(s)).
\end{align*}
The weak formulation of this reads as
\begin{align}\label{eq:WeakFormulation}
	0 = \int_{\d \tilde{\Gamma}(\epsilon)} \frac{\tilde{c}_s}{\norm{\tilde{c}_s}} \cdot \frac{\vec{\eta}_s}{\norm{\tilde{c}_s}} \dH^1 + \int_{\d \tilde{\Gamma}(\epsilon)} \left(\kappa_{\d \tilde{D}(\epsilon)} n_{\d \tilde{D}(\epsilon)} + \kappa_{\tilde{D}(\epsilon)} n_{\tilde{D}(\epsilon)}\right) \cdot \vec{\eta} \dH^1.
\end{align}
Before we can differentiate the equation with respect to $\epsilon$ we need some auxiliary equations, which are all obtained by simple calculations:
\begin{enumerate}
	\item Let $a(\epsilon,s)$ be a quantity smoothly depending on $\epsilon$ and $s$. Then
			\begin{align*}
				& \left.\frac{d}{d\epsilon} \int_{\d \tilde{\Gamma}(\epsilon)} a(\epsilon,s(p)) \dH^1\right|_{\epsilon=0} \\
				& \qquad = \int_{\d \Gamma^*} a_\epsilon(0,s(q)) \dH^1
				  + \int_{\d \Gamma^*} a(0,s(q)) \left(\underbrace{\frac{c_s(s(q))}{\norm{c_s(s(q))}}}_{= \vec{\tau}^*(c(s(q)))} \cdot \frac{\zeta_s(s(q))}{\norm{c_s(s(q))}}\right) \dH^1.
			\end{align*}
			where $s(q)$ is the abbreviation for $s(q) := c^{-1}(q)$ with $q \in \d \Gamma^*$.
	\item If we denote with $\tilde{P}$ the projection onto the $n_{\d D^*}$-$n_{D^*}$-plane, we get by Lemma \ref{lem:DiffNormedVector}(ii)
			\begin{align*}
				\left.\frac{d}{d\epsilon} \frac{\tilde{c}_s(\epsilon,s)}{\norm{\tilde{c}_s(\epsilon,s)}}\right|_{\epsilon=0}
					= P_{\frac{c_s(s)}{\norm{c_s(s)}}}\left(\frac{\tilde{c}_{s\epsilon}(0,s)}{\norm{c_s(s)}}\right)
					= P_{\vec{\tau}^*(c(s))}\left(\frac{\zeta_s(s)}{\norm{c_s(s)}}\right)
					= \tilde{P}\left(\frac{\zeta_s(s)}{\norm{c_s(s)}}\right).
			\end{align*}
	\item Using the second auxiliary calculation we obtain
			\begin{align*}
				\left.\frac{d}{d\epsilon} \int_{\d \tilde{\Gamma}(\epsilon)} \frac{\tilde{c}_s(\epsilon,s(p))}{\norm{\tilde{c}_s(\epsilon,s(p))}} \cdot \frac{\vec{\eta}_s(s(p))}{\norm{\tilde{c}_s(\epsilon,s(p))}} \dH^1\right|_{\epsilon=0} = \int_{\d \Gamma^*} \tilde{P}\left(\frac{\zeta_s(s(q))}{\norm{c_s(s(q))}}\right) \cdot \frac{\vec{\eta}_s(s(q))}{\norm{c_s(s(q))}} \dH^1.
			\end{align*}
\end{enumerate}
Using these auxiliary equations we can differentiate (\ref{eq:WeakFormulation}) with respect to $\epsilon$ and derive
\begin{align*}
	0 & = \int_{\d \Gamma^*} ((\kappa_{\d D^*})_\epsilon n_{\d D^*} + \kappa_{\d D^*} (n_{\d D^*})_\epsilon + (\kappa_{D^*})_\epsilon n_{D^*} + \kappa_{D^*} (n_{D^*})_\epsilon) \cdot \vec{\eta} \dH^1 \\
	  & + \int_{\d \Gamma^*} (\kappa_{\d D^*} n_{\d D^*} + \kappa_{D^*} n_{D^*}) \cdot \vec{\eta} \left(\vec{\tau}^* \cdot \frac{\zeta_s}{\norm{c_s}}\right) \dH^1
		 + \int_{\d \Gamma^*} \tilde{P}\left(\frac{\zeta_s}{\norm{c_s}}\right) \cdot \frac{\vec{\eta}_s}{\norm{c_s}} \dH^1
\end{align*}
for $\epsilon = 0$. Choosing $\vec{\eta}(s) := \xi(s) n_{\d D^*}(c(s))$ with an arbitrary function $\xi: [0,1] \longrightarrow \R$ we get
\begin{align*}
	0 & = \int_{\d \Gamma^*} (\kappa_{\d D^*})_\epsilon \xi + \kappa_{D^*} \xi ((n_{D^*})_\epsilon \cdot n_{\d D^*}) \dH^1 \\
	  & + \int_{\d \Gamma^*} \kappa_{\d D^*} \xi \left(\vec{\tau}^* \cdot \frac{\zeta_s}{\norm{c_s}}\right) \dH^1
	    + \int_{\d \Gamma^*} \tilde{P}\left(\frac{\zeta_s}{\norm{c_s}}\right) \cdot \frac{(\xi n_{\d D^*})_s}{\norm{c_s}} \dH^1.
\end{align*}
Integration by parts yields the rewritten equation
\begin{align}\label{eq:Rewritten}
	\int_{\d \Gamma^*} (\kappa_{\d D^*})_\epsilon \xi \dH^1
		& = \int_{\d \Gamma^*} \left(\tilde{P}\left(\frac{\zeta_s}{\norm{c_s}}\right)\right)_s \cdot \frac{\xi n_{\d D^*}}{\norm{c_s}} \dH^1 \notag \\
		& - \int_{\d \Gamma^*} \kappa_{D^*} \xi ((n_{D^*})_\epsilon \cdot n_{\d D^*}) + \kappa_{\d D^*} \xi \left(\vec{\tau}^* \cdot \frac{\zeta_s}{\norm{c_s}}\right) \dH^1.
\end{align}
We will now investigate the first integrand. Starting with
\begin{align}\label{eq:ZetaS}
	\zeta_s = \rho_s n_{\d D^*} + \rho (n_{\d D^*})_s
\end{align}
we project onto the $n_{\d D^*}$-$n_{D^*}$-plane and differentiate with respect to $s$, which gives
\begin{align*}
	\left(\tilde{P}\left(\frac{\zeta_s}{\norm{c_s}}\right)\right)_s
		& = \left(\frac{\rho_s}{\norm{c_s}}\right)_s n_{\d D^*} + \frac{\rho_s}{\norm{c_s}} (n_{\d D^*})_s
		  + \left(\frac{1}{\norm{c_s}} \rho \skp{(n_{\d D^*})_s}{n_{D^*}}\right)_s n_{D^*} \\
		& + \frac{1}{\norm{c_s}} \rho \skp{(n_{\d D^*})_s}{n_{D^*}} (n_{D^*})_s.
\end{align*}
This and the fact $\skp{(n_{\d D^*})_s}{n_{D^*}} = -\skp{n_{\d D^*}}{(n_{D^*})_s}$ finally lead to
\begin{align*}
	\left(\tilde{P}\left(\frac{\zeta_s}{\norm{c_s}}\right)\right)_s \cdot \frac{\xi n_{\d D^*}}{\norm{c_s}} = \rho_{\sigma\sigma} \xi - \rho \skp{(n_{D^*})_\sigma}{n_{\d D^*}}^2 \xi.
\end{align*}
In addition, (\ref{eq:ZetaS}) shows that the third integrand in (\ref{eq:Rewritten}) reads as
\begin{align*}
	\left(\vec{\tau}^* \cdot \frac{\zeta_s}{\norm{c_s}}\right) = \rho_\sigma \underbrace{(\vec{\tau}^* \cdot n_{\d D^*})}_{= 0} + \frac{\rho}{\norm{c_s}} \skp{\vec{\tau}^*}{(n_{\d D^*})_s} = \skp{\vec{\tau}^*}{(n_{\d D^*})_\sigma} \rho.
\end{align*}
Inserting these two facts into (\ref{eq:Rewritten}) we obtain
\begin{align*}
	\int_{\d \Gamma^*} (\kappa_{\d D^*})_\epsilon \xi \dH^1
		& = \int_{\d \Gamma^*} \rho_{\sigma\sigma} \xi - \rho \skp{(n_{D^*})_\sigma}{n_{\d D^*}}^2 \xi \dH^1 \\
		& - \int_{\d \Gamma^*} \kappa_{D^*} \skp{(n_{D^*})_\epsilon}{n_{\d D^*}} \xi + \kappa_{\d D^*} \rho \skp{\vec{\tau}^*}{(n_{\d D^*})_\sigma} \xi \dH^1.
\end{align*}
Since $\xi$ was chosen arbitrarily, we can again apply the fundamental lemma of the calculus of variation to end up with
\begin{align*}
	(\kappa_{\d D^*})_\epsilon = \rho_{\sigma\sigma} - \skp{(n_{D^*})_\sigma}{n_{\d D^*}}^2 \rho - \kappa_{D^*} \skp{(n_{D^*})_\epsilon}{n_{\d D^*}} - \kappa_{\d D^*} \skp{\vec{\tau}^*}{(n_{\d D^*})_\sigma} \rho.
\end{align*}
We want to have an equation where $\rho$ is contained in every single term, therefore we rewrite the third term as
\begin{align*}
	(n_{D^*})_\epsilon = \d_{\rho n_{\d D^*}} n_{D^*} = \rho \d_{n_{\d D^*}} n_{D^*},
\end{align*}
which leads to
\begin{align*}
	\skp{(n_{D^*})_\epsilon}{n_{\d D^*}} = \skp{\d_{n_{\d D^*}} n_{D^*}}{n_{\d D^*}} \rho = -II_{D^*}(n_{\d D^*},n_{\d D^*}) \rho.
\end{align*}
Finally, we have the desired expression
\begin{align*}
	(\kappa_{\d D^*})_\epsilon = \rho_{\sigma\sigma} - \skp{(n_{D^*})_\sigma}{n_{\d D^*}}^2 \rho + \kappa_{D^*} II_{D^*}(n_{\d D^*},n_{\d D^*}) \rho - \kappa_{\d D^*} \skp{\vec{\tau}^*}{(n_{\d D^*})_\sigma} \rho.
\end{align*}
\end{proof}

Combining the results from Lemma \ref{lem:LinearVGamma} to Lemma \ref{lem:LinearGeodesicCurvature} the linearization of (\ref{eq:Flow1})-(\ref{eq:Flow2}) is given by
\begin{align}\label{eq:LinearFlow1}
	\d_t \rho(t) & = \Delta_{\Gamma^*} \rho(t) + |\sigma^*|^2 \rho(t)
                  + \left(\nabla_{\Gamma^*} H_{\Gamma^*} \cdot P\left(\d_w \Psi(0)\right)\right) \rho(t) \notag \\
                & - \mint_{\Gamma^*} (\Delta_{\Gamma^*} + |\sigma^*|^2 - H_{\Gamma^*}^2 + \bar{H}(\mathbb{O}) H_{\Gamma^*}) \rho(t) \dH^2 \notag \\
                & + \frac{1}{\int_{\Gamma^*}\limits 1 \dH^2} \int_{\d \Gamma^*} \left(H_{\Gamma^*} - \bar{H}(\mathbb{O})\right) \cot(\alpha) \rho(t) \dH^1 + f(t) \hspace*{20mm} \text{ in } [0,T] \times \Gamma^*, \\ \label{eq:LinearFlow2}
	\d_t \rho(t) & = -\sin(\alpha)^2 (n_{\d \Gamma^*} \cdot \nabla_{\Gamma^*} \rho(t))
                  - \sin(\alpha) II_{D^*}(n_{\d D^*},n_{\d D^*}) \rho(t) \notag \\
                & + \sin(\alpha) \cos(\alpha) II_{\Gamma^*}(n_{\d \Gamma^*},n_{\d \Gamma^*}) \rho(t) + b \sin(\alpha) \rho_{\sigma\sigma}(t) \notag \\
                & + b \sin(\alpha) \kappa_{D^*} II_{D^*}(n_{\d D^*},n_{\d D^*}) \rho(t) - b \sin(\alpha) \kappa_{\d D^*} \skp{\vec{\tau}^*}{(n_{\d D^*})_\sigma} \rho(t) \notag \\
                & - b \sin(\alpha) \skp{n_{\d D^*}}{(n_{D^*})_\sigma}^2 \rho(t) + g_0(t) \hspace*{44.5mm} \text{ on } [0,T] \times \d \Gamma^*, \\ \label{eq:LinearFlow3}
	\rho(0) & = \rho_0 \hspace*{83.5mm} \text{ in } \Gamma^*,
\end{align}
where we have dropped the argument $q$ for a more convenient notation
and introduced $f$ and $g_0$ to compensate for the difference between
the linearized and the nonlinear operator as well as an initial condition. This linearization will be the starting point for the short-time existence of solutions of the MCF in the section to follow.

\section{Local existence of solutions of the volume-preserving MCF with line tension}\label{sec:LocalExistenceMCF}

\subsection{Short-time existence of solutions for the linearized volume-preserving mean curvature flow}\label{ssec:ShortTimeExLinearMCF}

In this section we will show that the flow (\ref{eq:Flow1})-(\ref{eq:Flow2}) has a unique strong solution for short times. Our goal will be achieved by first considering solutions of the linearized flow (\ref{eq:LinearFlow1})-(\ref{eq:LinearFlow3}) and then apply Banach's fixed point argument to transfer the short-time existence to the non-linear flow.

In a first step we want to show that for fixed $T > 0$ the flow
\begin{alignat}{2}\label{eq:LocalFlow1}
	\d_t \rho(t) & = \Delta_{\Gamma^*} \rho(t) + |\sigma^*|^2 \rho(t)
		  + \left(\nabla_{\Gamma^*} H_{\Gamma^*} \cdot P\left(\d_w \Psi(0)\right)\right) \rho(t) + f(t)
			& \ & \text{ in } [0,T] \times \Gamma^*, \\ \label{eq:LocalFlow2}
	\d_t \rho(t) & = -\sin(\alpha)^2 (n_{\d \Gamma^*} \cdot \nabla_{\Gamma^*} \rho(t))
		  - \sin(\alpha) II_{D^*}(n_{\d D^*},n_{\d D^*}) \rho(t) \notag \\
		& + \sin(\alpha) \cos(\alpha) II_{\Gamma^*}(n_{\d \Gamma^*},n_{\d \Gamma^*}) \rho(t)
		  + b \sin(\alpha) \rho_{\sigma\sigma}(t) \notag \\
		& + b \sin(\alpha) \left(\kappa_{D^*} II_{D^*}(n_{\d D^*},n_{\d D^*}) - \kappa_{\d D^*} \skp{\vec{\tau}^*}{(n_{\d D^*})_\sigma}\right) \rho(t) + g_0(t) \notag \\
		& - b \sin(\alpha) \skp{n_{\d D^*}}{(n_{D^*})_\sigma}^2 \rho(t) \hspace*{49mm}&&  \text{ on } [0,T] \times \d \Gamma^*, \\ \label{eq:LocalFlow3}
	\rho(0) & = \rho_0 \hspace*{88mm} &&\text{ in } \Gamma^*,
\end{alignat}
which is (\ref{eq:LinearFlow1})-(\ref{eq:LinearFlow3}) without the non-local part, has a unique solution. To derive such a result we will use results of \cite{DPZ08}.

As a starting point we want to formulate the linearization problem in the notation of \cite{DPZ08} such that (\ref{eq:LocalFlow1})-(\ref{eq:LocalFlow3}) is transformed into a problem of the form
\begin{align*}
	\frac{d}{dt} u(t) + \mathcal{A}(D) u(t) &= f(t) & &\text{ in } J \times \Gamma^*, \\
	\frac{d}{dt} \varrho(t) + \mathcal{B}_0(D) u(t) + \mathcal{C}_0(D_\d) \varrho(t) &= g_0(t) & &\text{ on } J \times \d \Gamma^*, \\
	\mathcal{B}_1(D) u(t) + \mathcal{C}_1(D_\d) \varrho(t) &= g_1(t) & &\text{ on } J \times \d \Gamma^*, \\
	u(0) &= u_0 & &\text{ in } \Gamma^*, \\
	\varrho(0) &= \varrho_0 & &\text{ on } \d \Gamma^*.
\end{align*}
Although the authors only consider domains in $\R^n$ the results carry over to smooth manifolds. We would have to use a partition of unity and local coordinates several times, but for the sake of notation we skip these technicalities.

Dropping the argument $q$ the operators and functions in our case read as
\begin{align*}
	\mathcal{A}(D)   &:= -\Delta_{\Gamma^*} - |\sigma^*|^2 - (\nabla_{\Gamma^*} H_{\Gamma^*} \cdot P(\d_w \Psi(0))) & & \\
	\mathcal{B}_0(D) &:= \sin(\alpha)^2 (n_{\d \Gamma^*} \cdot \nabla_{\Gamma^*}) + \sin(\alpha) II_{D^*}(n_{\d D^*},n_{\d D^*}) & & \\
						  & - \sin(\alpha) \cos(\alpha) II_{\Gamma^*}(n_{\d \Gamma^*},n_{\d \Gamma^*}) & \mathcal{B}_1(D) &:= 1 \\
	\mathcal{C}_0(D_\d) &:= -b \sin(\alpha) \d_\sigma^2 - b \sin(\alpha) \kappa_{D^*} II_{D^*}(n_{\d D^*},n_{\d D^*}) & & \\
						  & + b \sin(\alpha) \kappa_{\d D^*} \skp{\vec{\tau}^*}{(n_{\d D^*})_\sigma} + b \sin(\alpha) \skp{n_{\d D^*}}{(n_{D^*})_\sigma}^2 & \mathcal{C}_1(D_\d) &:= -1 \\
	u(t)          &:= \rho(t) & \varrho(t) &:= \rho(t)|_{\d \Gamma^*}.
\end{align*}

We note that the required condition ``all $\mathcal{B}_j$ and at least one $\mathcal{C}_j$ are non-trivial'' is satisfied. Moreover, in our case we have $E := F := \R$, which are obviously of type $\mathcal{HT}$ since the Hilbert-transform is continuous on $L_2(\R;\R)$. The interval we want to consider is $[0,T]$ denoted by $J$ as in \cite{DPZ08}. Because of
\begin{align*}
	l & := \max\{\ord(\mathcal{C}_0) - \ord(\mathcal{B}_0) + \ord(\mathcal{B}_0), \ord(\mathcal{C}_1) - \ord(\mathcal{B}_1) + \ord(\mathcal{B}_0)\} = 2, \\
	2m & := \ord(\mathcal{A}) = 2
\end{align*}
we have to consider the setting that is called ``case 1'' in \cite{DPZ08}. In our situation the required function spaces simplify to
\begin{align}\label{eq:Spaces}
	X					 & := L_p(J;L_p(\Gamma^*;\R)), \notag \\
	Z_u				 & := W^1_p(J; L_p(\Gamma^*;\R)) \cap L_p(J; W^2_p(\Gamma^*;\R)), \notag \\
	\pi Z_u			 & := W^{2-\frac{2}{p}}_p(\Gamma^*;\R), \notag \\
	Y_0				 & := W^{\frac{1}{2}-\frac{1}{2p}}_p(J; L_p(\d \Gamma^*;\R))
								\cap L_p(J; W^{1-\frac{1}{p}}_p(\d \Gamma^*;\R)), \notag \\
	Y_1				 & := W^{1-\frac{1}{2p}}_p(J; L_p(\d \Gamma^*;\R))
								\cap L_p(J; W^{2-\frac{1}{p}}_p(\d \Gamma^*;\R)), \notag \\
	Z_\varrho		 & := W^{\frac{3}{2}-\frac{1}{2p}}_p(J; L_p(\d \Gamma^*;\R))
								\cap L_p(J; W^{3-\frac{1}{p}}_p(\d \Gamma^*;\R)), \notag \\
	\pi Z_\varrho	 & := W^{3-\frac{3}{p}}_p(\d \Gamma^*;\R), \notag \\
	\pi_1 Z_\varrho & := W^{1-\frac{3}{p}}_p(\d \Gamma^*;\R),\nomenclature[3.2]{$L_p(M)$}{$L_p$-space on $M$}\nomenclature[3.3]{$W^s_p(M)$}{Sobolev-Slobodeckij space of order $s \in \R_+$ on $M$}
\end{align}
where we have to assure $\frac{2}{p} \notin \N$ and $\kappa_0 := 1 - \frac{\ord(\mathcal{B}_0)}{2m} + \frac{1}{2mp} = \frac{1}{2} - \frac{1}{2p} > \frac{1}{p}$ for the trace spaces. Therefore we assume $p > 3$. As the principle parts of the operators we obtain
\begin{align*}
	\mathcal{A}^\sharp(-i\nabla_{\Gamma^*}) &= -\Delta_{\Gamma^*} = (-i\nabla_{\Gamma^*}) \cdot (-i\nabla_{\Gamma^*}), & & \\
	\mathcal{B}_0^\sharp(-i\nabla_{\Gamma^*}) &= i \sin(\alpha(q))^2 (n_{\d \Gamma^*}(q) \cdot (-i\nabla_{\Gamma^*})), & \mathcal{B}_1^\sharp(-i\nabla_{\Gamma^*}) &= 1 ,\\
	\mathcal{C}_0^\sharp(-i\d_\sigma) &= b \sin(\alpha(q)) (-i\d_\sigma)^2, & \mathcal{C}_1^\sharp(t,q,-i\d_\sigma) &= -1.
\end{align*}
To apply the theorems of \cite{DPZ08} we have to check the respective assumptions. We remark that we can ignore the assumptions (LS$^-_\infty$) and (LS$^+_\infty$) due to the case $l = 2m$ and furthermore we can also ignore assumptions (SD), (SB) and (SC), since we have assumed all involved surfaces and operators to be smooth enough. Now we only have to revise the ellipticity assumption (E) and the Lopatinskii-Shapiro condition (LS).

To prove that condition (E) is satisfied let $t \in J$, $q \in \Gamma^*$ and $\xi \in \R^2$ with $\norm{\xi} = 1$. Then we see
\begin{align*}
	\sigma(\mathcal{A}^\sharp(\xi))
		& = \left\{\lambda \in \C \left| \ \lambda - \mathcal{A}^\sharp(\xi) = \lambda - \norm{\xi}^2 = 0\right.\right\} = \{1\} \subseteq \C_+ := \{\lambda \in \C \mid \Re(\lambda) > 0\}.
\end{align*}
In order to check condition (LS), the finite dimension of $E = F = \R$ allows us to prove the equivalent condition that the desired ODE given by
\begin{align*}
	(\lambda + \mathcal{A}^\sharp(\hat{\xi},-i\d_y)) v(y) & = 0, \qquad y>0, \\
	\mathcal{B}_0^\sharp(\hat{\xi},-i\d_y) v(0) + (\lambda + \mathcal{C}_0^\sharp(\hat{\xi})) \sigma & = 0, \\
	\mathcal{B}_1^\sharp(\hat{\xi},-i\d_y) v(0) + \mathcal{C}_1^\sharp(\hat{\xi}) \sigma & = 0
\end{align*}
has only the trivial solution $(v,\sigma) = (0,0)$ in
\begin{align*}
	C_0(\R_+;\R) \times \R := \left\{v: [0,\infty) \longrightarrow \R \left| \ v \text{ is continuous and } \lim_{y \rightarrow \infty} v(y) = 0\right.\right\} \times \R.
\end{align*}
Thus let $\hat{\xi} \in \R$, $\lambda \in \bar{\C_+}$ with $|\hat{\xi}| + |\lambda| \neq 0$. Then the ODE to be considered is
\begin{align*}
	&\text{(I)} & \lambda v(y) + \hat{\xi}^2 v(y) - v''(y) &= 0,\quad y>0, \\
	&\text{(II)} & i \sin(\alpha(q))^2 \left(n_{\d \Gamma^*}(q) \cdot (\hat{\xi} v(0), -i v'(0))^T\right) + \lambda \sigma + b \sin(\alpha(q)) \hat{\xi}^2 \sigma &= 0, \\
	&\text{(III)} & v(0) - \sigma &= 0.
\end{align*}
Equation (I) immediately shows that $v(y) = c_1 e^{\mu y} + c_2 e^{-\mu y}$ with $\mu := \sqrt{\lambda + \hat{\xi}^2} \neq 0$. Since $\mu$ and $-\mu$ appear in $v$ we can w.l.o.g. choose for $\mu$ the complex square root that satisfies $\Re(\mu) > 0$. There is no chance that $\Re(\mu) = 0$ due to the choice of $\lambda$ and $\hat{\xi}$. Since we require $v \in C_0(\R_+;\R)$, one can see that $c_1 = 0$. Now (III) shows $c_2 = v(0) = \sigma$. As demanded in condition (LS) we identify the positive part of the last coordinate axis with the inner normal to $\d \Gamma^*$, i.e., $-n_{\d \Gamma^*} \entspr \left(\begin{smallmatrix} 0 \\ 1 \end{smallmatrix}\right)$. With this identification (II) reads as
\begin{align*}
	0 & = i \sin(\alpha(q))^2 \begin{pmatrix} 0 \\ -1 \end{pmatrix} \cdot \begin{pmatrix} \hat{\xi} v(0) \\ -i v'(0) \end{pmatrix} + \lambda \sigma + b \sin(\alpha(q)) \hat{\xi}^2 \sigma \\
	  & = -\sin(\alpha(q))^2 v'(0) + (\lambda + b \sin(\alpha(q)) \hat{\xi}^2) \sigma.
\end{align*}
This shows that we have
\begin{align*}
	(\lambda + b \sin(\alpha(q)) \hat{\xi}^2) \sigma = \sin(\alpha(q))^2 v'(0) = -\mu \sigma \sin(\alpha(q))^2 e^{-\mu 0} = -\sigma \sin(\alpha(q))^2 \sqrt{\lambda + \hat{\xi}^2},
\end{align*}
which is either the case for $\sigma = 0$ or if
\begin{align}\label{eq:2ndOption}
	\lambda + b \sin(\alpha(q)) \hat{\xi}^2 = -\sin(\alpha(q))^2 \sqrt{\lambda + \hat{\xi}^2}.
\end{align}
This is not possible since $\Re(\lambda + b \sin(\alpha(q)) \hat{\xi}^2) \geq 0$ whereas $-\sin(\alpha(q))^2 \sqrt{\lambda + \hat{\xi}^2}$ has negative real part due to $\Re(\mu) > 0$. This leads to $\sigma = 0$ and hence to $v \equiv 0$, which is the desired condition (LS).

Upon having proved all assumptions of \cite{DPZ08}, we can state our first theorem.

\begin{thm}\label{thm:LocalExistenceGeneral}
Let $3 < p < \infty$, $J := [0,T]$ and the spaces be defined as in (\ref{eq:Spaces}). Then the problem
\begin{align}\label{eq:LocalLinearMCF1}
	\frac{d}{dt} u(t) + \mathcal{A}(D) u(t) &= f(t) & &\text{ in } J \times \Gamma^*, \\ \label{eq:LocalLinearMCF2}
	\frac{d}{dt} \varrho(t) + \mathcal{B}_0(D) u(t) + \mathcal{C}_0(D_\d) \varrho(t) &= g_0(t) & &\text{ on } J \times \d \Gamma^*, \\ \label{eq:LocalLinearMCF3}
	\mathcal{B}_1(D) u(t) + \mathcal{C}_1(D_\d) \varrho(t) &= g_1(t) & &\text{ on } J \times \d \Gamma^*, \\ \label{eq:LocalLinearMCF4}
	u(0) &= u_0 & &\text{ in } \Gamma^*, \\ \label{eq:LocalLinearMCF5}
	\varrho(0) &= \varrho_0 & &\text{ on } \d \Gamma^*
\end{align}
has a unique solution $(u,\varrho) \in Z_u \times Z_\varrho$ if and only if
\begin{align*}
	f \in X, \qquad u_0 \in \pi Z_u, \qquad \varrho_0 \in \pi Z_\varrho, \qquad g_0 \in Y_0, \qquad g_1 \in Y_1, \\
	g_0(0) - \mathcal{B}_0(D) u_0 - \mathcal{C}_0(D_\d) \varrho_0 \in \pi_1 Z_\varrho, \qquad \mathcal{B}_1(D) u_0 + \mathcal{C}_1(D_\d) \varrho_0 = g_1(0).
\end{align*}
\end{thm}
\begin{proof}
Follows from Theorem 2.1 in \cite{DPZ08} applied to our specific case.
\end{proof}

From this theorem we deduce a corollary that gives us the existence and uniqueness of solutions of the flow (\ref{eq:LocalFlow1})-(\ref{eq:LocalFlow3}) on each finite interval.

\begin{cor}\label{cor:LocalExistence1}
Let $3 < p < \infty$, $J := [0,T]$ and the spaces be defined as in (\ref{eq:Spaces}). Then (\ref{eq:LocalFlow1})-(\ref{eq:LocalFlow3}) has a unique solution $\rho \in Z_u$ with $\rho|_{\d \Gamma^*} \in Z_\varrho$ if and only if $f \in X$, $g_0 \in Y_0$, $\rho_0 \in \pi Z_u$ and $\rho_0|_{\d \Gamma^*} \in \pi Z_\varrho$.
\end{cor}
\begin{proof}
Follows from Theorem \ref{thm:LocalExistenceGeneral} with $g_1 \equiv 0$. Then the condition $\mathcal{B}_1(0) u_0 + \mathcal{C}_1(0) \varrho_0 = g_1(0)$ is valid since $u_0|_{\d \Gamma^*} = \rho_0|_{\d \Gamma^*} = \varrho_0$. Moreover,
\begin{align*}
	g_0(0) - \mathcal{B}_0(-i\nabla_{\Gamma^*})\rho_0 - \mathcal{C}_0(-i \d_\sigma) \rho_0|_{\d \Gamma^*} \in \pi_1 Z_\varrho = W^{1-\frac{3}{p}}_p(\d \Gamma^*;\R)
\end{align*}
can be ignored since on the one hand $\rho_0|_{\d \Gamma^*} \in \pi Z_\varrho = W^{3-\frac{3}{p}}_p(\d \Gamma^*;\R)$ and $\mathcal{C}_0$ is of second order and on the other hand $\rho_0 \in \pi Z_u = W^{2-\frac{2}{p}}_p(\Gamma^*;\R)$, $\mathcal{B}_0$ is of first order and the trace operator $\gamma_0$ maps from $W^{1-\frac{2}{p}}_p(\Gamma^*;\R)$ to $W^{1-\frac{3}{p}}_p(\d \Gamma^*;\R)$ for $p > 3$. In addition, $g_0(0) \in \pi_1 Z_\varrho$, since the trace operator $\gamma_0$ maps from $Y_0$ to $\pi_1 Z_\varrho$ as one can see from (A.25) in \cite{Gru95}.
\end{proof}

Now we want to move on to the more important considerations about the non-local part, which we ignored in (\ref{eq:LocalFlow1})-(\ref{eq:LocalFlow3}), but has to be included for the flow (\ref{eq:LinearFlow1})-(\ref{eq:LinearFlow3}). The basic ingredient will be a perturbation result of semigroup theory and the time-independence of the operators $\mathcal{A}$, $\mathcal{B}_0$, $\mathcal{B}_1$, $\mathcal{C}_0$ and $\mathcal{C}_1$.

For a second theorem from \cite{DPZ08} we have to define the operator
\begin{align*}
	A: \mathcal{D}(A) \longrightarrow \mathcal{W}(A): \begin{pmatrix}
																		  \rho \\
																		  \tilde{\rho}
																	  \end{pmatrix} \longmapsto
	\begin{pmatrix}
		\mathcal{A}(-i\nabla_{\Gamma^*}) & \mathbb{O} \\
		\mathcal{B}_0(-i\nabla_{\Gamma^*}) & \mathcal{C}_0(-i\d_\sigma)
	\end{pmatrix} \begin{pmatrix}
						  \rho \\
						  \tilde{\rho}
					  \end{pmatrix},
\end{align*}
where the domain and codomain are
\begin{align*}
	\mathcal{D}(A) & := \left\{\left.(\rho,\tilde{\rho})^T \in W^2_p(\Gamma^*;\R) \times W^{3-\frac{1}{p}}_p(\d \Gamma^*;\R)
									 \right| \rho|_{\d \Gamma^*} = \tilde{\rho}\right\} \\
	\mathcal{W}(A) & := L_p(\Gamma^*;\R) \times W^{1-\frac{1}{p}}_p(\d \Gamma^*;\R).
\end{align*}

\begin{rem}\label{rem:DropConditionMCF}
Note that the condition $\mathcal{B}_0(-i\nabla_{\Gamma^*}) \rho + \mathcal{C}_0(-i \d_\sigma) \tilde{\rho} \in W^{1-\frac{1}{p}}_p(\d \Gamma^*;\R)$ from the original domain in \cite{DPZ08} is automatically satisfied by the same arguments as in the proof of Corollary \ref{cor:LocalExistence1}.
\end{rem}

For this new operator $A$ we get the following statement from \cite{DPZ08}.

\begin{thm}\label{thm:Semigroup1}
Let $3 < p < \infty$. Then the operator $-A$ generates an analytic semigroup in $\mathcal{W}(A)$, which has the property of maximal $L_p$-regularity on each finite interval $J = [0,T]$. Moreover, there is $\omega \geq 0$ such that $-(A + \omega \id)$ has maximal $L_p$-regularity on the half-line $\R_+$.
\end{thm}
\begin{proof}
Adapt Theorem 2.2 from \cite{DPZ08} to the given situation.
\end{proof}

\begin{rem}
For the same reason as in the proof of Corollary \ref{cor:LocalExistence1} we were allowed to erase the three conditions
\begin{align*}
	& \mathcal{B}_0(-i\nabla_{\Gamma^*}) \rho + \mathcal{C}_0(-i \d_\sigma) \rho|_{\d \Gamma^*} \in L_p(J;W^{1-\frac{1}{p}}_p(\d \Gamma^*;\R)), \\
	& \mathcal{B}_0(-i\nabla_{\Gamma^*}) \rho_0 + \mathcal{C}_0(-i \d_\sigma) \rho_0|_{\d \Gamma^*} \in \pi_1 Z_\varrho, \\
	& \mathcal{B}_1(-i\nabla_{\Gamma^*}) \rho_0 + \mathcal{C}_1(-i \d_\sigma) \rho_0|_{\d \Gamma^*} = g_1(0)
\end{align*}
from the original theorem in \cite{DPZ08}.
\end{rem}

\subsection{Short-time existence of solutions for the volume-preserving mean curvature flow}\label{ssec:ShortTimeExMCF}

Now we want to prove the short-time existence of solutions of the non-linear flow
\begin{align}\label{eq:NonLinearMCF1}
	V_\Gamma(u(t))       &= H_\Gamma(u(t)) - \bar{H}(u(t)) & &\text{in } [0,T] \times \Gamma^*, \\ \label{eq:NonLinearMCF2}
	v_{\d D}(\varrho(t)) &= a + b \kappa_{\d D}(\varrho(t)) + \skp{n_\Gamma(u(t))}{n_D(u(t))} & &\text{on } [0,T] \times \d \Gamma^*, \\ \label{eq:NonLinearMCF3}
	u(t)                 &= \varrho(t) & &\text{on } [0,T] \times \d \Gamma^*, \\ \label{eq:NonLinearMCF4}
	u(0)                 &= u_0 & &\text{in } \Gamma^*, \\ \label{eq:NonLinearMCF5}
	\varrho(0)           &= \varrho_0 & &\text{on } \d \Gamma^*,
\end{align}
where we have adopted the structure of the linearized PDE in Theorem \ref{thm:LocalExistenceGeneral}. We will use the contraction mapping principle on the equation $L \Phi = (N \Phi, \Phi_0)$ to prove the desired short-time existence. To this end we define the functions $\Phi := (u,\varrho)$ and $\Phi_0 := (u_0,\varrho_0)$, the spaces
\begin{align*}
	\mathbb{E} & := Z_u \times Z_\varrho, \qquad \qquad \mathbb{F} := X \times Y_0 \times \{0\}, \\
	\mathbb{I} & := \left\{(u_0,\varrho_0) \in \pi Z_u \times \pi Z_\varrho \mid u_0|_{\d \Gamma^*} = \varrho_0\right\}
\end{align*}
the operator $L: \mathbb{E} \longrightarrow \mathbb{F} \times \mathbb{I}$ as the left-hand side of (\ref{eq:LocalLinearMCF1})-(\ref{eq:LocalLinearMCF5}) and for the right-hand side of the contraction mapping principle we define the non-linear operator $N: \mathbb{E} \longrightarrow \mathbb{F}$ as
\begin{align*}
	N(\Phi) := \begin{pmatrix}
					  H_\Gamma(u) - \bar{H}(u) - V_\Gamma(u) + \frac{d}{dt} u + \mathcal{A}(D) u \\
					  a + b \kappa_{\d D}(\varrho) + \skp{n_\Gamma(u)}{n_D(u)} - v_{\d D}(\varrho) + \frac{d}{dt} \varrho + \mathcal{B}_0(D) u + \mathcal{C}_0(D_\d) \varrho \\
					  0
				  \end{pmatrix}.
\end{align*}
In order to apply the contraction mapping principle we need the following technical lemmas.

\begin{lemma}\label{lem:Embeddings}
Let $1 < p < \infty$ and $J = [0,T]$. Then for all $\sigma \in (0,1)$ we have the following embeddings
\begin{align}\label{eq:CrucialEmbedding}
	Z_u & \hookrightarrow W^\sigma_p(J;W^{2(1 - \sigma)}_p(\Gamma^*;\R)), \notag \\
	Z_\varrho & \hookrightarrow W^{\sigma(\frac{3}{2} - \frac{1}{2p})}_p(J;W^{(1 - \sigma)(3 - \frac{1}{p})}_p(\d \Gamma^*;\R)), \notag \\
	Y_0 & \hookrightarrow W^{\sigma(\frac{1}{2} - \frac{1}{2p})}_p(J;W^{(1 - \sigma)(1 - \frac{1}{p})}_p(\d \Gamma^*;\R)), \notag \\
	Z_u & \hookrightarrow BUC(J;W^{2 - \frac{2}{p}}_p(\Gamma^*;\R)) \hookrightarrow BUC(J;BUC^1(\Gamma^*;\R)),
\end{align}
where $BUC^k$ denotes the space of bounded uniformly continuous functions.
\end{lemma}
\begin{proof}
Note that it is not necessary to distinguish $W^s_p(J;W^r_p(\Gamma^*;\R))$ and $W^s_p(J;W^r_p(K;\R))$. Then the statement follows from Lemma 4.3 and Lemma 4.4. in \cite{DSS08}, the assumption $p > 4$ and usual Sobolev embeddings.
\end{proof}

\begin{rem}
The second embedding in (\ref{eq:CrucialEmbedding}) is only valid for $p > 4$ and will be crucial in the considerations to follow. This is the reason why we are forced to restrict the range of $p$ from $p > 3$ in Theorem \ref{thm:LocalExistenceGeneral} to $p > 4$ in our final Theorem \ref{thm:ShortTimeExistenceMCF} below.
\end{rem}

Later in our most important technical lemma we will deal with quasi-linear differential operators. For this purpose it is helpful to know that the spaces containing the second highest derivatives are Banach algebras (cf. Lemma \ref{lem:AlgebraMCF}).

Since due to Lemma 4.3 of \cite{DSS08} with $\sigma = \frac{1}{2}$, we get $Z_u \hookrightarrow W^{\frac{1}{2}}_p(J;W^1_p(\Gamma^*;\R))$ and observe
\begin{align*}
	\nabla_{\Gamma^*} u \in W^{\frac{1}{2}}_p(J;L_p(\Gamma^*;\R)) \cap L_p(J;W^1_p(\Gamma^*;\R)).
\end{align*}
For $\sigma = \frac{2p - 1}{3p - 1}$, we get $Z_\varrho \hookrightarrow W^{1 - \frac{1}{2p}}_p(J;W^1_p(\d \Gamma^*;\R))$ and arrive at
\begin{align*}
	\d_{\sigma} \varrho \in W^{1 - \frac{1}{2p}}_p(J;L_p(\d \Gamma^*;\R)) \cap L_p(J;W^{2 - \frac{1}{p}}_p(\d \Gamma^*;\R)).
\end{align*}

\begin{lemma}\label{lem:AlgebraMCF}
Let $4 < p < \infty$. Then the spaces
\begin{align*}
	\text{(i)} &\quad & \nabla^1 Z_u &:= W^{\frac{1}{2}}_p(J;L_p(\Gamma^*;\R)) \cap L_p(J;W^1_p(\Gamma^*;\R)) \\
	\text{(ii)} &\quad & \nabla^1 Z_\varrho &:= W^{1 - \frac{1}{2p}}_p(J;L_p(\d \Gamma^*;\R)) \cap L_p(J;W^{2 - \frac{1}{p}}_p(\d \Gamma^*;\R)),
\end{align*}
which contain the first spacial derivatives of $u$ and the first arc-length derivatives of $\varrho$ are Banach algebras up to a constant in the norm estimate of the product.
\end{lemma}
\begin{proof}
(i) First we use Lemma \ref{lem:Embeddings} to obtain
\begin{align*}
	\nabla^1 Z_u \hookrightarrow BUC(J;W^{1 - \frac{2}{p}}_p(\Gamma^*;\R)) \hookrightarrow BUC(J;BUC(\Gamma^*;\R)),
\end{align*}
where we have used $p > 4$ in the second embedding. We use this embedding and the triangle inequality to prove for $f, g \in \nabla^1 Z_u$
\begin{align*}
	\norm{fg}_{\nabla^1 Z_u} \leq \hat{c} \norm{f}_{\nabla^1 Z_u} \norm{g}_{\nabla^1 Z_u}
\end{align*}
in a straight forward manner. \\
(ii) Consider Lemma \ref{lem:Embeddings} to obtain
\begin{align*}
	\nabla^1 Z_\varrho \hookrightarrow BUC(J;W^{2 - \frac{3}{p}}_p(\d \Gamma^*;\R)) \hookrightarrow BUC(J;BUC^1(\d \Gamma^*;\R)),
\end{align*}
where we used $p > 4$ in the second embedding. Combining this with the product rule one proves for $f, g \in \nabla^1 Z_\varrho$ via direct estimates
\begin{align*}
	\norm{fg}_{\nabla^1 Z_\varrho} \leq \tilde{c} \norm{f}_{\nabla^1 Z_\varrho} \norm{g}_{\nabla^1 Z_\varrho}.
\end{align*}
Further details can be found in \cite{Mue13}.
\end{proof}

\begin{lemma}\label{lem:TechnicalPreliminariesMCF}
Let $J := [0,T]$ and $4 < p < \infty$ and $\mathbb{B}^\mathbb{E}_r(\mathbb{O}) := \left\{\Phi \in \mathbb{E} \left| \norm{\Phi}_{\mathbb{E}} < r\right.\right\}$. Then there exists an $r > 0$ such that $N(\mathbb{B}^\mathbb{E}_r(\mathbb{O})) \subseteq \mathbb{F}$. Moreover, $N \in C^1(\mathbb{B}^\mathbb{E}_r(\mathbb{O});\mathbb{F})$ and $\norm{DN[\mathbb{O}]}_{\mathcal{L}(\mathbb{E},\mathbb{F})} \leq c T^{\frac{1}{q}}$ for some $q > p$, where $DN: \mathbb{B}^\mathbb{E}_r(\mathbb{O}) \longrightarrow \mathcal{L}(\mathbb{E},\mathbb{F})$ denotes the Fr\'echet derivative of $N$.
\end{lemma}
\begin{proof}
The linearization we calculated in Section \ref{ssec:LinearMCF} is indeed the Fr\'echet derivative as we will see later in this proof. Our first goal is to show
\begin{align*}
	F(u) & := H_\Gamma(u) - \bar{H}(u) - V_\Gamma(u) + \frac{d}{dt} u + \mathcal{A}(q,D) u \in X, \\
	G(u,\varrho) & := a + b \kappa_{\d D}(\varrho) + \skp{n_\Gamma(u)}{n_D(u)} - v_{\d D}(\varrho) + \frac{d}{dt} \varrho + \mathcal{B}_0(q,D) u + \mathcal{C}_0(q,D_\d) \varrho \in Y_0
\end{align*}
for all $(u,\varrho) \in \mathbb{B}^\mathbb{E}_r(\mathbb{O})$. For $r > 0 $ small enough all the terms appearing in $F$ and $G$ are well-defined and the linear parts of $F$ and $G$ can be omitted since
\begin{itemize}
	\item $\mathcal{A}(D) u \in X$ due to $u \in Z_u \subseteq L_p(J;W^2_p(\Gamma^*;\R))$ and $\mathcal{A}$ is of second order in space.
	\item $\frac{d}{dt} u \in X$ due to $u \in Z_u \subseteq W^1_p(J;L_p(\Gamma^*;\R))$ and $\frac{d}{dt}$ is of first order in time.
	\item $\frac{d}{dt} \varrho \in Y_0$ due to $\varrho \in Z_\varrho \subseteq W^{\frac{3}{2} - \frac{1}{2p}}_p(J;L_p(\d \Gamma^*;\R))$ and $\frac{d}{dt}$ is of first order in time as well as $\varrho \in Z_\varrho \hookrightarrow W^1_p(J;W^{1 - \frac{1}{p}}_p(\d \Gamma^*;\R))$ according to Lemma 4.3 of \cite{DSS08} with $\sigma = \frac{2p}{3p-1}$.
	\item $\mathcal{B}_0(D) u \in Y_0$ due to $u \in Z_u \hookrightarrow W^{\frac{1}{2}}_p(J;W^1_p(\Gamma^*;\R))$ because of Lemma \ref{lem:Embeddings}(i) with $\sigma = \frac{1}{2}$ and $\mathcal{B}_0$ is of first order in space. This leads to
			\begin{align*}
				\mathcal{B}_0(D) u \in W^{\frac{1}{2}}_p(J;L_p(\Gamma^*;\R)) \cap L_p(J;W^1_p(\Gamma^*;\R))
			\end{align*}
			and by (A.24) in \cite{Gru95} the trace operator $\gamma_0$ maps as follows
			\begin{align*}
				\gamma_0: W^{\frac{1}{2}}_p(J;L_p(\Gamma^*;\R)) \cap L_p(J;W^1_p(\Gamma^*;\R)) \longrightarrow Y_0.
			\end{align*}
	\item $\mathcal{C}_0(D_\d) \varrho \in Y_0$ due to $\varrho \in Z_\varrho \subseteq L_p(J;W^{3 - \frac{1}{p}}_p(\d \Gamma^*;\R))$ and $\mathcal{C}_0$ is of second order in space as well as $\varrho \in Z_\varrho \hookrightarrow W^{\frac{1}{2} - \frac{1}{2p}}_p(J;W^2_p(\d \Gamma^*;\R))$ due to Lemma 4.3 of \cite{DSS08} with $\sigma = \frac{p-1}{3p-1}$ and $\mathcal{C}_0$ is of second order in space.
\end{itemize}
Next we want to turn our attention to the two velocities in $F$ and $G$. We first remark
\begin{align*}
	\sup_{t \in J} \sup_{q \in \Gamma^*} \left|n_{\Gamma^*}(u) \cdot \d_w \Psi(u)\right| \leq \sup_{t \in J} \sup_{q \in \Gamma^*} \norm{n_{\Gamma^*}(u)} \norm{\d_w \Psi(u)} = \sup_{t \in J} \sup_{q \in \Gamma^*} \norm{\d_w \Psi(u)} =: c < \infty
\end{align*}
since $J$ is compact, $\d_w \Psi$ is continuous up to the boundary and since we have assumed $0 < \alpha < \pi$. Hence we obtain
\begin{align*}
	\norm{V_\Gamma(u)}_X \leq c \norm{\d_t u(t,q)}_{L_p(J;L_p(\Gamma^*;\R))} \leq c \norm{u(t,q)}_{W^1_p(J;L_p(\Gamma^*;\R))} < \infty.
\end{align*}
Similarly we get for the normal boundary velocity $\norm{v_{\d \Gamma}(\varrho)}_{W^{\frac{1}{2} - \frac{1}{2p}}_p(J;L_p(\d \Gamma^*;\R))} < \infty$ due to $\varrho \in Z_\varrho \hookrightarrow W^{\frac{1}{2} - \frac{1}{2p} + \epsilon}_p(J;C^1(\d \Gamma^*;\R))$ for some $\epsilon > 0$ by Lemma \ref{lem:Embeddings}. This proves $\norm{v_{\d \Gamma}(\varrho)}_{Y_0} < \infty$. \\
Now we consider the angle term and the constant $a$. Since $J$, $\d \Gamma^*$ and $\Gamma^*$ are bounded and
\begin{align*}
	\sup_{t \in J} \sup_{q \in \Gamma^*} \left|\skp{n_\Gamma(u)}{n_D(u)}\right| \leq \sup_{t \in J} \sup_{q \in \Gamma^*} \norm{n_\Gamma(u)} \norm{n_D(u)} = \sup_{t \in J} \sup_{q \in \Gamma^*} 1 = 1 < \infty
\end{align*}
we conclude $\norm{\skp{n_\Gamma(u)}{n_D(u)}}_{Y_0} < \infty$ which follows from the same estimates and embeddings as in Lemma \ref{lem:AlgebraMCF}(ii) and $\norm{a}_{Y_0} < \infty$. \\
Finally, we look at the remaining curvature terms. Due to Lemma \ref{lem:Embeddings} we see that $|u(t,q)|$ and $|\nabla_{\Gamma^*} u(t,q)|$ remain bounded. This shows that for a maybe even smaller $r$ the first fundamental form of all the hypersurfaces in the family $\left(\Gamma_{\rho}(t)\right)_{t \in J}$ is not degenerated. Because of the facts that on the one hand $H_\Gamma(u)$ depends linearly on the second space derivatives of $u$ and on the other hand the coefficients, that involve only $u$ and its first derivatives, are bounded as seen above, we get
\begin{align*}
	\norm{H_\Gamma(u)}_X \leq c \left(\norm{\nabla_{\Gamma^*}^2 u}_X + 1\right) \leq c \left(\norm{u}_{L_p(J;W^2_p(\Gamma^*;\R))} + 1\right) \leq c \left(\norm{u}_{Z_u} + 1\right) < \infty.
\end{align*}
For the non-local mean integral we first take a look at the area. Surely, $\int_{\Gamma(t)}\limits 1 \dH^2$ depends continuously on $t \in [0,T]$. Therefore, we obtain $0 < c \leq \int_{\Gamma(t)}\limits 1 \dH^2 \leq C$. This leads to
\begin{align*}
	\left|\bar{H}(u(t))\right| = \frac{1}{\int_{\Gamma(t)}\limits 1 \dH^2} \left|\int_{\Gamma(t)} H_{\Gamma(t)} \dH^2\right| \leq \frac{1}{c} \left|\int_{\Gamma^*} H_\Gamma(u(t)) J(u(t), \nabla_{\Gamma^*} u(t)) \dH^2\right|,
\end{align*}
where $J$ is some determinant term that includes no second or higher order derivatives of $u$. Thus, the integrand depends linearly on the second space derivatives of $u$ and the coefficients are bounded. The same argumentation as above finally leads to
\begin{align*}
	\norm{\bar{H}(u)}_X & = \left(\int_J \int_{\Gamma^*} |\bar{H}(u)|^p \dH^2 dt\right)^{\frac{1}{p}}
		  \leq C^{\frac{1}{p}} \left(\int_J |\bar{H}(u)|^p dt\right)^{\frac{1}{p}} \\
		& \leq \frac{C^{\frac{1}{p}}}{c} \left(\int_J \left|\int_{\Gamma^*} H_\Gamma(u) J(u, \nabla_{\Gamma^*} u) \dH^2\right|^p dt\right)^{\frac{1}{p}} \\
		& \leq \hat{c} \norm{H_\Gamma(u) J(u, \nabla_{\Gamma^*} u)}_X \leq c \left(\norm{\nabla_{\Gamma^*}^2 u}_X + 1\right)
		  \leq c \left(\norm{u}_{Z_u} + 1\right) < \infty.
\end{align*}
To estimate the geodesic curvature we observe $Z_\varrho \hookrightarrow W^{\frac{1}{p}}_p(J;W^{3 - \frac{3}{p}}_p(\d \Gamma^*;\R))$ if we choose $\sigma = \frac{2}{3p - 1}$ in Lemma 4.3 of \cite{DSS08}. By usual Sobolev embeddings and the assumption $p > 4$ we get in addition
\begin{align*}
	Z_\varrho \hookrightarrow W^{\frac{1}{p}}_p(J;W^{3 - \frac{3}{p}}_p(\d \Gamma^*;\R)) \hookrightarrow BUC(J;W^{3 - \frac{3}{p}}_p(\d \Gamma^*;\R)) \hookrightarrow BUC(J;BUC^2(\d \Gamma^*;\R))
\end{align*}
and hence $|\varrho(t,q)|$, $|\d_{\sigma} \varrho(t,q)|$ and $|\d_{\sigma}^2 \varrho(t,q)|$ remain bounded. The smooth dependence of $\kappa_{\d D}(\varrho)$ on $\varrho$ shows $\norm{\kappa_{\d D}(\varrho)}_{Y_0} < \infty$ in the same way as for $v_{\d \Gamma}(\varrho)$. This finally completes the proof of $N(\mathbb{B}^\mathbb{E}_r(\mathbb{O})) \subseteq \mathbb{F}$. \\
It remains to prove $N \in C^1(\mathbb{B}^\mathbb{E}_r(\mathbb{O});\mathbb{F})$. Here we note that we have calculated in Section \ref{ssec:LinearMCF} the first variations of the parts of $N$. We will denote these variations by the prefix $\delta$. Assuming the Lipschitz continuity of $\delta_\Phi N$ leads us to the Fr\'echet differentiability. Therefore for $N \in C^1(\mathbb{B}^\mathbb{E}_r(\mathbb{O});\mathbb{F})$ the Lipschitz continuity of $\delta_\Phi N$ remains to be proven. To simplify the formulas we look at each term in each component of $N$ separately. \\
Starting with $H_\Gamma$ we know that we can write
\begin{align*}
	H_\Gamma(u) = \sum_{|\alpha|=2} a_\alpha(u, \nabla_{\Gamma^*} u) \d^\alpha u + b(u, \nabla_{\Gamma^*} u)
\end{align*}
with $a_\alpha, b \in C^3(U)$ and $U \subseteq \R \times \R^2$ a closed neighborhood of $0$. Linearizing this we obtain
\begin{align*}
	(\delta_u H_\Gamma(u))(v) & = \sum_{|\alpha|=2} \left(\d_1 a_\alpha(u, \nabla_{\Gamma^*} u) (\d^\alpha u) v + \d^\alpha u \left(\d_2 a_\alpha(u,\nabla_{\Gamma^*} u) \cdot \nabla_{\Gamma^*} v\right)\right. \\
		& \hspace*{12mm} \left.+ a_\alpha(u, \nabla_{\Gamma^*} u) \d^\alpha v\right) + \d_1 b(u, \nabla_{\Gamma^*} u) v + \d_2 b(u, \nabla_{\Gamma^*} u) \cdot \nabla_{\Gamma^*} v.
\end{align*}
Due to $a_\alpha, b \in C^3(U)$ the coefficients $a_\alpha, \d_1 a_\alpha, \d_2 a_\alpha, \d_1 b$ and $\d_2 b$ satisfy a Lipschitz condition on $\bar{B_r(0)} \subseteq \nabla^1 Z_u$, i.e.,
\begin{align*}
	\norm{\d_1 a_\alpha(u, \nabla_{\Gamma^*} u) - \d_1 a_\alpha(\tilde{u}, \nabla_{\Gamma^*} \tilde{u})}_{\nabla^1 Z_u} \leq c \left(\norm{u - \tilde{u}}_{\nabla^1 Z_u} + \norm{\nabla_{\Gamma^*} u - \nabla_{\Gamma^*} \tilde{u}}_{\nabla^1 Z_u}\right)
\end{align*}
for $u, \tilde{u} \in \bar{B_r(0)}$. This Lipschitz condition can be seen via
\begin{align*}
	\norm{f(u) - f(\tilde{u})}_{\nabla^1 Z_u} & = \norm{\int^1_0 \frac{d}{dt} f(\tilde{u} + t(u - \tilde{u})) dt}_{\nabla^1 Z_u} \\
		& \leq c \underbrace{\norm{\int^1_0 f'(\tilde{u} + t(u - \tilde{u})) dt}_{\nabla^1 Z_u}}_{\leq c} \norm{u - \tilde{u}}_{\nabla^1 Z_u} \leq c \norm{u - \tilde{u}}_{\nabla^1 Z_u},
\end{align*}
where we have used Lemma \ref{lem:AlgebraMCF}. Because of $\norm{\bullet}_{\nabla^1 Z_u} \leq c \norm{\bullet}_{Z_u}$, any functions $u, \tilde{u} \in \bar{B_r(0)} \subseteq Z_u$ are also in $\bar{B_{cr}(0)} \subseteq \nabla^1 Z_u$. Hence for $u, \tilde{u} \in \bar{B_r(0)} \subseteq Z_u$ we get
\begin{align*}
	\norm{\delta_u H_\Gamma(u) - \delta_u H_\Gamma(\tilde{u})}_{\mathcal{L}(Z_u,X)}
	& \leq \sum_{|\alpha|=2} \left(\norm{\d_1 a_\alpha(u, \nabla_{\Gamma^*} u) \d^\alpha u - \d_1 a_\alpha(u, \nabla_{\Gamma^*} u) \d^\alpha \tilde{u}}_X\right. \\
	& + \left.\norm{\d_1 a_\alpha(u, \nabla_{\Gamma^*} u) \d^\alpha \tilde{u} - \d_1 a_\alpha(\tilde{u}, \nabla_{\Gamma^*} \tilde{u}) \d^\alpha \tilde{u}}_X\right) \norm{\id}_{\mathcal{L}(Z_u,X)} \\
	& + \sum_{|\alpha|=2}\left(\norm{\d^\alpha u \d_2 a_\alpha(u, \nabla_{\Gamma^*} u) - \d^\alpha \tilde{u} \d_2 a_\alpha(u, \nabla_{\Gamma^*} u)}_X\right. \\
	& + \sum_{|\alpha|=2}\left.\norm{\d^\alpha \tilde{u} \d_2 a_\alpha(u, \nabla_{\Gamma^*} u) - \d^\alpha \tilde{u} \d_2 a_\alpha(\tilde{u}, \nabla_{\Gamma^*} \tilde{u})}_X\right) \norm{\nabla_{\Gamma^*}}_{\mathcal{L}(Z_u,X)} \\
	& + \sum_{|\alpha|=2}\norm{a_\alpha(u, \nabla_{\Gamma^*} u) - a_\alpha(\tilde{u}, \nabla_{\Gamma^*} \tilde{u})}_X \norm{\d^\alpha}_{\mathcal{L}(Z_u,X)} \\
	& + \norm{\d_1 b(u, \nabla_{\Gamma^*} u) - \d_1 b(\tilde{u}, \nabla_{\Gamma^*} \tilde{u})}_X \norm{\id}_{\mathcal{L}(Z_u,X)} \\
	& + \norm{\d_2 b(u, \nabla_{\Gamma^*} u) - \d_2 b(\tilde{u}, \nabla_{\Gamma^*} \tilde{u})}_X \norm{\nabla_{\Gamma^*}}_{\mathcal{L}(Z_u,X)}.
\end{align*}
Since $\nabla^1 Z_u \hookrightarrow L_\infty(J;L_\infty(\Gamma^*;\R))$ we have $\norm{\bullet}_{L_\infty(J;L_\infty(\Gamma^*;\R))} \leq c \norm{\bullet}_{\nabla^1 Z_u}$. Moreover, in the same manner as for the linear parts $\mathcal{A}, \mathcal{B}_0, \mathcal{B}_1, \mathcal{C}_0$ and $\mathcal{C}_1$ we prove $\norm{\id}_{\mathcal{L}(Z_u,X)} < \infty$, $\norm{\nabla_{\Gamma^*}}_{\mathcal{L}(Z_u,X)} < \infty$ and $\norm{\d^\alpha}_{\mathcal{L}(Z_u,X)} < \infty$ for $|\alpha| = 2$, which enables us to continue the inequality above as follows
\begin{align*}
	\norm{\delta_u H_\Gamma(u) - \delta_u H_\Gamma(\tilde{u})}_{\mathcal{L}(Z_u,X)}
		& \leq c \sum_{|\alpha|=2} \underbrace{\norm{\d_1 a_\alpha(u, \nabla_{\Gamma^*} u)}_{L_\infty(J;L_\infty(\Gamma^*))}}_{\leq \tilde{c} \norm{\d_1 a_\alpha(u, \nabla_{\Gamma^*} u)}_{\nabla^1 Z_u} \leq c(r)} \underbrace{\norm{\d^\alpha (u - \tilde{u})}_X}_{\leq \norm{u - \tilde{u}}_{Z_u}} \\
		& + \sum_{|\alpha|=2}\underbrace{\norm{\d_1 a_\alpha(u, \nabla_{\Gamma^*} u) - \d_1 a_\alpha(\tilde{u}, \nabla_{\Gamma^*} \tilde{u})}_{L_\infty(J;L_\infty(\Gamma^*))}}_{\leq \tilde{c} \norm{\ldots}_{\nabla^1 Z_u} \leq \hat{c} \left(\norm{u - \tilde{u}}_{\nabla^1 Z_u} + \norm{\nabla_{\Gamma^*} u - \nabla_{\Gamma^*} \tilde{u}}_{\nabla^1 Z_u}\right)} \underbrace{\norm{\d^\alpha \tilde{u}}_X}_{\leq r} \\
		& + \sum_{|\alpha|=2}\norm{\d_2 a_\alpha(u, \nabla_{\Gamma^*} u)}_{L_\infty(J;L_\infty(\Gamma^*))} \norm{\d^\alpha (u - \tilde{u})}_X \\
		& +\sum_{|\alpha|=2} \norm{\d_2 a_\alpha(u, \nabla_{\Gamma^*} u) - \d_2 a_\alpha(\tilde{u}, \nabla_{\Gamma^*} \tilde{u})}_{L_\infty(J;L_\infty(\Gamma^*))} \norm{\d^\alpha \tilde{u}}_X \\
		& + 3c \Big(\underbrace{\norm{u - \tilde{u}}_{\nabla^1 Z_u}}_{\leq \tilde{c} \norm{u - \tilde{u}}_{Z_u}} + \underbrace{\norm{\nabla_{\Gamma^*} u - \nabla_{\Gamma^*} \tilde{u}}_{\nabla^1 Z_u}}_{\leq \tilde{c} \norm{u - \tilde{u}}_{Z_u}}\Big) \\
		& \leq c(r) \norm{u - \tilde{u}}_{Z_u}.
\end{align*}
This shows Lipschitz continuity of $\delta_u H_\Gamma: \bar{B_r(0)} \subseteq Z_u \longrightarrow \mathcal{L}(Z_u,X)$ and hence $H_\Gamma \in C^1(\bar{B_r(0)},X)$. \\
Similar considerations can be made for $\kappa_{\d D}$, $V_\Gamma$ and $v_{\d D}$. Also in the same manner we see that the angle term $W(u) := \skp{n_\Gamma(u)}{n_D(u)}$ satisfies $W \in C^1(\bar{B_r(0)},\nabla^1 Z_u)$. Taking equation (A.24) from \cite{Gru95} into account we have for the trace operator
\begin{align*}
	\gamma_0: W^{\frac{1}{2}}_p(J;L_p(\Gamma^*)) \cap L_p(J;W^1_p(\Gamma^*)) \longrightarrow W^{\frac{1}{2}-\frac{1}{2p}}_p(J;L_p(\d \Gamma^*)) \cap L_p(J;W^{1-\frac{1}{p}}_p(\d \Gamma^*)),
\end{align*}
which is equivalent to $\gamma_0: \nabla^1 Z_u \longrightarrow Y_0$ and proves $\gamma_0 \circ W \in C^1(\bar{B_r(0)},Y_0)$. \\
As seen before the integral mean of the mean curvature has the form
\begin{align*}
	\bar{H}(u(t)) = \frac{1}{\int_{\Gamma^*}\limits J(u(t), \nabla_{\Gamma^*} u(t)) \dH^2} \int_{\Gamma^*} H_\Gamma(u) J(u(t), \nabla_{\Gamma^*} u(t)) \dH^2,
\end{align*}
where $J$ is some determinant term. Hence we can write
\begin{align*}
	\bar{H}(u) = \int_{\Gamma^*} \sum_{|\alpha|=2} a_\alpha(u, \nabla_{\Gamma^*} u) \d^\alpha u + b(u, \nabla_{\Gamma^*} u) \dH^2
\end{align*}
with $a_\alpha$ and $b$ similar to the considerations for $H_\Gamma$, simply including the terms $J(u, \nabla_{\Gamma^*} u)$ and $\left(\int_{\Gamma^*}\limits J(u(t), \nabla_{\Gamma^*} u(t)) \dH^2\right)^{-1}$. With the same estimates as for $H_\Gamma$ we obtain
\begin{align*}
	\norm{\delta_u \bar{H}(u) - \delta_u \bar{H}(\tilde{u})}_{\mathcal{L}(Z_u,X)}
		& \leq \ldots \leq \int_{\Gamma^*} c(r) \norm{u - \tilde{u}}_{Z_u} \dH^2 \leq \tilde{c}(r) \norm{u - \tilde{u}}_{Z_u}
\end{align*}
showing $\delta_u \bar{H}: \bar{B_r(0)} \subseteq Z_u \longrightarrow \mathcal{L}(Z_u,X)$ is Lipschitz continuous and $\bar{H} \in C^1(\bar{B_r(0)},X)$. All of these continuity statements show $N \in C^1(\mathbb{B}^\mathbb{E}_r(\mathbb{O});\mathbb{F})$ if we choose the radius $r$ appropriately. \\
Now we consider the remaining statement $\norm{DN[\mathbb{O}]}_{\mathcal{L}(\mathbb{E},\mathbb{F})} \leq c T^{\frac{1}{q}}$. First we remark that without the $\bar{H}$-term in $N$ we would obtain $DN[\mathbb{O}] = 0$. Therefore, it suffices to consider $\norm{D\bar{H}[\mathbb{O}]}_{\mathcal{L}(\mathbb{E},X)}$. We estimate with the help of Lemma \ref{lem:LinearHMean}
\begin{align*}
	\norm{D\bar{H}[\mathbb{O}](u,\varrho)}_X
		& \leq \norm{\frac{1}{\int_{\Gamma^*} 1 \dH^2} \int_{\Gamma^*} (\Delta_{\Gamma^*} + |\sigma^*|^2 - H_{\Gamma^*}^2 + \bar{H}(\mathbb{O}) H_{\Gamma^*}) u \dH^2}_X \\
		& + \norm{\frac{1}{\int_{\Gamma^*} 1 \dH^2} \int_{\d \Gamma^*} (H_{\Gamma^*} - \bar{H}(\mathbb{O})) \cot(\alpha) \varrho \dH^1}_X \\
		& \leq A(\Gamma^*)^{\frac{1}{p} - 1} \norm{\int_{\Gamma^*} (\Delta_{\Gamma^*} + \underbrace{|\sigma^*|^2 - H_{\Gamma^*}^2 + \bar{H}(\mathbb{O}) H_{\Gamma^*}}_{\text{bounded}}) u \dH^2}_{L_p(J)} \\
		& + A(\Gamma^*)^{\frac{1}{p} - 1} \norm{\int_{\d \Gamma^*} \underbrace{(H_{\Gamma^*} - \bar{H}(\mathbb{O})) \cot(\alpha)}_{\text{bounded}} \varrho \dH^1}_{L_p(J)} \\
		& \leq \tilde{c} \left(\norm{\int_{\Gamma^*} \Delta_{\Gamma^*} u \dH^2}_{L_p(J)} + \norm{\int_{\Gamma^*} u \dH^2}_{L_p(J)} + \norm{\int_{\d \Gamma^*} \varrho \dH^1}_{L_p(J)}\right).
\end{align*}
Using Gauss' theorem for hypersurfaces we can write the first integral as a lower order boundary integral. Choosing $q > p$ arbitrarily and $r > p$ such that $\frac{1}{p} = \frac{1}{q} + \frac{1}{r}$ and utilizing H\"older's inequality we can continue the estimate above as follows
\begin{align*}
	\norm{D\bar{H}[\mathbb{O}](u,\varrho)}_X
		& \leq \tilde{c} \left(\norm{\int_{\d \Gamma^*} \nabla_{\Gamma^*} u \cdot n_{\d \Gamma^*} \dH^1}_{L_p(J)} + \norm{\int_{\Gamma^*} u \dH^2}_{L_p(J)} + \norm{\int_{\d \Gamma^*} \varrho \dH^1}_{L_p(J)}\right) \\
		& \leq \hat{c} \left(\norm{1}_{L_q(J)} \norm{\int_{\d \Gamma^*} \nabla_{\Gamma^*} u \cdot n_{\d \Gamma^*} \dH^1}_{L_r(J)} + \norm{1}_{L_q(J)} \norm{\int_{\Gamma^*} u \dH^2}_{L_r(J)}\right. \\
		& + \left.\norm{1}_{L_q(J)} \norm{\int_{\d \Gamma^*} \varrho \dH^1}_{L_r(J)}\right) \\
		& \leq c T^{\frac{1}{q}} \left(\norm{\nabla_{\Gamma^*} u}_{L_r(J;L_1(\d \Gamma^*))} + \norm{u}_{L_r(J;L_1(\Gamma^*))} + \norm{\varrho}_{L_r(J;L_1(\d \Gamma^*))}\right).
\end{align*}
As seen before we have $\nabla_{\Gamma^*} u \in \nabla^1 Z_u$ and by Lemma 4.3 of \cite{DSS08} with $\sigma = 1 - \frac{2}{p}$, the trace operator $\gamma_0$ and $r > p > 0$ we obtain the embeddings
\begin{align*}
	\nabla^1 Z_u & = W^{\frac{1}{2}}_p(J;L_p(\Gamma^*)) \cap L_p(J;W^1_p(\Gamma^*)) \hookrightarrow W^{\frac{1}{2}-\frac{1}{p}}_p(J;W^{\frac{2}{p}}_p(\Gamma^*)) \\
					 & \hookrightarrow W^{\frac{1}{2}-\frac{1}{p}}_p(J;W^{\frac{1}{p}}_p(\d \Gamma^*)) \hookrightarrow L_r(J;L_1(\d \Gamma^*)),
\end{align*}
which shows $\norm{\nabla_{\Gamma^*} u}_{L_r(J;L_1(\d \Gamma^*))} \leq c \norm{\nabla_{\Gamma^*} u}_{\nabla^1 Z_u} \leq \hat{c} \norm{u}_{Z_u}$. Without the trace operator in the estimate above we show $\norm{u}_{L_r(J;L_1(\Gamma^*))} \leq \norm{u}_{\nabla^1 Z_u} \leq \hat{c} \norm{u}_{Z_u}$. Finally $\norm{\varrho}_{L_r(J;L_1(\d \Gamma^*))} \leq c \norm{\varrho}_{Z_\varrho}$, because of
\begin{align*}
	Z_\varrho = W^{\frac{3}{2}-\frac{1}{2p}}_p(J;L_p(\d \Gamma^*)) \cap L_p(J;W^{3-\frac{1}{p}}_p(\Gamma^*)) \subseteq W^{\frac{3}{2}-\frac{1}{2p}}_p(J;L_p(\d \Gamma^*)) \hookrightarrow L_r(J;L_1(\d \Gamma^*)).
\end{align*}
Using these three facts we see
\begin{align*}
	\norm{D\bar{H}[\mathbb{O}](u,\varrho)}_X \leq \hat{c} T^{\frac{1}{q}} \left(2 \norm{u}_{Z_u} + \norm{\varrho}_{Z_\varrho}\right) \leq c T^{\frac{1}{q}} \norm{(u,\varrho)}_\mathbb{E}
\end{align*}
proving the desired estimate $\norm{D\bar{H}[\mathbb{O}]}_{\mathcal{L}(\mathbb{E},X)} \leq c T^{\frac{1}{q}}$. For more details on this proof we refer to \cite{Mue13}.
\end{proof}

\begin{rem}
(i) An important fact for the following considerations is that $L$ is an isomorphism. We do not need to consider the condition $g_0(0) - \mathcal{B}_0(D) u_0 - \mathcal{C}_0(D_\d) \varrho_0 \in \pi_1 Z_\varrho$ in Theorem \ref{thm:LocalExistenceGeneral} due to the same arguments as in the proof of Corollary \ref{cor:LocalExistence1}. Moreover, $\mathcal{B}_1(D) u_0 + \mathcal{C}_1(D_\d) \varrho_0 = g_1(0)$ can be dropped, because $g_1 \equiv 0$ and $(u_0,\varrho_0) \in \mathbb{I}$. Due to Theorem \ref{thm:LocalExistenceGeneral} $L$ is an isomorphism between $\mathbb{E}$ and $\mathbb{F} \times \mathbb{I}$. \\
(ii) Although we have not indicated this dependence so far, the spaces $\mathbb{E}$ and $\mathbb{F}$ actually depend on $T$ and should have been better denoted by $\mathbb{E}_T$ and $\mathbb{F}_T$. The same is true for the operators $L$ and $N$. The justification for this notational inexactness will be given in the following Lemma. This will be the first and also last segment where we will use the exact notation to indicate the dependence on $T$.
\end{rem}

In order to get the norm estimates uniformly bounded in $T$, we choose the norms on $\mathbb{E}_T,Y_T,\mathbb{F}_T$ as follows:
\begin{align*}
  \|(u,\rho)\|_{\mathbb{E}_T} &= \|u\|_{Z_u}+ \|\rho\|_{Z_\rho} + \|u|_{t=0}\|_{\pi Z_u}+ \|\rho|_{t=0}\|_{\pi Z_\rho},\\
  \| g\|_{Y_0} &= \|g\|_{W^{\frac12-\frac1{2p}}_p(0,T;L_p(\partial\Gamma^\ast))} +\|g\|_{L_p(0,T;W^{1-\frac1p}_p(\partial\Gamma^\ast))}+ \|g|_{t=0}\|_{\pi_1 Z_\rho},\\
  \|(f,g,0)\|_{\mathbb{F}_T} &= \|f\|_{X}+ \|g\|_{Y_0},
\end{align*}
where $X$ is normed in the standard way.

\begin{lemma}\label{lem:UniformlyMCF}
Let $T_0 > 0$ be fixed and $T \in (0,T_0)$ arbitrary. \\
(i) There exists a bounded extension operator from $\mathbb{F}_T$ to $\mathbb{F}_{T_0}$, i.e., for all $f \in \mathbb{F}_T$ there is a $\tilde{f} \in \mathbb{F}_{T_0}$ with $\left.\tilde{f}\right|_{[0,T]} = f$ and $\norm{\tilde{f}}_{\mathbb{F}_{T_0}} \leq c(T_0) \norm{f}_{\mathbb{F}_T}$. \\
(ii) The operator norm of $L_T^{-1}: \mathbb{F}_T \times \mathbb{I} \longrightarrow \mathbb{E}_T$ is uniformly bounded in $T$. \\
(iii) There exists a bounded extension operator from $\mathbb{E}_T$ to $\mathbb{E}_{T_0}$, i.e., for all $\Phi \in \mathbb{E}_T$ there is some $\tilde{\Phi} \in \mathbb{E}_{T_0}$ with $\left.\tilde{\Phi}\right|_{[0,T]} = \Phi$ and $\norm{\tilde{\Phi}}_{\mathbb{E}_{T_0}} \leq c(T_0) \norm{\Phi}_{\mathbb{E}_T}$. \\
(iv) The uniform estimate $\norm{DN_T[\Phi] - DN[\mathbb{O}]}_{\mathcal{L}(\mathbb{E}_T,\mathbb{F}_T)} \leq c(T_0) \norm{\Phi}_{\mathbb{E}_T} < \infty$ holds for all $\Phi \in B^{\mathbb{E}_T}_r(0)$.
\end{lemma}
\begin{proof}
(i) Let $(f_1, f_2, 0) \in \mathbb{F}_T$. To define the extension we solve
\begin{align}\label{eq:ExtensionEq}
	\frac{d}{dt} \hat{g}(t) - \d_\sigma^2 \hat{g}(t) &= 0		  & \hspace*{-9mm} &\text{on } [0,T_0] \times \d \Gamma^*, \\
													  \hat{g}(0) &= f_2(T) & \hspace*{-9mm} &\text{on } \d \Gamma^*, \notag
\end{align}
where the trace in $t = T$ of a function $f_2 \in \tilde{Y}_0^T$ is an element of $\pi_1 Z_\varrho$ (cf. (A.25) of \cite{Gru95}). We obtain a unique
\begin{align*}
	\hat{g} \in \tilde{Y}_0^{T_0} := W^1_p((0,T_0); W^{-1 - \frac{1}{p}}_p(\d \Gamma^*;\R)) \cap L_p((0,T_0); W^{1 - \frac{1}{p}}_p(\d \Gamma^*;\R))\hookrightarrow Y_0^{T_0}
\end{align*}
with $\norm{\hat{g}}_{\tilde{Y}_0^{T_0}} \leq c(T_0) \norm{f_2(T)}_{\pi_1 Z_\varrho}$. We define the extension $(\tilde{f}_1, \tilde{f}_2, 0) \in \mathbb{F}_{T_0}$ by
\begin{align*}
	(\tilde{f}_1, \tilde{f}_2, 0) := \begin{cases}
													(f_1, f_2, 0)          & \text{for } t \in [0,T] \\
													(0, \hat{g}(t - T), 0) & \text{for } t \in (T,T_0]
												\end{cases}
\end{align*}
and have the estimate
\begin{align*}
	\norm{(\tilde{f}_1, \tilde{f}_2, 0)}_{\mathbb{F}_{T_0}}
		& \leq \norm{(f_1, f_2, 0)}_{\mathbb{F}_T} + \norm{(0, \hat{g}, 0)}_{\mathbb{F}_{T_0}} = \norm{(f_1, f_2, 0)}_{\mathbb{F}_T} + \norm{\hat{g}}_{Y_0^{T_0}} \\
		& \leq \norm{(f_1, f_2, 0)}_{\mathbb{F}_T} + c(T_0) \norm{f_2(T)}_{\pi_1 Z_\varrho} \leq \hat{c}(T_0) \norm{(f_1, f_2, 0)}_{\mathbb{F}_T}
\end{align*}
where the uniform estimate $\norm{f_2(T)}_{\pi_1 Z_\varrho} \leq c \norm{f_2}_{{Y}_0^T}$ follows from Theorem III.4.10.2 of \cite{Ama95} and Lemma 7.2 of \cite{Ama05}. \\[1ex]
(ii) We know that for $(f,\Phi_0) \in \mathbb{F}_T \times \mathbb{I}$ there is a unique solution $\Phi \in \mathbb{E}_T$ of $L_T \Phi = (f,\Phi_0)$. We use the extension $\tilde{f} \in \mathbb{F}_{T_0}$ from (i) to obtain a unique solution $\tilde{\Phi} \in \mathbb{E}_{T_0}$ such that $L_{T_0} \tilde{\Phi} = (\tilde{f},\Phi_0)$. Comparing $\Phi$ and $\tilde{\Phi}$ we see by the uniqueness that $\left.\tilde{\Phi}\right|_{[0,T]} = \Phi$ holds. Therefore we obtain the following estimate
\begin{align*}
	\norm{L_T^{-1} (f,\Phi_0)}_{\mathbb{E}_T} & = \norm{\Phi}_{\mathbb{E}_T} \leq \norm{\tilde{\Phi}}_{\mathbb{E}_{T_0}} = \norm{L_{T_0}^{-1} (\tilde{f}, \Phi_0)}_{\mathbb{E}_{T_0}} \\
												& \leq c(T_0) \left(\norm{\tilde{f}}_{\mathbb{F}_{T_0}} + \norm{\Phi_0}_{\mathbb{I}}\right) \leq \tilde{c}(T_0) \left(\norm{f}_{\mathbb{F}_T} + \norm{\Phi_0}_{\mathbb{I}}\right).
\end{align*}
This proves
\begin{align*}
	\sup_{T \in (0,T_0]} \norm{L_T^{-1}}_{\mathcal{L}(\mathbb{F}_T \times \mathbb{I},\mathbb{E}_T)} \leq \tilde{c}(T_0) < \infty,
\end{align*}
where it is important to note that $\tilde{c}(T_0)$ only depends on the fixed $T_0$ but not on $T$. \\[1ex]
(iii) For $\Phi \in \mathbb{E}_T$ we define $(f, \Phi_0) := L_T \Phi \in \mathbb{F}_T \times \mathbb{I}$ and use the extension from (i) to obtain $(\tilde{f},\Phi_0) \in \mathbb{F}_{T_0} \times \mathbb{I}$. Solving $L_{T_0} \tilde{\Phi} = (\tilde{f},\Phi_0)$ on $[0,T_0]$ leads to a unique $\tilde{\Phi} \in \mathbb{E}_{T_0}$. Due to the uniqueness on $[0,T]$ we have $\left.\tilde{\Phi}\right|_{[0,T]} = \Phi$. Moreover, we use $\norm{\tilde{f}}_{\mathbb{F}_{T_0}} \leq c(T_0) \norm{f}_{\mathbb{F}_T}$ from (i) to end up with
\begin{align*}
	\norm{\tilde{\Phi}}_{\mathbb{E}_{T_0}} & \leq c(T_0) \left(\norm{\tilde{f}}_{\mathbb{F}_{T_0}} + \norm{\Phi_0}_{\mathbb{I}}\right) \leq \hat{c}(T_0) \left(\norm{f}_{\mathbb{F}_T} + \norm{\Phi_0}_{\mathbb{I}}\right) \\
		& = \hat{c}(T_0) \norm{(f, \Phi_0)}_{\mathbb{F}_T \times \mathbb{I}} = \hat{c}(T_0) \norm{L_T \Phi}_{\mathbb{F}_T \times \mathbb{I}} \leq \tilde{c}(T_0) \norm{L_T} \norm{\Phi}_{\mathbb{E}_T},
\end{align*}
where $\norm{L_T} \leq c$ follows from direct estimates since the coefficients are uniformly bounded. \\[1ex]
(iv) Via the extension from (iii) we see
\begin{align*}
	N_T(\Phi) = N_T\left(\left.\tilde{\Phi}\right|_{[0,T]}\right) = \left.N_{T_0}\left(\tilde{\Phi}\right)\right|_{[0,T]}
\end{align*}
and hence $DN_T[\Phi](v) = \left.DN_{T_0}[\tilde{\Phi}](\tilde{v})\right|_{[0,T]}$. Therefore we get for the norms the following estimate
\begin{align*}
	\norm{DN_T[\Phi](v) - DN_T[\mathbb{O}](v)}_{\mathbb{F}_T} & \leq \norm{DN_{T_0}[\tilde{\Phi}](\tilde{v}) - DN_{T_0}[\mathbb{O}](\tilde{v})}_{\mathbb{F}_{T_0}} \\
		& \leq \norm{DN_{T_0}[\tilde{\Phi}] - DN_{T_0}[\mathbb{O}]}_{\mathcal{L}(\mathbb{E}_{T_0},\mathbb{F}_{T_0})} \norm{\tilde{v}}_{\mathbb{E}_{T_0}} \\
		& \leq c(T_0) \norm{DN_{T_0}[\tilde{\Phi}] - DN_{T_0}[\mathbb{O}]}_{\mathcal{L}(\mathbb{E}_{T_0},\mathbb{F}_{T_0})} \norm{v}_{\mathbb{E}_T} \\
		& \leq c(T_0) \tilde{c}(T_0) \norm{\tilde{\Phi}}_{\mathbb{E}_{T_0}} \norm{v}_{\mathbb{E}_T} \\
		& \leq \underbrace{c(T_0)^2 \tilde{c}(T_0)}_{=: \hat{c}(T_0)} \norm{\Phi}_{\mathbb{E}_T} \norm{v}_{\mathbb{E}_T}.
\end{align*}
This leads to $\norm{DN_T[\Phi]- DN_T[\mathbb{O}]}_{\mathcal{L}(\mathbb{E}_T,\mathbb{F}_T)} \leq \hat{c}(T_0) \norm{\Phi}_{\mathbb{E}_T} < \infty$.
\end{proof}

The two recently proven lemmas are the main tools for the application of the contraction mapping principle.

\begin{lemma}\label{lem:FixpointNonlinearMCF}
Let $4 < p < \infty$ and $J := [0,T]$ where $T > 0$ must be chosen sufficiently small. Then there exists some $\epsilon > 0$ such that for each $\Phi_0 \in \mathbb{I}$ with $\norm{\Phi_0}_\mathbb{I} < \epsilon$ there exists a unique solution $\Phi = (u, \varrho) \in \mathbb{E}$ of the equation $L \Phi = (N(\Phi),\Phi_0)$.
\end{lemma}
\begin{proof}
For simplicity we will skip the index $T$ of $L,N$, and the function spaces again in the following.
The equation $L \Phi = (N(\Phi),\Phi_0)$ is equivalent to the fixed point problem $K(\Phi) = \Phi$, where
\begin{align*}
	K(\Phi) := L^{-1} (N(\Phi),\Phi_0) \qquad \forall \, \Phi \in \mathbb{B}^\mathbb{E}_r(\mathbb{O}).
\end{align*}
We set $\mathbb{X}_r := \left\{\Phi \in \mathbb{B}^\mathbb{E}_r(\mathbb{O}) \left| \ \Phi(0) = \Phi_0\right.\right\}$. By Lemma \ref{lem:UniformlyMCF}(iv) we can choose $r > 0$ independent of $T$ such that
\begin{align*}
	\sup_{\Psi \in \mathbb{B}^\mathbb{E}_r(\mathbb{O})} \norm{DN[\Psi]}_{\mathcal{L}(\mathbb{E};\mathbb{F})} \leq \frac{1}{4 \sup_{T \in [0,T_0]}\limits \norm{L^{-1}}} + \norm{DN[\mathbb{O}]}_{\mathcal{L}(\mathbb{E};\mathbb{F})}.
\end{align*}
Then we see that for all $T \in [0,T_0]$ we have
\begin{align*}
	\sup_{\Psi \in \mathbb{B}^\mathbb{E}_r(\mathbb{O})} \norm{DN[\Psi]}_{\mathcal{L}(\mathbb{E};\mathbb{F})} \leq \frac{1}{4 \norm{L^{-1}}} + \norm{DN[\mathbb{O}]}_{\mathcal{L}(\mathbb{E};\mathbb{F})}.
\end{align*}
Before stating the main estimate we have to look at $\norm{N(\mathbb{O})}_\mathbb{F}$. Here we observe that
\begin{align*}
	\norm{N(\mathbb{O})}_\mathbb{F} = T^{\frac{1}{p}} \norm{H_{\Gamma^*} - \bar{H}(\mathbb{O})}_{L_p(\Gamma^*;\R)} + T^{\frac{1}{p}} \norm{a + b \kappa_{\d D^*} + \skp{n_{\Gamma^*}}{n_{D^*}}}_{W^{1 - \frac{1}{p}}_p(\d \Gamma^*; \R)}
\end{align*}
because $H_{\Gamma^*}, \bar{H}(\mathbb{O}), a, b \kappa_{\d D^*}$ and $\skp{n_{\Gamma^*}}{n_{D^*}}$ are time-independent. Hence $\norm{N(\mathbb{O})}_\mathbb{F} \xrightarrow[T \rightarrow 0]{} 0$. This fact and Lemma \ref{lem:TechnicalPreliminariesMCF} show that for a sufficiently small time interval $[0,T]$ we get $\norm{N(\mathbb{O})}_\mathbb{F} \leq \epsilon$ and $\norm{DN[\mathbb{O}]}_{\mathcal{L}(\mathbb{E};\mathbb{F})} \leq \epsilon$ for arbitrarily small $\epsilon > 0$. We use these facts in the estimate
\begin{align*}
	\norm{K(\Phi)}_\mathbb{E} & \leq \norm{L^{-1}} \left(\norm{N(\Phi)}_\mathbb{F} + \norm{\Phi_0}_\mathbb{I}\right)
		  \leq \norm{L^{-1}} \left(\norm{N(\Phi) - N(\mathbb{O})}_\mathbb{F} + \norm{N(\mathbb{O})}_\mathbb{F} + \norm{\Phi_0}_\mathbb{I}\right) \\
		& \leq \norm{L^{-1}} \left(\sup_{\Psi \in \mathbb{B}^\mathbb{E}_r(\mathbb{O})} \norm{DN[\Psi]}_{\mathcal{L}(\mathbb{E};\mathbb{F})} \norm{\Phi}_\mathbb{E} + \norm{N(\mathbb{O})}_\mathbb{F} + \norm{\Phi_0}_\mathbb{I}\right) \\
		& \leq \frac{1}{4} \norm{\Phi}_\mathbb{E} + \norm{L^{-1}} \norm{DN[\mathbb{O}]}_{\mathcal{L}(\mathbb{E};\mathbb{F})} \norm{\Phi}_\mathbb{E} + \norm{L^{-1}} \norm{N(\mathbb{O})}_\mathbb{F} + \norm{L^{-1}} \norm{\Phi_0}_\mathbb{I} \\
		& \leq \frac{r}{4} + \norm{L^{-1}} r \epsilon + 2 \norm{L^{-1}} \epsilon
\end{align*}
for every $\Phi \in \mathbb{X}_r$. By choosing
\begin{align*}
	\epsilon(r) \leq \frac{\min\{1,r\}}{4 \sup_{T \in [0,T_0]}\limits \norm{L^{-1}}}
\end{align*}
independent of $T$, we get $\norm{K(\Phi)} \leq \frac{r}{4} + \frac{r}{4} + \frac{r}{2} = r$, i.e., $K(\mathbb{X}_r) \subseteq \mathbb{X}_r$. To see that $K$ is contractive, we use Lemma \ref{lem:TechnicalPreliminariesMCF} again and observe that for all $\Phi_1, \Phi_2 \in \mathbb{X}_r$ the following holds
\begin{align*}
	\norm{K(\Phi_1) - K(\Phi_2)}_\mathbb{E} & \leq \norm{L^{-1}} \norm{N(\Phi_1) - N(\Phi_2)}_\mathbb{F} \\
		& \leq \norm{L^{-1}} \sup_{\Psi \in \mathbb{B}^\mathbb{E}_r(\mathbb{O})} \norm{DN[\Psi]}_{\mathcal{L}(\mathbb{E};\mathbb{F})} \norm{\Phi_1 - \Phi_2}_\mathbb{E} \\
		& \leq \frac{1}{4} \norm{\Phi_1 - \Phi_2}_\mathbb{E} + \norm{L^{-1}} \norm{DN[\mathbb{O}]}_{\mathcal{L}(\mathbb{E};\mathbb{F})} \norm{\Phi_1 - \Phi_2}_\mathbb{E} \\
		& \leq \frac{1}{4} \norm{\Phi_1 - \Phi_2}_\mathbb{E} + c \norm{L^{-1}} T^{\frac{1}{q}} \norm{\Phi_1 - \Phi_2}_\mathbb{E}.
\end{align*}
Choosing $T$ smaller than $\left(\frac{1}{4 c \norm{L^{-1}}}\right)^q$ we see $\norm{K(\Phi_1) - K(\Phi_2)}_\mathbb{E} \leq \frac{1}{2} \norm{\Phi_1 - \Phi_2}_\mathbb{E}$ and therefore $K: \mathbb{X}_r \longrightarrow \mathbb{X}_r$ is a contraction and the assertion follows from the contraction mapping principle.
\end{proof}

Transforming this statement into our original situation we can establish the following theorem.

\begin{thm}\label{thm:ShortTimeExistenceMCF}
Let $T > 0$ be sufficiently small and $4 < p < \infty$. Then there exists an $\epsilon > 0$ such that for each $\rho_0 \in \pi Z_u$ with $\rho_0|_{\d \Gamma^*} \in \pi Z_\varrho$ and $\norm{\rho_0}_{\pi Z_u} + \norm{\rho_0|_{\d \Gamma^*}}_{\pi Z_\varrho} < \epsilon$ there exists a unique solution $\rho \in Z_u$ with $\rho|_{\d \Gamma^*} \in Z_\varrho$ of the system
\begin{align*}
	V_\Gamma(\rho(t)) &= H_\Gamma(\rho(t)) - \bar{H}(\rho(t)) & &\text{in } [0,T] \times \Gamma^*, \\
	v_{\d D}(\rho(t)) &= a + b \kappa_{\d D}(\rho(t)) + \skp{n_\Gamma(\rho(t))}{n_D(\rho(t))} & &\text{on } [0,T] \times \d \Gamma^*, \\
	\rho(0) &= \rho_0 & &\text{in } \Gamma^*.
\end{align*}
\end{thm}

\begin{rem}\label{rem:Smallness}
Although it seems as if we would require smallness of the time interval and the initial hypersurface parametrized by $\rho_0$, this is actually not the case. We can start the evolution with rather general initial surfaces.
\end{rem}

This theorem completes the first part of this paper. In the next two sections we will examine the Willmore flow with the same strategy to prove the analogous result.

\section{The Willmore flow and its linearization}\label{sec:LinearWillmore}

In this section we want to consider an evolving hypersurface $\Gamma = (\Gamma(t))_{t \in I}$ in $\R^3$ whose motion is driven by the Willmore flow with line tension given by
\begin{align}\label{eq:WillmoreFlow1}
	V_\Gamma(\Psi(q,\rho(t,q))) &= -\Delta_\Gamma H_\Gamma(\Psi(q,\rho(t,q))) - \frac{1}{2} H_\Gamma(\Psi(q,\rho(t,q))) & & \notag \\
										 &\left(H_\Gamma(\Psi(q,\rho(t,q)))^2 - 4 K_\Gamma(\Psi(q,\rho(t,q)))\right) & &\text{in } \Gamma^*, \\ \label{eq:WillmoreFlow2}
	H_\Gamma(\Psi(q,\rho(t,q))) &= 0 & &\text{on } \d \Gamma^*, \\ \label{eq:WillmoreFlow3}
	v_{\d D}(\Psi(q,\rho(t,q))) &= \frac{1}{2} \sin(\alpha(q)) (\nabla_\Gamma H_\Gamma(\Psi(q,\rho(t,q))) \cdot n_{\d \Gamma}(\Psi(q,\rho(t,q)))) & & \notag \\
										 &+ a + b \kappa_{\d D}(\Psi(q,\rho(t,q))) & &\text{on } \d \Gamma^*,
\end{align}
and $\Psi$ is the curvilinear coordinate system as in Subsection \ref{ssec:MCF}.

Again this flow is motivated by the gradient flow of a certain energy namely
\begin{align}\label{eq:WillmoreEnergyLineTension}
	\WE(\Gamma) := \frac{1}{4} \int_\Gamma H^2 \dH^2 - a \int_D 1 \dH^2 + b \int_{\d \Gamma} 1 \dH^1.
\end{align}
Calculate the first variation of that energy, where we vary the geometry in the same way as in (\ref{eq:Variation})-(\ref{eq:FeasibleSet}) we get
\begin{align}\label{eq:VarWillmoreLineTemp}
	(\delta \WE(\Gamma))(\zeta) & = \frac{1}{2} \int_\Gamma \left(\Delta_\Gamma H_\Gamma + H_\Gamma \sum^2_{i=1} \kappa_i^2 - \frac{1}{2} H_\Gamma^3\right) (n_\Gamma \cdot \zeta) \dH^2 \notag \\ 
										 & + \frac{1}{2} \int_{\d \Gamma} \frac{1}{2} H_\Gamma^2 (n_{\d \Gamma} \cdot \zeta) + H_\Gamma (\nabla_\Gamma (n_\Gamma \cdot \zeta) \cdot n_{\d \Gamma}) - (\nabla_\Gamma H_\Gamma \cdot n_{\d \Gamma}) (n_\Gamma \cdot \zeta) \dH^1 \notag \\
										 & - \int_{\d \Gamma} a (n_{\d D} \cdot \zeta) + b (\vec{\kappa} \cdot \zeta) \dH^1,
\end{align}
With clever choices of the vector field $\zeta$ and the fundamental lemma of calculus of variations one can derive the three necessary conditions for stationary solutions to be
\begin{align}\label{eq:WillmoreNecCond1}
	\text{(a)} \quad &\Delta_\Gamma H_\Gamma + \frac{1}{2} H_\Gamma \left(H_\Gamma^2 - 4 K_\Gamma\right) = 0 & \quad &\text{in } \Gamma, \\ \label{eq:WillmoreNecCond2}
	\text{(b)} \quad &H_\Gamma = 0 & \quad &\text{on } \d \Gamma, \\ \label{eq:WillmoreNecCond3}
	\text{(c)} \quad &\frac{1}{2} \sin(\alpha) (\nabla_\Gamma H_\Gamma \cdot n_{\d \Gamma}) + a + b \kappa_{\d D} = 0 & \quad &\text{on } \d \Gamma,
\end{align}
where $K_\Gamma$ is the Gauss curvature of $\Gamma$.

We have already linearized some parts of this flow in Section \ref{ssec:LinearMCF} and will now only calculate the highest order derivatives of the remaining parts as we have seen in Section \ref{ssec:ShortTimeExLinearMCF} that only these are important for the short-time existence. Therefore, we will not compute the lower order terms like $K_\Gamma$, whose linearization obviously contains only first and second order derivatives of $\rho$.

For the next linearization of $\Delta_\Gamma H_\Gamma$ we need to indicate the dependence of the operator $\Delta_\Gamma$ on $\rho$. Following the notation of \cite{Dep10} we transform the surface gradient $\nabla_{\Gamma_\rho(t)}$ and the Laplace-Beltrami-Operator $\Delta_{\Gamma_\rho(t)}$ onto the reference surface $\Gamma^*$ using the pullback metric. The operators then read as
\begin{align}\label{eq:OperatorTransformation1}
	\Delta_{\Gamma_\rho(t)} H_{\Gamma_\rho(t)}(\Psi(q,\rho(t,q))) & = \Delta^\rho_{\Gamma^*} \tilde{H}_\rho(t,q), \\ \label{eq:OperatorTransformation2}
	\nabla_{\Gamma_\rho(t)} H_{\Gamma_\rho(t)}(\Psi(q,\rho(t,q))) & = d_q \Psi\left(\nabla^\rho_{\Gamma^*} \tilde{H}_\rho(t,q)\right),
\end{align}
where $\tilde{H}_\rho(t,q) := H_{\Gamma_\rho(t)}(\Psi(q,\rho(t,q)))$.

\begin{lemma}\label{lem:LinearLaplaceH}
The linearization of $\Delta_\Gamma H_\Gamma$ has the form
\begin{align*}
	\left.\frac{d}{d\epsilon} \Delta_\Gamma H_\Gamma(\Psi(q,\epsilon\rho(t,q)))\right|_{\epsilon=0} = \Delta_{\Gamma^*} \Delta_{\Gamma^*} \rho(t,q) + G_1\left(q,\rho(t,q),\nabla_{\Gamma^*} \rho(t,q),\nabla_{\Gamma^*}^2 \rho(t,q)\right),
\end{align*}
where $G_1$ is a smooth function.
\end{lemma}
\begin{proof}
By the product rule we obtain
\begin{align*}
	\left.\frac{d}{d\epsilon} \Delta^{\epsilon \rho}_{\Gamma^*} \tilde{H}_{\epsilon \rho}(t,q)\right|_{\epsilon=0}
		& = \left.\frac{d}{d\epsilon} \Delta^{\epsilon \rho}_{\Gamma^*}\right|_{\epsilon=0} \tilde{H}_0(t,q) + \Delta^0_{\Gamma^*} \left.\frac{d}{d\epsilon} \tilde{H}_{\epsilon \rho}(t,q)\right|_{\epsilon=0}.
\end{align*}
For $\rho \equiv 0$ we obviously see $\Delta^0_{\Gamma^*} = \Delta_{\Gamma^*}$ and $\tilde{H}_0 = H_{\Gamma^*}$. In combination with Lemma \ref{lem:LinearHGamma} we have
\begin{align*}
	\left.\frac{d}{d\epsilon} \Delta^{\epsilon \rho}_{\Gamma^*} \tilde{H}_{\epsilon \rho}(t,q)\right|_{\epsilon=0}
		& = \left.\frac{d}{d\epsilon} \Delta^{\epsilon \rho}_{\Gamma^*}\right|_{\epsilon=0} H_{\Gamma^*}(q) + \Delta_{\Gamma^*} \Delta_{\Gamma^*} \rho(t,q) \\
		& + \Delta_{\Gamma^*} \left(|\sigma^*|^2(q) \rho(t,q) + \left(\nabla_{\Gamma^*} H_{\Gamma^*}(q) \cdot P\left(\d_w \Psi(q,0)\right)\right) \rho(t,q)\right).
\end{align*}
The linearization of the Laplace-Beltrami operator applied to $H_{\Gamma^*}$ contains at most second order derivatives of $\rho$ and the claim follows.
\end{proof}

\begin{lemma}\label{lem:LinearGradHConormal}
The linearization of $\nabla_\Gamma H_\Gamma \cdot n_{\d \Gamma}$ is of the form
\begin{align*}
	\left.\frac{d}{d\epsilon} \nabla_\Gamma H_\Gamma(\Psi(q,\epsilon\rho(t,q))) \cdot n_{\d \Gamma}(\Psi(q,\epsilon\rho(t,q)))\right|_{\epsilon=0}
		& = \nabla_{\Gamma^*} \Delta_{\Gamma^*} \rho(t,q) \cdot n_{\d \Gamma^*}(q) \\
		& + G_2\left(q,\rho(t,q),\nabla_{\Gamma^*} \rho(t,q),\nabla_{\Gamma^*}^2 \rho(t,q)\right),
\end{align*}
where $G_2$ is a smooth function.
\end{lemma}
\begin{proof}
First we decompose the desired expression using $\nabla_{\Gamma_0(t)} H_{\Gamma_0(t)} = \nabla_{\Gamma^*} H_{\Gamma^*}$ and $n_{\d \Gamma_0(t)} = n_{\d \Gamma^*}$ as follows
\begin{align*}
	\left.\frac{d}{d\epsilon} \nabla_{\Gamma_{\epsilon \rho}(t)} H_{\Gamma_{\epsilon \rho}(t)} \cdot n_{\d \Gamma_{\epsilon \rho}(t)}\right|_{\epsilon=0} = \left.\frac{d}{d\epsilon} \nabla_{\Gamma_{\epsilon \rho}(t)} H_{\Gamma_{\epsilon \rho}(t)}\right|_{\epsilon=0} \cdot n_{\d \Gamma^*} + \nabla_{\Gamma^*} H_{\Gamma^*} \cdot \left.\frac{d}{d\epsilon} n_{\d \Gamma_{\epsilon \rho}(t)}\right|_{\epsilon=0}.
\end{align*}
While the linearization of the conormal contains $\rho$ and its first derivatives, a closer look at the first term shows
\begin{align*}
	\left.\frac{d}{d\epsilon} \nabla_{\Gamma_{\epsilon \rho}(t)} H_{\Gamma_{\epsilon \rho}(t)}\right|_{\epsilon=0}
		= \nabla_{\Gamma^*} \left.\frac{d}{d\epsilon} \tilde{H}_{\epsilon \rho}\right|_{\epsilon=0} + LOT,
\end{align*}
where LOT stands for lower order terms that include at most second order derivatives of $\rho$. Using Lemma \ref{lem:LinearHGamma} again we see
\begin{align*}
	\nabla_{\Gamma^*} \left.\frac{d}{d\epsilon} \tilde{H}_{\epsilon \rho}\right|_{\epsilon=0}
		= \nabla_{\Gamma^*}\left(\Delta_{\Gamma^*} \rho + |\sigma^*|^2 \rho + \left(\nabla_{\Gamma^*} H_{\Gamma^*} \cdot P\left(\d_w \Psi(q,0)\right)\right) \rho\right).
\end{align*}
This proves the desired statement.
\end{proof}

Combining Lemmas \ref{lem:LinearLaplaceH} - \ref{lem:LinearGradHConormal} and the results of Section \ref{ssec:LinearMCF} we obtain the linearization of the flow (\ref{eq:WillmoreFlow1})-(\ref{eq:WillmoreFlow3}) as
\begin{align}\label{eq:LinearWillmoreFlow1}
	\d_t \rho(t) &= -\Delta_{\Gamma^*} \Delta_{\Gamma^*} \rho(t) + F_1(\rho(t),\nabla_{\Gamma^*} \rho(t),\nabla_{\Gamma^*}^2 \rho(t)) + f(t) & &\text{in } [0,T] \times \Gamma^*, \\ \label{eq:LinearWillmoreFlow2}
	\d_t \rho(t) &= \frac{1}{2} \sin(\alpha) (\nabla_{\Gamma^*} \Delta_{\Gamma^*} \rho(t) \cdot n_{\d \Gamma^*}) + b \sin(\alpha) \d_\sigma^2 \rho(t) & & \notag \\
					 &+ F_2(\rho(t),\nabla_{\Gamma^*} \rho(t),\nabla_{\Gamma^*}^2 \rho(t,q)) + g_0(t) & &\text{on } [0,T] \times \d \Gamma^*, \\ \label{eq:LinearWillmoreFlow3}
	0 				 &= \Delta_{\Gamma^*} \rho(t) + |\sigma^*|^2 \rho(t) + \left(\nabla_{\Gamma^*} H_{\Gamma^*} \cdot P\left(\d_w \Psi(0)\right)\right) \rho(t) + g_1(t) & &\text{on } [0,T] \times \d \Gamma^*, \\ \label{eq:LinearWillmoreFlow4}
	\rho(0)		 & = \rho_0 & &\text{in } \Gamma^*,
\end{align}
where $F_1$ and $F_2$ are smooth functions and we have suppressed the argument $q$ again.

\section{Local existence of solutions of the Willmore flow with line tension}\label{sec:LocalExistenceWillmore}

\subsection{Short-time existence of solutions for the linearized Willmore flow}\label{ssec:ShortTimeExLinearWF}

Now we will do the same considerations as in Section \ref{sec:LocalExistenceMCF} to show that the flow (\ref{eq:WillmoreFlow1})-(\ref{eq:WillmoreFlow3}) has a unique strong solution for short times. Again we consider first the linearized Willmore flow from (\ref{eq:LinearWillmoreFlow1})-(\ref{eq:LinearWillmoreFlow4}), which we will solve by using \cite{DPZ08} again.

First we adopt again all the notation from \cite{DPZ08}. Skipping the argument $q$ the operators and functions in our case read as
\begin{align*}
	\mathcal{A}(D) &:= \Delta_{\Gamma^*} \Delta_{\Gamma^*} + LOT, & & \\
	\mathcal{B}_0(D) &:= -\frac{1}{2} \sin(\alpha) (n_{\d \Gamma^*} \cdot \nabla_{\Gamma^*} \Delta_{\Gamma^*}) + LOT, & \mathcal{C}_0(D_\d) &:= -b \sin(\alpha) \d_\sigma^2 + LOT, \\
	\mathcal{B}_1(D) &:= \Delta_{\Gamma^*} + |\sigma^*|^2 + \left(\nabla_{\Gamma^*} H_{\Gamma^*} \cdot P\left(\d_w \Psi(0)\right)\right), & \mathcal{C}_1(D_\d) &:= 0, \\
	\mathcal{B}_2(D) &:= 1, & \mathcal{C}_2(D_\d) &:= -1, \\
	u(t) &:= \rho(t), & \varrho(t) &:= \rho(t)|_{\d \Gamma^*},
\end{align*}
where $LOT$ stands for some unspecified lower order terms, which will be unimportant for the maximal regularity result below.

We note again that the required condition ``all $\mathcal{B}_j$ and at least one $\mathcal{C}_j$ are non-trivial'' is satisfied. We still have $E := F := \R$ of type $\mathcal{HT}$ and the interval we want to consider is $[0,T]$ denoted by $J$. Because of
\begin{align*}
	l := \max_{i \in \{0,1,2\}} \{\ord(\mathcal{C}_i) - \ord(\mathcal{B}_i) + \ord(\mathcal{B}_0)\} = 3 \qquad \text{ and } \qquad 2m := \ord(\mathcal{A}) = 4
\end{align*}
we have to consider the setting that is called ``case 2'' in \cite{DPZ08}. In our situation the spaces simplify to
\begin{align}\label{eq:SpacesWillmore}
	X					 & := L_p(J;L_p(\Gamma^*;\R)), \notag \\
	Z_u				 & := W^1_p(J; L_p(\Gamma^*;\R)) \cap L_p(J; W^4_p(\Gamma^*;\R)), \notag \\
	\pi Z_u			 & := W^{4-\frac{4}{p}}_p(\Gamma^*;\R), \notag \\
	Y_0				 & := W^{\frac{1}{4}-\frac{1}{4p}}_p(J; L_p(\d \Gamma^*;\R))
								\cap L_p(J; W^{1-\frac{1}{p}}_p(\d \Gamma^*;\R)), \notag \\
	Y_1				 & := W^{\frac{1}{2}-\frac{1}{4p}}_p(J; L_p(\d \Gamma^*;\R))
								\cap L_p(J; W^{2-\frac{1}{p}}_p(\d \Gamma^*;\R)), \notag \\
	Y_2				 & := W^{1-\frac{1}{4p}}_p(J; L_p(\d \Gamma^*;\R))
								\cap L_p(J; W^{4-\frac{1}{p}}_p(\d \Gamma^*;\R)), \notag \\
	Z_\varrho		 & := W^{\frac{5}{4}-\frac{1}{4p}}_p(J; L_p(\d \Gamma^*;\R))
								\cap W^1_p(J; W^{1-\frac{1}{p}}_p(\d \Gamma^*;\R))
								\cap L_p(J; W^{4-\frac{1}{p}}_p(\d \Gamma^*;\R)), \notag \\
	\pi Z_\varrho	 & := W^{4-\frac{4}{p}}_p(\d \Gamma^*;\R), \notag \\
	\pi_1 Z_\varrho & := W^{1-\frac{5}{p}}_p(\d \Gamma^*;\R),
\end{align}
where we have to assure $\frac{4}{p} \notin \N$ and $\kappa_0 := 1 - \frac{\ord(\mathcal{B}_0)}{2m} + \frac{1}{2mp} = \frac{1}{2} - \frac{1}{2p} > \frac{1}{p}$ for the trace spaces. Therefore we assume $p > 5$. As the principle parts of the operators we obtain
\begin{align*}
	\mathcal{A}^\sharp(-i\nabla_{\Gamma^*}) & = \Delta_{\Gamma^*} \Delta_{\Gamma^*} = \left((-i\nabla_{\Gamma^*}) \cdot (-i\nabla_{\Gamma^*})\right)^2, & & \\
	\mathcal{B}_0^\sharp(-i\nabla_{\Gamma^*}) & = \frac{1}{2} i \sin(\alpha(q)) (n_{\d \Gamma^*}(q) \cdot (-i\nabla_{\Gamma^*}) \left((-i\nabla_{\Gamma^*}) \cdot (-i\nabla_{\Gamma^*})\right)), & & \\
	\mathcal{C}_0^\sharp(-i\d_\sigma) & = b \sin(\alpha(q)) (-i\d_\sigma)^2, & & \\
	\mathcal{B}_1^\sharp(-i\nabla_{\Gamma^*}) & = -\left((-i\nabla_{\Gamma^*}) \cdot (-i\nabla_{\Gamma^*})\right), & \mathcal{B}_2^\sharp(-i\nabla_{\Gamma^*}) & = 1, \\
	\mathcal{C}_1^\sharp(-i\d_\sigma) & = 0, & \mathcal{C}_2^\sharp(-i\d_\sigma) & = -1.
\end{align*}
Again we have to check some assumptions to apply the theorems of \cite{DPZ08}. Due to the case $l < 2m$ this time we can only ignore the assumptions (LS$^+_\infty$), (SD), (SB) and (SC), but have to check the assumptions (E), (LS) and henceforth (LS$^-_\infty$).

For assumption (E) we let $t \in J$, $q \in \Gamma^*$ and $\xi \in \R^2$ with $\norm{\xi} = 1$. Then we see
\begin{align*}
	\sigma(\mathcal{A}^\sharp(\xi)) & = \left\{\lambda \in \C \left| \ \lambda - \mathcal{A}^\sharp(\xi) = \lambda - \norm{\xi}^4 = 0\right.\right\} = \{1\} \subseteq \C_+.
\end{align*}

For checking the condition (LS) we prove that the ODE given by
\begin{align*}
	&\text{(I)} & \lambda v(y) + \hat{\xi}^4 v(y) - 2 \hat{\xi}^2 v''(y) + v''''(y) &= 0,\quad y>0, \\
	&\text{(II)} & -\frac{1}{2} \sin(\alpha(q)) \left(\hat{\xi}^2 v'(0) - v'''(0)\right) + \lambda \sigma + b \sin(\alpha(q)) \hat{\xi}^2 \sigma &= 0, \\
	&\text{(III)} & -\hat{\xi}^2 v(0) + v''(0) &= 0, \\
	&\text{(IV)} & v(0) - \sigma &= 0,
\end{align*}
has only the trivial solution in $C_0(\R_+;\R) \times \R$ for $\hat{\xi} \in \R$ and $\lambda \in \bar{\C_+}$ with $|\hat{\xi}| + |\lambda| \neq 0$. Now we will look at these equations step by step.

Equation (I) shows
\begin{align*}
	v(y) = \begin{cases}
				 c_1 e^{\mu_1 y} + c_2 e^{-\mu_1 y} + c_3 e^{\mu_2 y} + c_4 e^{-\mu_2 y} & \text{if} \quad \lambda \neq 0, \\
				 c_1 e^{|\hat{\xi}| y} + c_2 e^{-|\hat{\xi}| y} + c_3 y e^{|\hat{\xi}| y} + c_4 y e^{-|\hat{\xi}| y} & \text{if} \quad \lambda = 0,
			 \end{cases}
\end{align*}
where we can again choose w.l.o.g. the roots to satisfy $\Re(\mu_i) > 0$ for $i \in \{1,2\}$ since $\mu_i$ and $-\mu_i$ both appear in $v$. $\Re(\mu_i) = 0$ is again not possible. Furthermore we require $v \in C_0(\R_+;\R)$, which leads to $c_1 = c_3 = 0$. Now (IV) shows
\begin{align*}
	\sigma = v(0) = \begin{cases}
							 c_2 + c_4  & \text{if} \quad \lambda \neq 0, \\
							 c_2 & \text{if} \quad \lambda = 0,
						 \end{cases}
\end{align*}
and (III) transforms into
\begin{align*}
	\begin{cases}
		c_2 \mu_1^2 + (\sigma - c_2) \mu_2^2 = c_2 \hat{\xi}^2 + (\sigma - c_2) \hat{\xi}^2 & \text{if} \quad \lambda \neq 0, \\
		\sigma \hat{\xi}^2 - 2 |\hat{\xi}| c_4 = \sigma \hat{\xi}^2 & \text{if} \quad \lambda = 0.
	\end{cases}
\end{align*}
Using $\mu_1^2 = \hat{\xi}^2 + \sqrt{-\lambda}$ and $\mu_2^2 = \hat{\xi}^2 - \sqrt{-\lambda}$ we see that these conditions transform into $c_2 = \frac{\sigma}{2}$ if $\lambda \neq 0$ and $c_4 = 0$ if $\lambda = 0$. Therefore we get
\begin{align*}
	v(y) = \begin{cases}
				 \frac{\sigma}{2} \left(e^{-\mu_1 y} + e^{-\mu_2 y}\right) & \text{if} \quad \lambda \neq 0 \\
				 \sigma e^{-|\hat{\xi}| y} & \text{if} \quad \lambda = 0
			 \end{cases}.
\end{align*}
Equation (II) is the only remaining and can be written as
\begin{align*}
	\begin{cases}
		\frac{\sigma}{4} \sin(\alpha) \sqrt{-\lambda} (\mu_1 - \mu_2) = \lambda \sigma + b \sin(\alpha(q)) \hat{\xi}^2 \sigma & \text{if} \quad \lambda \neq 0, \\
		b \sin(\alpha) \hat{\xi}^2 \sigma = 0 & \text{if} \quad \lambda = 0,
	\end{cases},
\end{align*}
because $\hat{\xi}^2 v'(0) - v'''(0) = \frac{\sigma}{2} \sqrt{-\lambda} (\mu_1 - \mu_2)$ as one can easily calculate. In the case $\lambda = 0$ we have $b, \sin(\alpha), \hat{\xi}^2 \in \R_+$ which proves $\sigma = 0$ and hence $v \equiv 0$, which shows (LS) in this case. In the case $\lambda \neq 0$ we either have $\sigma = 0$ which would mean that (LS) is satisfied or
\begin{align}\label{eq:2ndOptionWillmore}
	\underbrace{\frac{1}{4} \sin(\alpha) \sqrt{-\lambda} (\mu_1 - \mu_2)}_{=: L} = \underbrace{\lambda + b \sin(\alpha(q)) \hat{\xi}^2}_{=: R}.
\end{align}
To prove condition (LS) completely we only have to show that (\ref{eq:2ndOptionWillmore}) is not possible. First we remark that we can assume w.l.o.g. $\Im(\sqrt{-\lambda}) > 0$ since the other choice would change the roles of $\mu_1$ and $\mu_2$, but also causes $\sqrt{-\lambda}$ to bring an additional ``$-$'' into the left hand side. Therefore both sides are independent of the choice of $\sqrt{-\lambda}$. Now we have to distinguish two cases.

\textbf{Case 1: $\Im(\lambda) \geq 0$.} Due to $\Im(\sqrt{-\lambda}) > 0$ we see that $\arg(\sqrt{-\lambda}) \in \left[\frac{\pi}{2},\frac{3\pi}{4}\right]$. Denoting $c := \Re(\sqrt{-\lambda}) \leq 0$ and $d := \Im(\sqrt{-\lambda}) > 0$ we can write the square roots $\mu_1$ and $\mu_2$ as follows
\begin{align*}
	\mu_1 & = \sqrt{\frac{\sqrt{(\hat{\xi}^2 + c)^2 + d^2} + (\hat{\xi}^2 + c)}{2}} + i \underbrace{\sgn(\Im(\hat{\xi}^2 + \sqrt{-\lambda}))}_{= 1} \sqrt{\frac{\sqrt{(\hat{\xi}^2 + c)^2 + d^2} - (\hat{\xi}^2 + c)}{2}}, \\
	\mu_2 & = \sqrt{\frac{\sqrt{(\hat{\xi}^2 - c)^2 + d^2} + (\hat{\xi}^2 - c)}{2}} + i \underbrace{\sgn(\Im(\hat{\xi}^2 - \sqrt{-\lambda}))}_{= -1} \sqrt{\frac{\sqrt{(\hat{\xi}^2 - c)^2 + d^2} - (\hat{\xi}^2 - c)}{2}}.
\end{align*}
Due to $c \leq 0$ we see that $\Re(\mu_1) \leq \Re(\mu_2)$, which gives us $\Re(\mu_1 - \mu_2) \leq 0$ and the imaginary parts obviously satisfy $\Im(\mu_1) > 0$ and  $\Im(\mu_2) < 0$, leading to $\Im(\mu_1 - \mu_2) > 0$. This proves $\arg(\mu_1 - \mu_2) \in \left[\frac{\pi}{2},\pi\right)$. Hence we get
\begin{align*}
	\arg(\sqrt{-\lambda} (\mu_1 - \mu_2)) = \arg(\sqrt{-\lambda}) + \arg(\mu_1 - \mu_2) \in \left[\frac{\pi}{2},\frac{3\pi}{4}\right] + \left[\frac{\pi}{2},\pi\right) = \left[\pi,\frac{7\pi}{4}\right).
\end{align*}
Therefore the whole left-hand side $L$ satisfies $\arg(L) \in \left[\pi,\frac{7\pi}{4}\right)$. The right-hand $R$ side obviously fulfills $\arg(R) \in \left[0, \frac{\pi}{2}\right]$ instead.

\textbf{Case 2: $\Im(\lambda) < 0$.} By the same arguments we get $\arg(\sqrt{-\lambda}) \in \left[\frac{\pi}{4},\frac{\pi}{2}\right]$, $c := \Re(\sqrt{-\lambda}) \geq 0$ and $d := \Im(\sqrt{-\lambda}) > 0$. Due to $c \geq 0$ we get this time $\Re(\mu_1 - \mu_2) \geq 0$ and $\Im(\mu_1 - \mu_2) > 0$. This proves $\arg(\mu_1 - \mu_2) \in \left(0,\frac{\pi}{2}\right]$. Hence we now obtain
\begin{align*}
	\arg(\sqrt{-\lambda} (\mu_1 - \mu_2)) = \arg(\sqrt{-\lambda}) + \arg(\mu_1 - \mu_2) \in \left[\frac{\pi}{4},\frac{\pi}{2}\right] + \left(0,\frac{\pi}{2}\right] = \left(\frac{\pi}{4},\pi\right].
\end{align*}
Therefore we see $\arg(L) \in \left(\frac{\pi}{4},\pi\right]$. The right-hand side obviously fulfills $\arg(R) \in \left[\frac{3\pi}{2},2\pi\right]$ instead.

In both cases we get $\arg(L) \neq \arg(R)$ which proves that (\ref{eq:2ndOptionWillmore}) is not true in either case. This completes the proof of (LS).

Next we prove the validity of assumption (LS$^-_\infty$). First we have to show that for $\hat{\xi} \in \R$, $\lambda \in \bar{\C_+}$ with $|\hat{\xi}| + |\lambda| \neq 0$ the ODE
\begin{align*}
	&\text{(I)} & \lambda v(y) + \hat{\xi}^4 v(y) - 2 \hat{\xi}^2 v''(y) + v''''(y) &= 0,\quad y>0, \\
	&\text{(II)} & -\hat{\xi}^2 v(0) + v''(0) &= 0, \\
	&\text{(III)} & v(0) &= 0
\end{align*}
only has the trivial solution in $C_0(\R_+;\R)$. The same arguments as before show that the function $v \in C_0(\R_+;\R)$ has the form
\begin{align*}
	v(y) = \begin{cases}
				 c_2 e^{-\mu_1 y} + c_4 e^{-\mu_2 y} & \text{if} \quad \lambda \neq 0, \\
				 c_2 e^{-|\hat{\xi}| y} + c_4 y e^{-|\hat{\xi}| y} & \text{if} \quad \lambda = 0,
			 \end{cases}
\end{align*}
where we again choose w.l.o.g. $\Re(\mu_i) > 0$. Equation (III) shows that
\begin{align*}
	0 = v(0) = \begin{cases}
					  c_2 + c_4 & \text{if} \quad \lambda \neq 0, \\
					  c_2 & \text{if} \quad \lambda = 0,
				  \end{cases}
\end{align*}
and equation (II) transforms into
\begin{align*}
	\begin{cases}
		c_2 \mu_1^2 - c_2 \mu_2^2 = c_2 \hat{\xi}^2 - c_2 \hat{\xi}^2 = 0 & \text{if} \quad \lambda \neq 0, \\
		-2 |\hat{\xi}| c_4 = 0 & \text{if} \quad \lambda = 0.
	\end{cases}
\end{align*}
Using $\mu_1^2 \neq \mu_2^2$ this gives $c_2 = 0$ if $\lambda \neq 0$ and $c_4 = 0$ if $\lambda = 0$. Hence we have shown $v \equiv 0$.

The second statement to prove in (LS$^-_\infty$) is that for $\lambda \in \bar{\C_+}$
\begin{align*}
	&\text{(I)} & v(y) - 2 v''(y) + v''''(y) &= 0,\quad y>0, \\
	&\text{(II)} & -\frac{1}{2} \sin(\alpha(q)) \left(v'(0) - v'''(0)\right) + \lambda \sigma + b \sin(\alpha(q)) \sigma &= 0, \\
	&\text{(III)} & -v(0) + v''(0) &= 0, \\
	&\text{(IV)} & v(0) - \sigma &= 0
\end{align*}
only has the trivial solution in $C_0(\R_+;\R) \times \R$. Equation (I) and $v \in C_0(\R_+;\R)$ show that
\begin{align*}
	v(y) = c_2 e^{-y} + c_4 y e^{-y}.
\end{align*}
Equation (IV) gives $\sigma = v(0) = c_2$ and equation (III) reads as $\sigma - 2 c_4 = \sigma$, hence $c_4 = 0$. We end up with $v(y) = \sigma e^{-y}$. Computing $v'(0) - v'''(0) = 0$ we see that (II) simplifies to $(\lambda + b \sin(\alpha(q))) \sigma = 0$. Since $\lambda \neq -b \sin(\alpha(q)) \in \R_-$, we get $\sigma = 0$ leading to $v \equiv 0$. Finally this completes (LS$^-_\infty$).

Now we have proven all assumptions from \cite{DPZ08} and can state the following theorem.

\begin{thm}\label{thm:LocalExistenceWillmoreGeneral}
Let $5 < p < \infty$, $J := [0,T]$ and the spaces be defined as in (\ref{eq:SpacesWillmore}). Then the problem
\begin{align}\label{eq:LocalLinearWillmore1}
	\frac{d}{dt} u(t) + \mathcal{A}(D) u(t) &= f(t) & &\text{ in } J \times \Gamma^*, \\ \label{eq:LocalLinearWillmore2}
	\frac{d}{dt} \varrho(t) + \mathcal{B}_0(D) u(t) + \mathcal{C}_0(D_\d) \varrho(t) &= g_0(t) & &\text{ on } J \times \d \Gamma^*, \\ \label{eq:LocalLinearWillmore3}
	\mathcal{B}_1(D) u(t) + \mathcal{C}_1(D_\d) \varrho(t) &= g_1(t) & &\text{ on } J \times \d \Gamma^*, \\ \label{eq:LocalLinearWillmore4}
	\mathcal{B}_2(D) u(t) + \mathcal{C}_2(D_\d) \varrho(t) &= g_2(t) & &\text{ on } J \times \d \Gamma^*, \\ \label{eq:LocalLinearWillmore5}
	u(0) &= u_0 & &\text{ in } \Gamma^*, \\ \label{eq:LocalLinearWillmore6}
	\varrho(0) &= \varrho_0 & &\text{ on } \d \Gamma^*
\end{align}
has a unique solution $(u,\varrho) \in Z_u \times Z_\varrho$ if and only if
\begin{align*}
	f \in X, \qquad u_0 \in \pi Z_u, \qquad \varrho_0 \in \pi Z_\varrho, \qquad g_0 \in Y_0, \\
	g_1 \in Y_1, \qquad g_2 \in Y_2, \qquad g_0(0) - \mathcal{B}_0(D) u_0 - \mathcal{C}_0(D_\d) \varrho_0 \in \pi_1 Z_\varrho, \\
	\mathcal{B}_1(D) u_0 + \mathcal{C}_1(D_\d) \varrho_0 = g_1(0), \qquad \mathcal{B}_2(D) u_0 + \mathcal{C}_2(D_\d) \varrho_0 = g_2(0).
\end{align*}
\end{thm}
\begin{proof}
Follows from Theorem 2.1 in \cite{DPZ08} adapted to this specific case.
\end{proof}

\begin{cor}\label{cor:LocalExistenceWillmore}
Let $5 < p < \infty$, $J := [0,T]$ and the spaces be defined as in (\ref{eq:SpacesWillmore}). Then (\ref{eq:LinearWillmoreFlow1})-(\ref{eq:LinearWillmoreFlow4}) has a unique solution $\rho \in Z_u$ with $\rho|_{\d \Gamma^*} \in Z_\varrho$ if and only if $f \in X$, $g_0 \in Y_0$, $g_1 \in Y_1$, $\rho_0 \in \pi Z_u$ and $\rho_0|_{\d \Gamma^*} \in \pi Z_\varrho$ and $\Delta_{\Gamma^*} \rho_0 + |\sigma^*|^2 \rho_0 + \left(\nabla_{\Gamma^*} H_{\Gamma^*} \cdot P\left(\d_w \Psi(0)\right)\right) \rho_0 = 0$.
\end{cor}
\begin{proof}
Follows from Theorem \ref{thm:LocalExistenceWillmoreGeneral} if we choose $g_2 \equiv 0$. The condition $\mathcal{B}_2(0) u_0 + \mathcal{C}_2(0) \varrho_0 = g_2(0)$ is valid since $u_0|_{\d \Gamma^*} = \rho_0|_{\d \Gamma^*} = \varrho_0$. Finally,
\begin{align*}
	g_0(0) - \mathcal{B}_0(-i\nabla_{\Gamma^*})\rho_0 - \mathcal{C}_0(-i \d_\sigma) \rho_0|_{\d \Gamma^*} \in \pi_1 Z_\varrho = W^{1-\frac{5}{p}}_p(\d \Gamma^*;\R)
\end{align*}
can be ignored, because $\rho_0|_{\d \Gamma^*} \in \pi Z_\varrho = W^{4-\frac{4}{p}}_p(\d \Gamma^*;\R)$, $\mathcal{C}_0$ is of second order as well as $\rho_0 \in \pi Z_u = W^{4-\frac{4}{p}}_p(\Gamma^*;\R)$, $\mathcal{B}_0$ is of third order and the trace operator maps from $W^{1-\frac{4}{p}}_p(\Gamma^*;\R)$ to $W^{1-\frac{5}{p}}_p(\d \Gamma^*;\R)$. In addition, $g_0(0) \in \pi_1 Z_\varrho$, since the trace operator $\gamma_0$ maps from $Y_0$ to $\pi_1 Z_\varrho$ as one can see from (A.25) in \cite{Gru95}. Finally, $\mathcal{B}_1(D) u_0 + \mathcal{C}_1(D_\d) \varrho_0 = g_1(0)$ simplifies to
\begin{align*}
	\Delta_{\Gamma^*} \rho_0 + |\sigma^*|^2 \rho_0 + \left(\nabla_{\Gamma^*} H_{\Gamma^*} \cdot P\left(\d_w \Psi(0)\right)\right) \rho_0 = 0.
\end{align*}
\end{proof}

\subsection{Short-time existence of solutions for the Willmore flow}\label{ssec:ShortTimeExWF}

Now we want to prove short-time existence of solutions of the non-linear flow
\begin{align}\label{eq:NonLinearWF1}
	V_\Gamma(u(t)) &= -\Delta_\Gamma H_\Gamma(u(t)) - \frac{1}{2} H_\Gamma(u(t)) \left(H_\Gamma(u(t))^2 - 4 K_\Gamma(u(t))\right) & &\text{in } J \times \Gamma^*, \\ \label{eq:NonLinearWF2}
	v_{\d D}(\varrho(t)) &= \frac{1}{2} \sin(\alpha) (\nabla_\Gamma H_\Gamma(u(t)) \cdot n_{\d \Gamma}(u(t))) + a + b \kappa_{\d D}(\varrho(t)) & &\text{on } J \times \d \Gamma^*, \\ \label{eq:NonLinearWF3}
	0 &= H_\Gamma(u(t)) & &\text{on } J \times \d \Gamma^*, \\ \label{eq:NonLinearWF4}
	0 &= u(t) - \varrho(t) & &\text{on } J \times \d \Gamma^*, \\ \label{eq:NonLinearWF5}
	u(0) &= u_0 & &\text{in } \Gamma^*, \\ \label{eq:NonLinearWF6}
	\varrho(0) &= \varrho_0 & &\text{on } \d \Gamma^*,
\end{align}									 
where we have adopted the structure of the linearized PDE in Theorem \ref{thm:LocalExistenceGeneral} again. To prove short-time existence we follow the same strategy as in Section \ref{ssec:ShortTimeExMCF}. We define again functions $\Phi := (u,\varrho)$, $\Phi_0 := (u_0,\varrho_0)$, spaces
\begin{align*}
	\mathbb{E} & := Z_u \times Z_\varrho, \qquad \mathbb{F} := X \times Y_0 \times Y_1 \times \{0\}, \\
	\mathbb{I} & := \left\{(u_0,\varrho_0) \in \pi Z_u \times \pi Z_\varrho \mid u_0|_{\d \Gamma^*} = \varrho_0\right\}
\end{align*}
and the operator $L: \mathbb{E} \longrightarrow \mathbb{F} \times \mathbb{I}$ as the left-hand side of (\ref{eq:LocalLinearWillmore1})-(\ref{eq:LocalLinearWillmore6}). For the right-hand side of the contraction mapping principle we define the non-linear operator $N: \mathbb{E} \to \mathbb{F}$ as
\begin{align*}
	N(\Phi) := \left(F(u), G(u,\varrho), H(u,\varrho), 0\right)^T,
\end{align*}
where
\begin{align*}
	F(u)			 & := -\Delta_\Gamma H_\Gamma(u) - \frac{1}{2} H_\Gamma(u) \left(H_\Gamma(u)^2 - 4 K_\Gamma(u)\right) - V_\Gamma(u) + \frac{d}{dt} u + \mathcal{A}(D) u \\
	G(u,\varrho) & := \frac{1}{2} \sin(\alpha) (\nabla_\Gamma H_\Gamma(u) \cdot n_{\d \Gamma}(u)) + a + b \kappa_{\d D}(\varrho) - v_{\d D}(\varrho) \\
					 & + \frac{d}{dt} \varrho + \mathcal{B}_0(D) u + \mathcal{C}_0(D_\d) \varrho \\
	H(u,\varrho) & := H_\Gamma(u) + \mathcal{B}_1(D) u + \mathcal{C}_1(D_\d) \varrho.
\end{align*}
In order to solve the equation $L \Phi = (N \Phi, \Phi_0)$ by the contraction mapping principle we have to prove some technical lemmas again.

\begin{lemma}\label{lem:EmbeddingsWillmore}
(i) Let $6 < p < \infty$. Then one gets for any $\sigma \in (0,1)$
\begin{align}\label{eq:CrucialEmbeddingWillmore}
	Z_u & \hookrightarrow W^\sigma_p(J;W^{4(1 - \sigma)}_p(\Gamma^*;\R)), \notag \\
	Z_\varrho & \hookrightarrow W^{\sigma(\frac{5}{4} - \frac{1}{4p})}_p(J;W^{(1 - \sigma)(4 - \frac{1}{p})}_p(\d \Gamma^*;\R)), \notag \\
	Y_0 & \hookrightarrow W^{\sigma(\frac{1}{4} - \frac{1}{4p})}_p(J;W^{(1 - \sigma)(1 - \frac{1}{p})}_p(\d \Gamma^*;\R)), \notag \\
	Z_u & \hookrightarrow BUC(J;W^{4 - \frac{4}{p}}_p(\Gamma^*;\R)) \hookrightarrow BUC(J;BUC^3(\Gamma^*;\R)), \\
	Z_\varrho & \hookrightarrow BUC(J;W^{4 - \frac{4}{p}}_p(\d \Gamma^*;\R)) \hookrightarrow BUC(J;BUC^3(\d \Gamma^*;\R)). \notag
\end{align}
\end{lemma}
\begin{proof}
Follows from Lemmas 4.3 and 4.4. of \cite{DSS08} and usual Sobolev embeddings with our assumption $p > 6$.
\end{proof}

\begin{rem}
The embedding (\ref{eq:CrucialEmbeddingWillmore}) is only valid for $p > 6$ and will be crucial in the considerations to follow. That is why we are forced to change the range of $p$ from $p > 5$ in Theorem \ref{thm:LocalExistenceWillmoreGeneral} to $p > 6$ in our final Theorem \ref{thm:ShortTimeExistenceWillmore}.
\end{rem}

Following the strategy of Section \ref{ssec:ShortTimeExMCF} we define
\begin{align*}
	\nabla^3 Z_u & := W^{\frac{1}{4}}_p(J;L_p(\Gamma^*;\R)) \cap L_p(J;W^1_p(\Gamma^*;\R)), \\
	\nabla^2 Z_u & := W^{\frac{1}{2}}_p(J;L_p(\Gamma^*;\R)) \cap L_p(J;W^2_p(\Gamma^*;\R)), \\
	\nabla^2 Z_\varrho & := W^{\frac{2}{3} - \frac{1}{3p}}_p(J;L_p(\d \Gamma^*;\R)) \cap L_p(J;W^{2 - \frac{1}{p}}_p(\d \Gamma^*;\R))
\end{align*}
as the spaces containing the third and second space derivatives of $u$ and the second arc-length derivatives of $\varrho$, respectively.

\begin{lemma}\label{lem:AlgebraWillmore}
Let $6 < p < \infty$. Then the spaces
\begin{align*}
	\text{(i)} &\quad & \nabla^3 Z_u &= W^{\frac{1}{4}}_p(J;L_p(\Gamma^*;\R)) \cap L_p(J;W^1_p(\Gamma^*;\R)), \\
	\text{(ii)} &\quad & \nabla^2 Z_u &= W^{\frac{1}{2}}_p(J;L_p(\Gamma^*;\R)) \cap L_p(J;W^2_p(\Gamma^*;\R)), \\
	\text{(iii)} &\quad & \nabla^2 Z_\varrho &= W^{\frac{2}{3} - \frac{1}{3p}}_p(J;L_p(\d \Gamma^*;\R)) \cap L_p(J;W^{2 - \frac{1}{p}}_p(\d \Gamma^*;\R))
\end{align*}
are Banach algebras up to a constant in the norm estimate for the product.
\end{lemma}
\begin{proof}
The crucial ingredient for the proof is the embedding into the space of bounded uniformly continuous functions. Therefore we will only prove these embeddings and the rest will follow in the same manner as in Lemma \ref{lem:AlgebraMCF}. \\
(i) Lemma 4.3 of \cite{DSS08} gives the embedding
\begin{align*}
	\nabla^3 Z_u \hookrightarrow BUC(J;W^{1 - \frac{4}{p}}_p(\Gamma^*;\R)) \hookrightarrow BUC(J;BUC(\Gamma^*;\R)),
\end{align*}
where we used $p > 6$ in the second step. \\
(ii) Here Lemma 4.4 of \cite{DSS08} gives the embedding
\begin{align*}
	\nabla^2 Z_u \hookrightarrow BUC(J;W^{2 - \frac{4}{p}}_p(\Gamma^*;\R)) \hookrightarrow BUC(J;BUC^1(\Gamma^*;\R)),
\end{align*}
where again we have used $p > 6$ in the second step. \\
(iii) The same lemma as in (i) shows
\begin{align*}
	\nabla^2 Z_\varrho \hookrightarrow BUC(J;W^{4 - \frac{4}{p}}_p(\d \Gamma^*;\R)) \hookrightarrow BUC(J;BUC^3(\d \Gamma^*;\R)),
\end{align*}
where we used $p > 5$ in the second embedding.
\end{proof}

\begin{lemma}\label{lem:TechnicalPreliminariesWillmore}
Let $J := [0,T]$ and $6 < p < \infty$ and $\mathbb{B}^\mathbb{E}_r(\mathbb{O}) := \left\{\Phi \in \mathbb{E} \left| \norm{\Phi}_{\mathbb{E}} < r\right.\right\}$. Then there exists an $r > 0$ such that $N(\mathbb{B}^\mathbb{E}_r(\mathbb{O})) \subseteq \mathbb{F}$. Moreover, $N \in C^1(\mathbb{B}^\mathbb{E}_r(\mathbb{O});\mathbb{F})$ and $DN[\mathbb{O}] = \mathbb{O}$, where $DN$ denotes the Fr\'echet derivative of $N$.
\end{lemma}
\begin{proof}
The fact that $DN[\mathbb{O}] = \mathbb{O}$ is obvious due to the structure of the linearization in Section \ref{sec:LinearWillmore}. \\
Our first goal is to show that $F(u) \in X$, $G(u,\varrho) \in Y_0$ and $H(u) \in Y_1$ for all $(u,\varrho) \in \mathbb{B}^\mathbb{E}_r(\mathbb{O})$. For $r > 0$ small enough all the terms appearing in $F$, $G$ and $H$ are well-defined and the linear parts of $F$, $G$ and $H$ can be omitted by the same arguments as in Lemma \ref{lem:TechnicalPreliminariesMCF}. Also the velocities in $F$ and $G$ are bounded in the same manner as in the proof of Lemma \ref{lem:TechnicalPreliminariesMCF}. \\
By Lemma \ref{lem:EmbeddingsWillmore} we see that $|u(t,q)|$, $|\nabla_{\Gamma^*} u(t,q)|$, $|\nabla_{\Gamma^*}^2 u(t,q)|$ and $|\nabla_{\Gamma^*}^3 u(t,q)|$ stay bounded. This shows that for a maybe even smaller $r$ the first fundamental form of all the hypersurfaces in the family $\left(\Gamma_{\rho}(t)\right)_{t \in J}$ is not degenerated. Because of the fact that $\Delta_\Gamma H_\Gamma(u)$ depends linearly on the fourth space derivatives of $u$ and that the coefficients, involving only $u$ and its first to third spacial derivatives, are bounded, we get
\begin{align*}
	\norm{\Delta_\Gamma H_\Gamma(u)}_X \leq c \left(\norm{\nabla_{\Gamma^*}^4 u}_X + 1\right) \leq c \left(\norm{u}_{L_p(J;W^4_p(\Gamma^*;\R))} + 1\right) \leq c \left(\norm{u}_{Z_u} + 1\right) < \infty.
\end{align*}
Because of the fact that $|u(t,q)|$, $|\nabla_{\Gamma^*} u(t,q)|$, $|\nabla_{\Gamma^*}^2 u(t,q)|$ and $|\nabla_{\Gamma^*}^3 u(t,q)|$ stay bounded and $H_\Gamma(u)$ and $K_\Gamma(u)$ depend continuously on $u$ and its first and second spacial derivatives, we obtain that $\norm{H_\Gamma(u)}_X$, $\norm{H_\Gamma(u)^2}_X$, $\norm{K_\Gamma(u)}_X$, $\norm{H_\Gamma(u)}_{Y_1}$ and $\norm{\nabla_\Gamma H_\Gamma(u)}_{Y_0}$ are bounded in a similar way as in Lemma \ref{lem:TechnicalPreliminariesMCF}. \\
For $\kappa_{\d D}(\varrho)$ we observe by Lemma \ref{lem:EmbeddingsWillmore} that $Z_\varrho \hookrightarrow BUC(J;BUC^3(\d \Gamma^*;\R))$ and hence $|\varrho(t,q)|$, $|\d_{\sigma} \varrho(t,q)|$ and $|\d_{\sigma}^2 \varrho(t,q)|$ stay bounded. In the same way as for $H_\Gamma(u)$, the continuous dependence of $\kappa_{\d D}(\varrho)$ on $\varrho$ and its derivatives shows the boundedness and $\norm{\kappa_{\d D}(\varrho)}_{Y_0} < \infty$. This shows $N(\mathbb{B}^\mathbb{E}_r(\mathbb{O})) \subseteq \mathbb{F}$. \\
What is left is the fact that $N \in C^1(\mathbb{B}^\mathbb{E}_r(\mathbb{O});\mathbb{F})$, for which it is enough to show Lipschitz continuity of the several parts of $N$. Again we just consider the highest order term $\Delta_\Gamma H_\Gamma$, since the Lipschitz continuity of the remaining terms follows analogously. We know that we can write
\begin{align*}
	\Delta_\Gamma H_\Gamma(u) = \sum_{|\alpha|=4} a_\alpha(\underbrace{u, \nabla_{\Gamma^*} u, \nabla_{\Gamma^*}^2 u, \nabla_{\Gamma^*}^3 u}_{=: \mathbb{U}}) \d^\alpha u + b(\mathbb{U})
\end{align*}
with $a_\alpha, b \in C^3(U)$ and $U \subseteq \R \times \R^2 \times \R^{2 \times 2} \times \R^{2 \times 2 \times 2}$ a closed neighborhood of $0$. Linearizing this we obtain
\begin{align*}
	(D_u \Delta_\Gamma H_\Gamma(u))(v) & = \sum_{|\alpha|=4} \Big(\d^\alpha u \d_1 a_\alpha(\mathbb{U}) v + \d^\alpha u \d_2 a_\alpha(\mathbb{U})(\nabla_{\Gamma^*} v) + \d^\alpha u \d_3 a_\alpha(\mathbb{U})(\nabla_{\Gamma^*}^2 v)\\
& + \d^\alpha u \d_4 a_\alpha(\mathbb{U})(\nabla_{\Gamma^*}^3 v) + a_\alpha(\mathbb{U}) \d^\alpha v\Big) + \d_1 b(\mathbb{U}) v + \d_2 b(\mathbb{U})(\nabla_{\Gamma^*} v) \\
		& + \d_3 b(\mathbb{U})(\nabla_{\Gamma^*}^2 v) + \d_4 b(\mathbb{U})(\nabla_{\Gamma^*}^3 v)
\end{align*}
Again by the smoothness of $a$ and $b$ the coefficients $\d_i a_\alpha, \d_i b$ with $i \in \{1, 2, 3, 4\}$ and $a_\alpha$ satisfy a Lipschitz condition on $\bar{B_r(0)} \subseteq \nabla^3 Z_u$. Because of $\norm{\bullet}_{\nabla^3 Z_u} \leq c \norm{\bullet}_{Z_u}$, any two $u, \tilde{u} \in \bar{B_r(0)} \subseteq Z_u$ are also in $\bar{B_{cr}(0)} \subseteq \nabla^3 Z_u$. Similar to $H_\Gamma$ in Lemma \ref{lem:TechnicalPreliminariesMCF} we use this time $\norm{\nabla_{\Gamma^*}^i}_{\mathcal{L}(Z_u,X)} < \infty$ for $i \in \{0, 1, 2, 3, 4\}$ and the embedding $\nabla^3 Z_u \hookrightarrow X$ to prove for $u, \tilde{u} \in \bar{B_r(0)} \subseteq Z_u$ the estimate
\begin{align*}
	\left\|D_u \Delta_\Gamma H_\Gamma(\mathbb{U}) - D_u \Delta_\Gamma H_\Gamma(\tilde{\mathbb{U}})\right\|_{\mathcal{L}(Z_u,X)} \leq c(r) \norm{u - \tilde{u}}_{Z_u}.
\end{align*}
This shows the Lipschitz continuity of $D_u \Delta_\Gamma H_\Gamma: \bar{B_r(0)} \subseteq Z_u \longrightarrow \mathcal{L}(Z_u,X)$ and hence we see $\Delta_\Gamma H_\Gamma \in C^1(\bar{B_r(0)},X)$. For more details on this proof see \cite{Mue13}. 
\end{proof}

After proving this technical lemma we can again apply the contraction mapping principle in the following lemma.

\begin{rem}
Again it is important that $L$ is an isomorphism. We do not need to consider $g_0(0) - \mathcal{B}_0(0) u_0 - \mathcal{C}_0(0) \varrho_0 \in \pi_1 Z_\varrho$ from Theorem \ref{thm:LocalExistenceWillmoreGeneral} by the same argumentation as in proof of Corollary \ref{cor:LocalExistenceWillmore}. Moreover, the condition $\mathcal{B}_2(0) u_0 + \mathcal{C}_2(0) \varrho_0 = g_2(0)$ can be dropped because $g_2 \equiv 0$ and $(u_0,\varrho_0) \in \mathbb{I}$. Finally, $\mathcal{B}_1(0) u_0 + \mathcal{C}_1(0) \varrho_0 = g_1(0)$ reduces to $\mathcal{B}_1(0) u_0 = g_1(0)$. Due to Theorem \ref{thm:LocalExistenceWillmoreGeneral} the operator $L$ is an isomorphism between the spaces $\mathbb{E}$ and
\begin{align*}
	\mathbb{F}_0 \times \mathbb{I} := \left\{(f, g_0, g_1, 0, u_0, \varrho_0) \in \mathbb{F} \times \mathbb{I} \ \left| \ \mathcal{B}_1(0) u_0 = g_1(0) \right.\right\} \times \mathbb{I}.
\end{align*}
\end{rem}

\begin{lemma}\label{lem:UniformlyWillmore}
Let $T_0 > 0$ be fixed and $T \in (0,T_0)$ arbitrary. \\
(i) There exists a bounded extension operator from $\mathbb{F}_T$ to $\mathbb{F}_{T_0}$, i.e., for all $f \in \mathbb{F}_T$ there is a $\tilde{f} \in \mathbb{F}_{T_0}$ with $\left.\tilde{f}\right|_{[0,T]} = f$ and $\norm{\tilde{f}}_{\mathbb{F}_{T_0}} \leq c(T_0) \norm{f}_{\mathbb{F}_T}$. \\
(ii) The operator norm of $L_T^{-1}: \mathbb{F}_T \times \mathbb{I} \longrightarrow \mathbb{E}_T$ is uniformly bounded in $T$. \\
(iii) There exists a bounded extension operator from $\mathbb{E}_T$ to $\mathbb{E}_{T_0}$, i.e., for all $\Phi \in \mathbb{E}_T$ there is a $\tilde{\Phi} \in \mathbb{E}_{T_0}$ with $\left.\tilde{\Phi}\right|_{[0,T]} = \Phi$ and $\norm{\tilde{\Phi}}_{\mathbb{E}_{T_0}} \leq c(T_0) \norm{\Phi}_{\mathbb{E}_T}$. \\
(iv) The uniform estimate $\norm{DN_T[\Phi]}_{\mathcal{L}(\mathbb{E}_T,\mathbb{F}_T)} \leq c(T_0) \norm{\Phi}_{\mathbb{E}_T} < \infty$ holds for $\Phi \in B^{\mathbb{E}_T}_r(0)$.
\end{lemma}
\begin{proof}
Can be shown in the same way as in Lemma \ref{lem:UniformlyMCF} only with one additional component while proving (i) and replacing $-\d_\sigma^2$ by $\d_\sigma^4$ in (\ref{eq:ExtensionEq}) (cf. \cite{Mue13} for details).
\end{proof}

\begin{lemma}\label{lem:FixpointNonlinearWillmore}
Let $6 < p < \infty$ and $J := [0,T]$ where $T > 0$ must be chosen sufficiently small. Then there exists some $\epsilon > 0$ such that for each $\Phi_0 = (u_0, \varrho_0) \in \mathbb{I}$ with $\norm{\Phi_0}_\mathbb{I} < \epsilon$ and $H_\Gamma(u_0) = 0$ there exists a unique solution $\Phi = (u, \varrho) \in \mathbb{E}$ of $L \Phi = (N(\Phi),\Phi_0)$.
\end{lemma}
\begin{proof}
We set $\mathbb{X}_r := \left\{\Phi \in \mathbb{B}^\mathbb{E}_r(\mathbb{O}) \left| \ \Phi(0) = \Phi_0\right.\right\}$. The equation $L \Phi = (N(\Phi),\Phi_0)$ is equivalent to the fixed point problem $K(\Phi) = \Phi$, where
\begin{align*}
	K(\Phi) := L^{-1} (N(\Phi),\Phi_0) \qquad \forall \, \Phi \in \mathbb{B}^\mathbb{E}_r(\mathbb{O}).
\end{align*}
For the invertibility of $L$ we have to make sure that $(N(\Phi),\Phi_0) \in \mathbb{F}_0$, which means
\begin{align*}
	\mathcal{B}_1(D) u_0 = N^{(3)}(\Phi_0) = H_\Gamma(u_0) + \mathcal{B}_1(D) u_0.
\end{align*}
This is equivalent to $H_\Gamma(u_0) = 0$ on $\d \Gamma^*$. \\
Due to Lemma \ref{lem:UniformlyWillmore}(iv) and Lemma \ref{lem:TechnicalPreliminariesWillmore} we can choose $r > 0$ independent of $T$ such that
\begin{align*}
	\sup_{\Psi \in \mathbb{B}^\mathbb{E}_r(\mathbb{O})} \norm{DN[\Psi]}_{\mathcal{L}(\mathbb{E};\mathbb{F})} \leq \frac{1}{2 \sup_{T \in [0,T_0]}\limits \norm{L^{-1}}}.
\end{align*}
Taking a look at $\norm{N(\mathbb{O})}_\mathbb{F}$ we see
\begin{align*}
	\norm{N(\mathbb{O})}_\mathbb{F} & = T^{\frac{1}{p}} \norm{\Delta_{\Gamma^*} H_{\Gamma^*} + \frac{1}{2} H_{\Gamma^*} \left(H_{\Gamma^*}^2 - 4 K_{\Gamma^*}\right)}_{L_p(\Gamma^*;\R)} \\
		& + T^{\frac{1}{p}} \norm{\frac{1}{2} \sin(\alpha) (\nabla_{\Gamma^*} H_{\Gamma^*} \cdot n_{\d \Gamma^*}) + a + b \kappa_{\d D^*}}_{W^{1 - \frac{1}{p}}_p(\d \Gamma^*; \R)} + T^{\frac{1}{p}} \norm{H_{\Gamma^*}}_{W^{2 - \frac{1}{p}}_p(\d \Gamma^*; \R)},
\end{align*}
because $\Delta_{\Gamma^*} H_{\Gamma^*}, H_{\Gamma^*}, K_{\Gamma^*}, a, \kappa_{\d D^*}$ and $n_{\d \Gamma^*}$ are time-independent. Hence $\norm{N(\mathbb{O})}_\mathbb{F} \xrightarrow[T \rightarrow 0]{} 0$ and for a sufficiently small time interval $[0,T]$ we get $\norm{N(\mathbb{O})}_\mathbb{F} \leq \epsilon$. We use this fact in the estimate
\begin{align*}
	\norm{K(\Phi)}_\mathbb{E} & \leq \norm{L^{-1}} \left(\norm{N(\Phi)}_\mathbb{F} + \norm{\Phi_0}_\mathbb{I}\right)
		  \leq \norm{L^{-1}} \left(\norm{N(\Phi) - N(\mathbb{O})}_\mathbb{F} + \norm{N(\mathbb{O})}_\mathbb{F} + \norm{\Phi_0}_\mathbb{I}\right) \\
		& \leq \norm{L^{-1}} \left(\sup_{\Psi \in \mathbb{B}^\mathbb{E}_r(\mathbb{O})} \norm{DN[\Psi]}_{\mathcal{L}(\mathbb{E};\mathbb{F})} \norm{\Phi}_\mathbb{E} + \norm{N(\mathbb{O})}_\mathbb{F} + \norm{\Phi_0}_\mathbb{I}\right) \\
		& \leq \frac{1}{2} \norm{\Phi}_\mathbb{E} + \norm{L^{-1}} \norm{N(\mathbb{O})}_\mathbb{F} + \norm{L^{-1}} \norm{\Phi_0}_\mathbb{I} \leq \frac{r}{2} + 2 \norm{L^{-1}} \epsilon
\end{align*}
for every $\Phi \in \mathbb{X}_r$. By choosing
\begin{align*}
	\epsilon(r) \leq \frac{r}{4 \sup_{T \in [0,T_0]}\limits \norm{L^{-1}}}
\end{align*}
we get $\norm{K(\Phi)} \leq \frac{r}{2} + \frac{r}{2} = r$, i.e., $K(\mathbb{X}_r) \subseteq \mathbb{X}_r$. To see that $K$ is contractive, we observe that for all $\Phi_1, \Phi_2 \in \mathbb{X}_r$ the following holds
\begin{align*}
	\norm{K(\Phi_1) - K(\Phi_2)}_\mathbb{E} & \leq \norm{L^{-1}} \norm{N(\Phi_1) - N(\Phi_2)}_\mathbb{F} \\
		& \leq \norm{L^{-1}} \sup_{\Psi \in \mathbb{B}^\mathbb{E}_r(\mathbb{O})} \norm{DN[\Psi]}_{\mathcal{L}(\mathbb{E};\mathbb{F})} \norm{\Phi_1 - \Phi_2}_\mathbb{E} \leq \frac{1}{2} \norm{\Phi_1 - \Phi_2}_\mathbb{E}.
\end{align*}
Hence $K: \mathbb{X}_r \to \mathbb{X}_r$ is a contraction and the assertion follows from the contraction mapping principle.
\end{proof}

Transforming this statement to our original situation in terms of $\rho$ instead of $\Phi$ we end up with

\begin{thm}\label{thm:ShortTimeExistenceWillmore}
Let $T > 0$ be sufficiently small and $6 < p < \infty$. Then there exists an $\epsilon > 0$ such that for each $\rho_0 \in \pi Z_u$ with $\rho_0|_{\d \Gamma^*} \in \pi Z_\varrho$, $\norm{\rho_0}_{\pi Z_u} + \norm{\rho_0|_{\d \Gamma^*}}_{\pi Z_\varrho} < \epsilon$ and $H_\Gamma(\rho_0) = 0$ on $\d \Gamma^*$ there exists a unique solution $\rho \in Z_u$ with $\rho|_{\d \Gamma^*} \in Z_\varrho$ of the system
\begin{align*}
	V_\Gamma(\rho(t)) &= -\Delta_\Gamma H_\Gamma(\rho(t)) - \frac{1}{2} H_\Gamma(\rho(t)) \left(H_\Gamma(\rho(t))^2 - 4 K_\Gamma(\rho(t))\right) & &\text{in } [0,T] \times \Gamma^*, \\
	v_{\d D}(\rho(t)) &= \frac{1}{2} \sin(\alpha) (\nabla_\Gamma H_\Gamma(\rho(t)) \cdot n_{\d \Gamma}(\rho(t))) + a + b \kappa_{\d D}(\rho(t)) & &\text{on } [0,T] \times \d \Gamma^*, \\
	0 &= H_\Gamma(\rho(t)) & &\text{on } [0,T] \times \d \Gamma^*, \\
	\rho(0) &= \rho_0 & &\text{in } \Gamma^*.
\end{align*}
\end{thm}

Note that Remark \ref{rem:Smallness} is true in the situation of the Willmore flow as well.

%\nocite{*}															%Auch die nicht-zitierte Literatur wird angezeigt
\printbibliography												%Erstellt das Literaturverzeichnis mittels BibLaTeX

\end{document}